\documentclass[11pt,reqno]{amsart}

\usepackage{amsfonts}
\usepackage{eurosym}
\usepackage{amssymb}
\usepackage{amsthm}
\usepackage{amsmath}
\usepackage{amsaddr}
\usepackage{bm}
\usepackage{cite}
\usepackage{mathrsfs}
\usepackage{xcolor}
\usepackage[OT1]{fontenc}
\usepackage[left=2.5cm, right=2.5cm, top=3cm, bottom=2.5cm]{geometry}
\usepackage{hyperref}
\usepackage{graphicx}
\usepackage{mathtools}
\usepackage{bigints}

\hypersetup{colorlinks=true, linkcolor=blue, citecolor=red, urlcolor=blue}

\usepackage{times}

\usepackage{xpatch}
\makeatletter
\AtBeginDocument{\xpatchcmd{\@thm}{\thm@headpunct{.}}{\thm@headpunct{}}{}{}}
\makeatother

\allowdisplaybreaks
\newtheorem{theorem}{Theorem}[section]

\newtheorem{lemma}[theorem]{Lemma}
\newtheorem{proposition}[theorem]{Proposition}
\theoremstyle{remark}
\newtheorem{remark}[theorem]{Remark}

\def\non{\nonumber }

\newcommand \norm[1] {\left\Vert #1\right\Vert}
\def\n{\textbf{\textit{n}}}
\def\R{\mathbb{R}}

\def\G{\mathcal{G}}
\def\F2o{\overline{F_2}}

\def\d{{\rm d}}
\def \l {\langle}
\def \r {\rangle}

\def\u{u}

\def\ddt{\frac{\d}{\d t}}

\def\E{\mathcal{E}}
\def\div{\mathrm{div}}

\DeclareMathOperator*{\esssup}{ess\,sup}
\def \au {\rm}
\def \ti {\it}
\def \jou {\rm}
\def \bk {\it}
\def \no#1#2#3 {{\bf #1} (#3), #2.}
\def \eds#1#2#3 {#1, #2, #3.}
\def \nome#1#2 {{\bf #1}, (#2).}

\newcommand{\sref}[2]{\hyperref[#2]{#1 \ref*{#2}}}

\numberwithin{equation}{section}

\def\d{{\rm d}}

\newcommand{\habil}[1]{}

\newcommand{\uloc}{\operatorname{uloc}}

\newcommand{\eps}{\ensuremath{\varepsilon}}

\newcommand{\vp}{\phi}

\def \au {\rm}
\def \ti {\it}
\def \jou {\rm}
\def \bk {\it}
\def \no#1#2#3 {{\bf #1} (#3), #2.}
\def \eds#1#2#3 {#1, #2, #3.}

\makeatletter
\def\@settitle{\begin{center}%
  \baselineskip14\p@\relax
    \huge
  \@title
  \end{center}%
}
\makeatother

\begin{document}

\title[Cahn-Hilliard equation with nonlinear diffusion]
{\emph{On the Cahn-Hilliard equation with nonlinear diffusion: \\ the non-convex case}}
\author[Monica Conti, Stefania Gatti, Andrea Giorgini \& Giulio Schimperna]{Monica Conti, Stefania Gatti, Andrea Giorgini \& Giulio Schimperna}

\address{Politecnico di Milano\\
Dipartimento di Matematica\\
Via E. Bonardi 9, I-20133 Milano, Italy \\
\href{mailto:monica.conti@polimi.it}{monica.conti@polimi.it},
\href{mailto:andrea.giorgini@polimi.it}{andrea.giorgini@polimi.it}}

\address{Università degli Studi di Modena e Reggio Emilia\\
Dipartimento di Scienze Fisiche, Informatiche e Matematiche\\
Via Campi 213/B, I-41125 Modena, Italy
\\
\href{mailto:stefania.gatti@unimore.it}{stefania.gatti@unimore.it}}

\address{Universit\`{a} degli Studi di Pavia\\
Dipartimento di Matematica ``F. Casorati", IMATI - C.N.R. \\
via Ferrata 5, 27100, Pavia, Italy \\
\href{mailto:giulio.schimperna@unipv.it}{giulio.schimperna@unipv.it}}



\begin{abstract}
We investigate the Cahn-Hilliard equation with nonlinear diffusion and non-degenerate mobility modeling phase separation phenomena in complex systems (e.g., crystals and polymers). Previous results in the literature on this model relied on the strong convexity assumption of the gradient part of the energy, which excludes relevant cases. In this work, we remove the convexity condition and establish new qualitative properties of solutions under general assumptions on the diffusion and mobility functions. In two spatial dimensions, we prove uniqueness of weak solutions, their smoothing effect for positive times, and convergence to equilibrium as time tends to infinity. In three dimensions, we show local well-posedness of strong solutions for arbitrary initial data and global existence for data close to energy minimizers, yielding a Lyapunov stability principle. A key ingredient of our analysis is a Lojasiewicz-Simon inequality tailored to the nonlinear diffusion case, which enables us to characterize the longtime dynamics.
\end{abstract}

\maketitle

\section{Introduction}
\label{Sec-Intro}

We study the Cahn-Hilliard equation with nonlinear diffusion
\begin{align}
 \label{CH1}
\partial_t \phi &= \div(b(\phi)\nabla \mu),\\
 \label{CH2}
\mu &= -\div( a(\phi) \nabla \phi) + a'(\phi) \frac{|\nabla \phi|^2}{2}+\Psi'(\phi).
\end{align}
The initial-boundary value problem related to \eqref{CH1}-\eqref{CH2} is considered in $\Omega \times (0,\infty)$, where $\Omega$ is a bounded domain in $\mathbb{R}^d$, with $d=2$ and $d=3$. The system  \eqref{CH1}-\eqref{CH2}  is equipped with the following homogeneous Neumann boundary conditions
\begin{equation}
\label{CH-bc}
\partial_\n\phi= b(\phi) \partial_\n \mu=0 \quad \text{on }  \partial \Omega \times (0,\infty),
\end{equation}
and initial condition
\begin{equation}
\label{CH-ic}
 \phi|_{t=0}=\phi_0 \quad \text{in } \Omega.
\end{equation}
Here, $\n$ is the unit outward normal vector on $\partial \Omega$, and
 $\partial_\n$ denotes the outer normal derivative on $\partial \Omega$.
The state variable of the system is the difference of two fluid concentrations $\phi\colon \Omega \times [0,\infty)\to [-1,1]$ (order parameter).
The function $a\colon [-1,1]\to [0,\infty)$ represents nonlinear diffusion under the assumptions
\begin{equation}
a \in C^2([-1,1]): 0 < a_m \le a(s) \le a_M, \ s \in [-1,1].
\label{a-ndeg}
\end{equation}
The function $b\colon [-1,1]\to [0,\infty)$ is the non-degenerate Onsager mobility, which satisfies
\begin{equation}
b \in C^1([-1,1]): 0 < b_m \le b(s) \le b_M, \ s \in [-1,1].
\label{m-ndeg}
\end{equation}
The homogeneous free energy density $\Psi$ is the Flory-Huggins potential
\begin{equation}
\label{Log}
\Psi(s)=F(s)-\frac{\theta_0}{2}s^2=\frac{\theta}{2}\bigg[ (1+s)\log(1+s)+(1-s)\log(1-s)\bigg]-\frac{\theta_0}{2} s^2, \quad s \in [-1,1],
\end{equation}
where $\theta<\theta_0$ are constant positive parameters. The system \eqref{CH1}-\eqref{CH-bc} is a generalization of the classical Cahn-Hilliard equation, which is recovered when $a\equiv 1$.
The associated Ginzburg-Landau free energy is
\begin{equation}
\label{free-energy}
E(\phi) = \int_\Omega \frac{a(\phi)}{2} |\nabla \phi|^2 + \Psi(\phi) \,\d x,
\end{equation}
and \eqref{CH1}-\eqref{CH2}, together with \eqref{CH-bc}, is the resulting mass conserving gradient flow with respect to the metric of the dual space of $H^1(\Omega)$. The importance of including a nonlinear coefficient in the interfacial free energy density, depending on the order parameter, stems from the modeling of complex systems such as crystals, polymers and hydrogels (see, e.g., \cite{S2}). 
Different forms of $a$ significantly affect the free energy costs of interfaces between components. For instance, specific approximations arising from the description of phase separation include polynomials such as 
$$a(u) = a_0+ a_1u + a_2u^2,$$
where $a_0>0$ (cf. \cite{S36}), or even singular coefficients of de Gennes type as in \cite{S23}. We also mention that more general Ginzburg-Landau free energy (w.r.t. \eqref{free-energy}) are employed in amphiphilic systems (see, e.g., \cite{GS})

The mathematical analysis of system \eqref{CH1}-\eqref{CH-ic} has been intensively studied over the last decades in the constant diffusion case, namely $a\equiv1$. For the interested reader, we mention the works with constant mobility \cite{ABELS2009, AW07, DD, EL1991, GGW2018, MZ},
non-degenerate mobility \cite{BB1999, CGGG2025, S2007, SZ2013}, as well as degenerate mobility \cite{CMN2019, EG1996, DD2016, LMS2012},
and the references therein.
On the other hand, the nonlinear diffusion case $a(\phi)$ has only been partially developed so far. The main contribution in the literature is the work by the fourth author and Pawlow \cite{SP2013} (see, in particular, Sections 5 and 6) referring to the case with constant mobility. Therein, the authors first showed the existence of global weak solutions  under the assumption \eqref{a-ndeg}. The uniqueness of weak solutions and their regularization properties for positive times were achieved under the assumptions
\begin{equation}
\label{convex-condition}
a''(s)\geq 0, \quad \left( \frac{1}{a} \right)''(s)< - \kappa, \quad \forall \, s \in [-1,1],
\end{equation}
where $\kappa>0$. As observed in \cite{SP2013}, the above conditions entails that the gradient part of the energy
$$
J(\phi)= \int_\Omega \frac{a(\phi)}{2} |\nabla \phi|^2 \, \d x
\quad \text{is \textit{strictly convex}} .
$$
Concurrently, a Navier-Stokes/Cahn-Hilliard system, including \eqref{CH1}-\eqref{CH2} with non-degenerate or degenerate mobility, together with nonlinear diffusion, was studied by Abels, Garcke and Depner in \cite{ADG2013, ADG2013-2}, where they proved the existence of global weak solutions. Therefore, several mathematical questions on the model in presence of a \textit{non-convex} gradient energy part remains largely unexplored.

In this work, we investigate fine qualitative properties of the solutions to the Cahn-Hilliard system \eqref{CH1}-\eqref{CH-ic} without imposing the convexity assumption \eqref{convex-condition}. Our aim is twofold: to address uniqueness and regularity of solutions in the general setting  \eqref{a-ndeg}; to characterize the longtime behavior of solutions, a question that has not been explored so far, even in the \textit{convex} case \eqref{convex-condition}.
Under the general setting \eqref{a-ndeg}, our main results are summarized here below:
\begin{itemize}
\item Theorem \ref{well-pos-2D}:
In two space dimensions, the global weak solutions are unique, propagate regularity for positive times, and converge to an equilibrium point of the free energy $E(\phi)$ as $t\to \infty$.
\medskip

\item Theorem \ref{well-pos-3D}: In three space dimensions, strong solutions exist locally in time for arbitrarily large initial data, and globally in time for initial data sufficiently close to energy minimizers. The latter entails a Lyapunov stability principle. Furthermore, such global strong solutions converge to an equilibrium point of the free energy $E(\phi)$ as $t\to \infty$.
\end{itemize}

Before detailing our main results, let us point out some novelties of this work.
First of all, we show that any weak solution to the system exhibits enhanced regularity properties.
A key tool at this stage is a tailored regularity theory
for the associated stationary problem with nonlinear diffusion that we derive in Section \ref{s-stationary} based on a suitable reformulation of the
chemical potential borrowed from \cite{ADG2013}.
In particular, obtaining that any weak solution belongs to $L^4(0,\infty; H^2(\Omega))$ enables us to extend the techniques of \cite{CGGG2025}
to address the combined effects of the nonlinear diffusion $a$ and the variable mobility $b$.
This leads to uniqueness and global regularity results in two spatial dimensions.
Furthermore, based on the improved regularity of weak solution at any positive time, we resolve an open question from \cite[Remark 6.4]{SP2013} regarding the validity of the energy equality for $E$ in the non-convex setting. In summary, our analysis goes beyond the classical Cahn–Hilliard equation with constant coefficients,
providing a comprehensive well-posedness framework for the general nonlinear, non-degenerate system in two dimensions.

Secondly, we establish a crucial tool for analyzing the long-time behavior of solutions, as well as proving Lyapunov stability
in the spirit of Simon \cite{SIMON}: the Lojasiewicz–Simon inequality. For the classical Cahn–Hilliard energy with constant coefficient,
 this inequality was first derived in \cite{HR1999} and later refined in \cite{AW07}. However, such a tool has not yet been developed in the setting of nonlinear diffusion.
In order to present the new Lojasiewicz-Simon inequality related to the free energy \eqref{free-energy} with nonlinear coefficient $a$, we recall the stationary problem corresponding to the evolution system
\eqref{CH1}-\eqref{CH-ic}: 
\begin{alignat}{2}
\label{SSS-i}
-\div(a(\u)\nabla \u)+\frac{a'(\u)}{2} |\nabla \u|^2+\Psi'(\u)&=\overline{\frac{a'(\u)}{2} |\nabla \u|^2+\Psi'(\u)} \quad &&\text{in } \Omega, \\
\label{SSS-ii}
\partial_\n \u&=0 \quad &&\text{on } \partial \Omega,
\end{alignat}
together with
$$
\overline{\u}=\overline{\phi_0},
$$
where the notation $\overline{f}:=\frac{1}{|\Omega|}\int_\Omega f\,\d x$ is the spatial average.
For $m\in (-1,1)$, we denote the set of the stationary states as
\begin{align}
\mathcal{S}_m=&\big\lbrace \u\in H^2(\Omega) \text{ with }
\Psi'(\u) \in L^2(\Omega): \
\u \text{ solves } \eqref{SSS-i}-\eqref{SSS-ii} \text{ and } \overline{\u}=m
    \big\rbrace.\label{SSS1}
\end{align}
We show that
\begin{theorem}[Lojasiewicz--Simon inequality]
\label{LSg-intro}
Let $\Omega\subset\R^d$, $d=2,3$, be a bounded smooth domain. Assume that $a$ is real analytic on the open interval $(-1,1)$. Then, for any $\psi\in \mathcal{S}_m$, there exist $\theta\in (0, \frac12)$, $C_L>0$ and
$\beta >0$ such that the following inequality holds
\begin{align}
|E(u) -E(\psi)|^{1-\theta}\leq C_L \norm{-\div(a(u)\nabla u)+\frac{a'(u)}{2} |\nabla u|^2+\Psi'(u)-\overline{\frac{a'(u)}{2} |\nabla u|^2+\Psi'(u)}}_{L^2(\Omega)},\label{LS0}
\end{align}
for all $u\in H^2(\Omega)$ such that $\overline{u}=m$, $\partial_\n u=0$ on $\partial\Omega$ and $\|u-\psi\|_{H^2(\Omega)}\leq\beta$.
\end{theorem}

The proof of Theorem \ref{LSg-intro} relies on the approach proposed in \cite{RUPP}, where the stationary points are interpreted as constrained critical points of the free energy on a manifold $\mathcal{M}$ which accounts for the conservation of mass. It is important to point out that, on the contrary to the classical case with constant diffusion, it is not possible to work in the energy space $H^1(\Omega)$, but we rather need to consider the smaller space $H^2(\Omega)$, which requires an assumption on the vicinity of $u$ to $\psi$ in $H^2(\Omega)$.  This is due to the fact that, in the $H^1(\Omega)$ framework, the Frechet derivative $J'$ is not well defined due to the quadratic term arising from $a(\phi)$.

\medskip
\subsection{Main results}
For $m\in (-1,1)$, we introduce the space of admissible finite energy data
$$
\mathcal{V}_m = \Big\lbrace f \in H^1(\Omega)\cap L^\infty(\Omega): \| f\|_{L^\infty(\Omega)}\leq 1 \text{ and } \overline{f}=m \Big\rbrace,
$$
and we study the evolution system \eqref{CH1}-\eqref{CH-ic} with initial datum in $\mathcal{V}_m$.
In the two dimensional case, our results are collected in the following (some notation used in the statement will be specified in Subsection \ref{funct} below):

\begin{theorem}
\label{well-pos-2D}
Let $\Omega\subset \R^2$ be a bounded domain with smooth boundary. Assume that $\phi_0\in \mathcal{V}_m$. 
\medskip

\begin{enumerate}
\item \underline{Global weak solutions}: There exists a global weak solution $\phi$ to \eqref{CH1}-\eqref{CH-ic} in the following sense:
\begin{align}
\label{WS1}
&\phi \in  L^\infty( 0,\infty; \mathcal{V}_m)\cap L_{\uloc}^4([0,\infty);H^2(\Omega)),
\\
\label{WS2}
&\phi \in L^{\infty}(\Omega\times (0,\infty)) \text{ such that }  |\phi(x,t)|<1  \ \text{a.e. in }\Omega\times (0,\infty),
\\
\label{WS3}
&\partial_t \phi \in L^2(0,\infty; H_{(0)}^{-1}(\Omega)),
\\
\label{WS4}
& \mu \in L_{\uloc}^2([0,\infty);H^1(\Omega)),
\end{align}
and, for any $2\leq p <\infty$,
\begin{equation}
\label{WS5}
\phi\in L_{\uloc}^2([0,\infty);W^{2,p}(\Omega))\quad\text{and}\quad F'(\phi)\in L_{\uloc}^2([0,\infty);L^p(\Omega)),
\end{equation}
where 
\begin{align}
\label{e2}
\l \partial_t \phi,v\r
+ ( b(\phi) \nabla \mu, \nabla v )=0,
\quad \forall \, v \in H^1(\Omega), \ \text{a.e. in } (0,\infty),
\end{align}
and the chemical potential is given by
\begin{equation}
\label{e3}
\mu=-\div \left( a(\phi) \nabla \phi \right) + a'(\phi) \frac{|\nabla \phi|^2}{2}+\Psi'(\phi) \quad \text{a.e. in } \Omega \times (0,\infty).
\end{equation}
Moreover, $\partial_\n \phi=0$ almost everywhere
on $\partial\Omega\times(0,\infty)$, and
$\phi(\cdot,0)=\phi_0$ in $\Omega$.
\medskip

\item \underline{Uniqueness}: Assume that $b \in C^2([-1,1])$. Let $\phi_1, \phi_2$ be two weak solutions starting from $\phi_1^0, \phi_2^0$, respectively, with $\overline{\phi_1^0}=\overline{\phi_2^0}$. Then, for any $T>0$, there exists $C>0$ such that
\begin{equation}
\label{UNIQ}
\norm{\phi_1(t)-\phi_2(t)}_{H_{(0)}^{-1}(\Omega)}
\leq C \norm{ \phi_1^0-\phi_2^0 }_{H_{(0)}^{-1}(\Omega)}\!, \quad \forall \, t \in [0,T].
\end{equation}
The constant $C$ depends only on the parameters of the system, the final time $T$, and the initial free energies $E(\phi_1^0)$ and $E(\phi_2^0)$. In particular, the weak solution is unique.
\medskip


\item \underline{Propagation of regularity}: For any $\tau>0$, there holds
\begin{align}
\label{RS1}
& \phi \in L^\infty(\tau, \infty; H^3(\Omega)), \quad \partial_t \phi \in L_{\uloc}^2([\tau,\infty);H^1(\Omega))\cap L^\infty(\tau,\infty; H^1(\Omega)'),
\\
\label{RS2}
& \phi \in C(\overline{\Omega} \times [\tau, \infty)) \text{ such that }  |\phi(x,t)|\leq 1-\delta  \ \text{everywhere in } \overline{\Omega}\times [\tau,\infty),
\\
\label{RS3}
&\mu \in L^{\infty}(\tau,\infty; H^1(\Omega))\cap L_{\uloc}^2([\tau,\infty);H^3(\Omega)),
\end{align}
for some $\delta \in (0,1)$ depending on $\tau$ and the norm of $\phi_0$.
In particular, $\phi$ satisfies \eqref{CH1}-\eqref{CH2} almost everywhere in $\Omega \times (\tau,\infty)$ and  $\partial_\n \phi=\partial_\n \mu=0$ almost everywhere on $\partial\Omega\times(\tau,\infty)$.
\medskip

\item \underline{Energy equality}:
The weak solution satisfies the energy equality
\begin{equation}
\label{EI}
E(\phi(t))+ \int_0^t 
\left\| \sqrt{b(\phi(s))} \nabla \mu(s)\right\|_{L^2(\Omega)}^2 \, \d s =  E (\phi_0),
\end{equation}
for every $0\leq t < \infty$. Furthermore, $\phi \in C([0,\infty); H^1(\Omega))$.
\medskip

\item \underline{Convergence to equilibrium}: Let $a$ be analytic in $(-1,1)$, $b \in C^2([-1,1])$. Then, there exists a unique $\phi_\infty \in \mathcal{S}_m$ such that
$$\lim_{t\to\infty}\phi(t)=\phi_\infty\quad \text{in } H^2(\Omega).$$
\end{enumerate}
\end{theorem}

Concerning the evolutionary problem in a three dimensional domain, the main results are stated in the following.
\begin{theorem}
\label{well-pos-3D}
Let $\Omega\subset \R^3$ be a bounded smooth domain, and let $\phi_0\in \mathcal{V}_m$.
\begin{enumerate}
\item \underline{Global weak solutions}: There exists a global weak solution $\phi$ to
\eqref{CH1}-\eqref{CH-ic} such that
\begin{align}
\label{WS1-3}
&\phi \in  L^\infty( 0,\infty; \mathcal{V}_m)\cap L_{\uloc}^4([0,\infty);H^2(\Omega))\cap L_{\uloc}^2([0,\infty);W^{2,6}(\Omega)),
\\
\label{WS2-3}
&\phi \in L^{\infty}(\Omega\times (0,\infty)) \text{ such that }  |\phi(x,t)|<1  \ \text{a.e. in }\Omega\times (0,\infty),
\\
\label{WS3-3}
&\partial_t \phi \in L^2(0,\infty; H_{(0)}^{-1}(\Omega)),
\\
\label{WS4-3}
& \mu \in L_{\uloc}^2([0,\infty);H^1(\Omega)), \quad F'(\phi)\in L_{\uloc}^2([0,\infty);L^6(\Omega)),
\end{align}
which satisfies \eqref{e2}-\eqref{e3}, and $\partial_\n \phi=0$ almost everywhere on $\partial\Omega\times(0,\infty)$, as well as
$\phi(\cdot,0)=\phi_0$ in $\Omega$.

\item \underline{Local strong solutions}: Assume that $\phi_0\in \mathcal{V}_m\cap H^2(\Omega)$, $\partial_\n \phi_0=0$ on $\partial \Omega$ and
\begin{equation}
\label{MM-i}
\left\| -\div \left( a(\phi_0) \nabla \phi_0 \right) + a'(\phi_0) \frac{|\nabla \phi_0|^2}{2} + \Psi'(\phi_0) \right\|_{H^1(\Omega)}\leq M,
\end{equation}
for some $M>0$. Then, there exist $T_M>0$, depending only on $M$, $E(\phi_0)$ and $m$, and a unique strong solution $\phi$ to \eqref{CH1}-\eqref{CH-ic} on $[0,T_M]$
in the following sense
\begin{equation}
\label{SS-phi-i}
\begin{split}
&\phi \in L^\infty(0,T_M; W^{2,6}(\Omega)), \quad
\partial_t \phi \in L_{\uloc}^2(0 ,T_M;H^1(\Omega)),\\
&\phi \in L^{\infty}(\Omega\times (0,T_M)) \quad\text{such that}\quad |\phi(x,t)|<1
\ \text{a.e. }(x,t)\in \Omega\times (0,T_M),\\
&\mu \in L^\infty(0,T_M; H^1(\Omega)),
\quad F'(\phi) \in L^\infty(0,T_M; L^6(\Omega)).
\end{split}
\end{equation}
Besides, $\phi$ fulfills the equations \eqref{CH1}-\eqref{CH2} almost everywhere in $\Omega \times (0,T_M)$.

\medskip

\item \underline{Lyapunov Stability}: Let $a$ be analytic in $(-1,1)$ and let $\psi\in \mathcal{V}_m\cap H^2(\Omega)$ be a local minimizer of the energy $E$ in $\mathcal{V}_m$.
Assume that the initial datum $\phi_0\in \mathcal{V}_m\cap H^2(\Omega)$, $\partial_\n \phi_0=0$ on $\partial \Omega$ and satisfies \eqref{MM-i}.
Then, for any $\epsilon>0$, there exists a constant $\eta\in (0,1)$, depending on $\epsilon$ and $M$, such that, if
$$\|\vp_0-\psi\|_{H^2(\Omega)}\leq \eta,$$ the unique strong solution to  \eqref{CH1}-\eqref{CH-ic} is global in time. Besides, $\phi\in L^\infty(0,\infty; H^3(\Omega))$
and 
\begin{align}
\|\vp(t)-\psi\|_{H^2(\Omega)} \leq \epsilon,\quad \forall\, t\geq 0.\nonumber
\end{align}

\item \underline{Convergence to equilibrium}: Let the assumptions of part (3) be in place. Then, there exists a unique $\phi_\infty \in \mathcal{S}_m$ such that the unique global strong solution $\phi$ to  \eqref{CH1}-\eqref{CH-ic} obtained in part (3) satisfies
$$\lim_{t\to\infty}\phi(t)=\phi_\infty\quad \text{in } H^2(\Omega).$$

\end{enumerate}
\end{theorem}
\begin{remark}
We point out that the assumption $\psi\in H^2(\Omega)$ is not a restriction. Indeed, we prove in Section \ref{S-Lyapunov} that any energy minimizer satisfies such regularity.
\end{remark}
\section{Mathematical setting}
\label{S-Mathset}

\subsection{Function spaces}
\label{funct}
For $X$ Banach space, the set $X'$ denotes the dual space and $\l \cdot , \cdot \r_{X',X}$ the duality product. For $ 1 \le p \le \infty$ and an interval $I \subseteq [0, \infty)$, $L^p( I; X)$ is the space of all Bochner $p-$integrable functions defined from $I$ into $X$. For $1 \le p < \infty$, the set $L^p_{\rm uloc}([0,\infty), X)$ denotes the space of functions $f \in L^p(0,T;X)$ for any $T>0$ such that there exists $C>0$ such that
$$\norm{ f }_{L^p_{\rm uloc}([0,\infty), X)} :=
\sup_{t \ge 0} \int_t^{t+1}\norm{f(s)}_X^p \, \d s <\infty.
$$
The set of continuous functions $f: I \to X$ is denoted by $C(I;X)$. 
The space $BC_{\rm w}(I;X)$ consists of the topological vector space of
all $X$-valued bounded and weakly continuous functions (i.e., the map $t\in I \mapsto \langle \phi, f(t) \rangle_{X^*, X}$ is continuous for all $\phi \in X^*$). 
Finally, the set $BUC(I; X)$ is the subspace of all $X$-valued bounded and uniformly continuous functions with respect to the supremum norm.

Let $\Omega$ be a bounded smooth domain in $\R^d$, $d=2,3$. For any positive real $k$ and $1\leq p \leq \infty$, we denote by $W^{k,p}(\Omega)$ the Sobolev space of function in $L^p(\Omega)$ such that the distributional derivative of order up to $k$ is an element of $L^p(\Omega)$.  We use the notation $H^k(\Omega)$ for the Hilbert space $W^{k,2}(\Omega)$ with norm $\norm{\cdot}_ {H^k(\Omega)}$. In particular, for $k=1$, we denote the duality product between $H^1(\Omega)'$ and $H^1(\Omega)$ by $\l \cdot, \cdot \r$.

Given the definition of spatial average
\begin{equation*}
\overline{f}= \frac{1}{|\Omega|} \int_{\Omega} f(x) \, \d x \quad \text{if } f \in L^1(\Omega),
\end{equation*}
for any $m \in \mathbb{R}$, we introduce the function space
$$
H^1_{(m)}(\Omega)=\left\lbrace f \in H^1(\Omega) : \ \overline{f}=m \right\rbrace \! .
$$
If $m=0$, we recall the Poincar\'{e}-Wirtinger inequality
\begin{equation}\label{poincare} 
\norm{f}_{L^2(\Omega)} \le C_P  \norm{\nabla f}_{L^2(\Omega)} \!, \quad \forall \, f \in H_{(0)}^1(\Omega),
\end{equation}
where $C_P$ only depends on $\Omega$. Thus, 
$H^1_{(0)}(\Omega)$ is a Hilbert space endowed with inner product $(f,g)_{H^1_{(0)}(\Omega)}=(\nabla f, \nabla g)$ and corresponding norm $\| f\|_{H^1_{(0)}(\Omega)}=\| \nabla f \|_{L^2(\Omega)}$.
Extending the definition of total mass as
\begin{equation*}
\overline{f} = \frac{1}{|\Omega|} \l f,1 \r \quad \text{if } f \in H^1(\Omega)',
\end{equation*}
we define the Hilbert space
\begin{equation}
\label{Spazi-mn}
H^{-1}_{(0)}(\Omega)=\left\lbrace f  \in H^1(\Omega)' : \ \overline{f}=0 \right\rbrace \!,
\end{equation}
with corresponding norm defined as $\| f\|_{H^{-1}_{(0)}(\Omega)}:= \| f\|_{H^1(\Omega)'}$.

We also report the generalized Poincar\'{e} inequality (see \cite[Chapter II, Section 1.4]{TEMAM}):
there exists a positive constant $C=C(\Omega, \eta_\ast,\eta^\ast)$ such that
\begin{equation}
\label{normH1-2}
\| u\|_{H^1(\Omega)}\leq C \Big(\| \nabla u\|_{L^2(\Omega)}^2+ \Big|\int_{\Omega}  \eta u \, \d x\Big|^2\Big)^\frac12, \quad \forall \, u \in H^1(\Omega),
\end{equation}
where $\eta \in L^\infty(\Omega)$ is such that $0<\eta_\ast\leq \eta(x)\leq \eta^\ast$ for almost every $x\in \Omega$. 

Furthermore, the following Gagliardo-Nirenberg interpolation inequalities hold:
\begin{equation}
\label{GN-utile}
\| \nabla u \|_{L^{2p}(\Omega)} \leq C \| u\|_{L^\infty(\Omega)}^\frac12 \| u\|_{W^{2,p}(\Omega)}^\frac12, \quad \forall \, u \in W^{2,p}(\Omega), \ 2\leq p \leq \infty,
\end{equation}
which, in particular, gives
\begin{equation}
\label{inter-W14}
\| \nabla u \|_{L^4(\Omega)} \leq C \| u\|_{L^\infty(\Omega)}^\frac12 \| u\|_{H^2(\Omega)}^\frac12,
\end{equation}
Lastly, we have 
\begin{align}
\label{lady_2}
\| u \|_{L^4(\Omega)} \leq C \| u\|_{L^2(\Omega)}^\frac12 \| u\|_{H^1(\Omega)}^\frac12, \quad d=2,
\\
\label{lady_3}
\|  u \|_{L^3(\Omega)} \leq C \| u\|_{L^2(\Omega)}^\frac12 \| u\|_{H^1(\Omega)}^\frac12, \quad d=3.
\end{align}

\subsection{Elliptic problems.} 
We introduce the solution operator of the Laplace problem with homogeneous Neumann boundary condition $\G: H^{-1}_{(0)}(\Omega) \to H^{1}_{(0)}(\Omega)$ such that, for every $f \in H^{-1}_{(0)}(\Omega)$,  $\G f \in H^{1}_{(0)}(\Omega)$ satisfies
\begin{equation}\label{Gdef}
(\nabla \G f, \nabla v) = \l f, v \r, \quad \forall \, v \in H^{1}_{(0)}(\Omega).
\end{equation}
For any $f \in H^{-1}_{(0)}(\Omega)$, $\norm{\nabla \G f}_{L^2(\Omega)}$ is a norm on $H^{-1}_{(0)}(\Omega)$, that is equivalent to $\| f\|_{H^{-1}_{(0)}(\Omega)}$. Furthermore,
\begin{equation}\label{interp}
| \l f, v \r | \leq  \norm{\nabla \G f}_{L^2(\Omega)} \norm{\nabla v}_{L^2(\Omega)}, \quad \forall \, f \in H^{-1}_{(0)}(\Omega), v \in H^{1}_{(0)}(\Omega).
\end{equation}
Next, assuming \eqref{m-ndeg} and given a measurable function $q: \Omega \to [-1,1]$, we introduce the following elliptic problem
\begin{equation}\label{epgq}
\begin{cases}
-\div(b(q) \nabla u) = f \quad & \text{in}\ \Omega \\
b(q) \partial_\n u = 0 \quad & \text{on} \ \partial\Omega.
\end{cases}
\end{equation}
Similarly to the definition of $\G$ in \eqref{Gdef}, we define the solution operator $\G_q:  H^{-1}_{(0)}(\Omega)  \to  H^{1}_{(0)}(\Omega)$ as follows:
for every $f \in H^{-1}_{(0)}(\Omega)$,  $\G_q f \in H^{1}_{(0)}(\Omega)$ such that
\begin{equation}\label{def G_q}
( b(q) \nabla \G_q f ,\nabla v ) = \l f, v \r, \quad \forall \, v \in  H^{1}_{(0)}(\Omega).
\end{equation}
It is immediate to see that
\begin{equation}\label{equivalence of norm}
\sqrt{b_m} \norm{\sqrt{b(q)} \nabla \G_q f}_{L^2(\Omega)}
\le \norm{ \nabla \G f}_{L^2(\Omega)}
\le \sqrt{b_M}\norm{\sqrt{b(q)} \nabla \G_q f}_{L^2(\Omega)}\!, \quad \forall \, f \in  H^{-1}_{(0)}(\Omega),
\end{equation}
from which it follows that $\norm{\nabla \G_q f}_{L^2(\Omega)}$ is a norm on $ H^{-1}_{(0)}(\Omega)$ that is equivalent to $\norm{\nabla \G f}_{L^2(\Omega)}$.

We now report some elliptic estimates related to problem \eqref{epgq}. Let $f \in L^2(\Omega)	\cap  H^{-1}_{(0)}(\Omega)$. Assuming that $b\in C^1([-1,1])$ and $q \in H^2(\Omega)$,  we have
\begin{align}
\label{H2}
\| \mathcal{G}_{q} f\|_{H^2(\Omega)}\leq C \left( \| \nabla q\|_{L^2(\Omega)} \| q\|_{H^2(\Omega)} \| \nabla \mathcal{G}_q f\|_{L^2(\Omega)}+ \| f\|_{L^2(\Omega)} \right),  \quad d=2.
\\
\label{H2-3D}
\| \mathcal{G}_{q} f\|_{H^2(\Omega)}\leq C \left( \| q\|_{H^2(\Omega)}^2 \| \nabla \mathcal{G}_q f\|_{L^2(\Omega)}+ \| f\|_{L^2(\Omega)} \right),  \quad d=3.
\end{align}
More properties of $\G_q$ can be found in \cite{BB1999} and \cite{CGGG2025}.
\medskip

The following is a classical result in the theory of elliptic systems (see, e.g., \cite[Lemma 4]{ABELS2009}).
\begin{proposition}
\label{NP-nc}
Let $d=2,3$. Assume that $K \in C^1(\mathbb{R})$ such that
$0<\underline{K}\leq K(s)\leq \overline{K}$ for all $s \in \mathbb{R}$ and let 
$\varphi: \Omega \to \mathbb{R}$ be a measurable function.
Consider the Neumann problem
\begin{equation}
\label{NP-2}
\begin{cases}
- \mathrm{div}\, ( K(\varphi) \nabla u) + u =f,\quad &\text{ in }\Omega,\\
\partial_\n u=0, \quad &\text{ on }\partial \Omega.
\end{cases}
\end{equation}
If $\varphi \in W^{1,r}(\Omega)$, with $d<r \leq \infty$, and $f\in L^2(\Omega)$. Then, $u\in H^2(\Omega)$ and $\partial_\n u=0$ on $\partial \Omega$. Moreover, there exists a positive increasing function $Q$ depending on $\underline{K}$, $\overline{K}$ and $r$ such that,
$$
\| u\|_{H^2(\Omega)}\leq Q(R) \| f\|_{L^2(\Omega)},
$$
where $\| \varphi\|_{W^{1,r}(\Omega)}\leq R$.

\end{proposition}


\section{The stationary system}
\label{s-stationary}

Let $d=2,3$ and let $m\in (-1,1)$ be given. In this section we provide some regularity properties for the elements of $\mathcal{S}_m$.
To this aim, following \cite{ADG2013}, it is convenient to introduce the integral function $A:\R\to\R$
\begin{equation}
\label{AAA}
A(s):= \int_0^s \sqrt{a(\tau)}\, \d \tau.
\end{equation}
Since $\nabla A(\u)= A'(\u)\nabla \u= \sqrt{a(\u)}\nabla \u$, we notice that, for $\u\in \mathcal{V}_m \cap H^2(\Omega)$,
\begin{equation}
\label{div-Aphi}
-\sqrt{a(\u)} \Delta A(\u)=- \div (a(\u)\nabla \u)+a'(\u) \frac{|\nabla \u|^2}{2},
\end{equation}
and
\begin{equation}
\label{phi-Aphi-rel}
\Delta u= \frac{1}{\sqrt{a(u)}} \Delta A(u) - \frac{a'(u)}{2 a^2(u)} | \nabla A(u)|^2.
\end{equation}
In addition, we have
$$
\partial_\n A(\u)= \sqrt{a(\u)} \partial_\n \u=0 \quad \text{on } \partial \Omega \times (0,\infty).
$$
Accordingly, if $\u\in \mathcal{S}_m$, then $\u$ solves the following quasilinear elliptic problem with logarithmic nonlinear term
\begin{equation}
\label{NEUMANNSING}
\begin{cases}
-\sqrt{a(\u)}\Delta A(\u)+F'(\u)=f,\quad &\text{ in }\Omega,\\
\partial_\n A(\u)=0, \quad &\text{ on }\partial \Omega,
\end{cases}
\end{equation}
where $f$ is given by
$$f=\overline{\frac{a'(\u)}{2} |\nabla \u|^2+\Psi'(\u)}+\theta_0 \u.$$

\medskip

In the following, let $f\in L^2(\Omega)$ be a generic function and assume that $\u$ is any solution to \eqref{NEUMANNSING} in the following sense: $\u\in H^2(\Omega)$ with $F'(\u)\in L^2(\Omega)$, which satisfies \eqref{NEUMANNSING} almost everywhere. In particular, we observe that $\| \u\|_{L^\infty(\Omega)}\leq 1$. Then, the following result holds true.

\begin{lemma}
\label{ell2}
Let $d=2,3$ and  $f\in L^p(\Omega)$ with $2\leq p \leq \infty$.
Then, we have
$$
\| F'(\u)\|_{L^p(\Omega)}\leq C\| f\|_{L^p(\Omega)},
$$
for some positive constant $C$ independent of $p$.
\end{lemma}

\begin{proof}
For $k\in \mathbb{N}$, let us define the globally Lipschitz function $h_k: \mathbb{R}\rightarrow \mathbb{R}$
\begin{equation}
\label{trunc}
h_k(s)=
\begin{cases}
-1+\frac1k,&s<-1+\frac1k, \\
s,& s\in [-1+\frac1k,1-\frac1k],\\
1-\frac1k,&s>1-\frac1k,
\end{cases} \quad \text{and }\quad \u_k=h_k \circ \u.
\end{equation}
Since $\u\in H^1(\Omega)$, the classical result on compositions in Sobolev spaces yields $\u_k \in H^1(\Omega)$ and
$\nabla \u_k=\nabla \u \cdot\chi_{[-1+\frac1k,1-\frac1k]} (\u)$, for any $k>0$.
Since $a(\cdot)>0$, we see from \eqref{NEUMANNSING} that $u$ satisfies
$$
-\Delta A(\u)+\frac{F'(\u)}{\sqrt{a(\u)}}=\frac{f}{\sqrt{a(\u)}} \quad \text{a.e in } \Omega,
$$
where $f \in L^p(\Omega)$ by assumption.
Let $2\leq p < \infty$. Choosing as test function $|F'(\u_k)|^{p-2}F'(\u_k)$, which belongs to $H^1(\Omega)$ for any $k$, we have
\begin{align*}
\int_\Omega -\Delta A(\u)  |F'(\u_k)|^{p-2}F'(\u_k) \, \mathrm{d}x
&+\int_\Omega
\frac{F'(\u)}{\sqrt{a(\u)}} |F'(\u_k)|^{p-2}F'(\u_k) \, \mathrm{d}x
\\
&=
\int_\Omega \frac{f}{\sqrt{a(\u)}}  |F'(\u_k)|^{p-2}F'(\u_k) \, \mathrm{d}x.
\end{align*}
Since $F''(\u_k)$ is well-defined and positive, and $\partial_\n A(\u)=0$ on $\partial \Omega$, we learn that
\begin{align*}
&\int_\Omega -\Delta A(\u)  |F'(\u_k)|^{p-2}F'(\u_k) \, \mathrm{d}x
\\
&\quad =(p-1) \int_\Omega |F'(\u_k)|^{p-2}F''(\u_k)\nabla \u\cdot
\chi_{[-1+\frac1k,1-\frac1k]}(\u) \cdot \nabla A(\u) \, \mathrm{d}x \geq 0.
\end{align*}
Then, recalling that $a_m \le a(s) \le a_M$ for any $s \in [-1,1]$, we obtain
$$
\frac{1}{\sqrt{a_M}} \int_\Omega F'(\u)  |F'(\u_k)|^{p-2}F'(\u_k) \, \mathrm{d}x
\leq \frac{1}{\sqrt{a_m}} \int_\Omega |f|  |F'(\u_k)|^{p-1} \, \mathrm{d}x.
$$
By the H\"{o}lder inequality
$$
\int_\Omega |f|  |F'(\u_k)|^{p-1} \, \mathrm{d}x
\leq \|F'(\u_k)\|_{L^{p}(\Omega)}^{p-1} \|f\|_{L^{p}(\Omega)}.
$$
On the other side, by exploiting the monotonicity of $F'$ and the fact that $F'(s)s \geq 0$, we have
$$
\|F'(\u_k)\|^p_{L^{p}(\Omega)}\leq \int_\Omega F'(\u)  |F'(\u_k)|^{p-2}F'(\u_k) \, \mathrm{d}x.
$$
Therefore, we arrive at
$$
\| F'(\u_k)\|_{L^p(\Omega)}\leq \frac{\sqrt{a_M}}{\sqrt{a_m}}\| f\|_{L^p(\Omega)}.
$$
Passing to the limit as $k\to \infty$, by Fatou's lemma we end up with
$$
\| F'(\u)\|_{L^p(\Omega)}\leq \frac{\sqrt{a_M}}{\sqrt{a_m}} \| f\|_{L^p(\Omega)}.
$$
Finally, in the case when $f\in L^{\infty}(\Omega)$, we infer from the above estimate that
$
\| F'(\u)\|_{L^p(\Omega)}\leq C,
$
where $C$ is independent of $p$. Thus, passing to the limit as $p \to \infty$, the claim follows.
\end{proof}

We are now in the position to show the validity of the following crucial result.
\begin{proposition}\label{esci}
For each $\u\in \mathcal{S}_m$, there exists a constant $\delta_u \in (0,1)$ such that
 \begin{align}
 |\u(x)|\leq 1-\delta_u, \quad \forall\, x\in \overline{\Omega}.\label{sepsta}
 \end{align}
 Moreover, $\u\in H^3(\Omega)$.
 \end{proposition}

\begin{proof}
We notice that $\u$ solves the elliptic problem \eqref{NEUMANNSING} with $f$ given by
$$f=\overline{\frac{a'(\u)}{2} |\nabla \u|^2+\Psi'(\u)}+\theta_0 \u\in L^\infty(\Omega).$$
Therefore,  we learn by Lemma \ref{ell2} that $F'(\u)\in L^\infty(\Omega)$. Since $F'$ diverges at $\pm 1$ and
$\u\in C(\overline{\Omega})$, we immediately deduce the existence of $\delta_u>0$ such that
$|\u(x)|\leq 1-\delta_u$, for all $x\in \overline{\Omega}$.

\smallskip
In order to establish further regularity, we rewrite \eqref{NEUMANNSING} in the following form
\begin{equation}
\label{Aphi-ELL-0}
\begin{cases}
-\Delta A(\u)+ A(\u) = g \quad &\text{in }\Omega,\\
\partial_{\n} A(\u)=0\quad &\text{on }\partial\Omega,
\end{cases}
\end{equation}
where
$$
g=\frac{1}{\sqrt{a(\u)}}(f-F'(\u))+A(\u)=\frac{1}{\sqrt{a(\u)}}\left(-\Psi'(\u)+\overline{\frac{a'(\u)}{2} |\nabla \u|^2+\Psi'(\u)}\right)+A(\u).
$$
Since $\u$ is separated from $\pm 1$ by \eqref{sepsta}, we learn that $\Psi'(\u)$ and $\Psi''(\u)$ are in $L^\infty(\Omega)$. Therefore, since $\u\in H^2(\Omega)$, we easily see that
$g\in L^2(\Omega)$. Besides, since
$$
\partial_{x_i}g = \frac12
\frac{a'(\u)}{a^\frac32(\u)} \left( \Psi'(\u)-\overline{\frac{a'(\u)}{2} |\nabla \u|^2+\Psi'(\u)}\right) \partial_{x_i} \u
-
\frac{1}{\sqrt{a(\u)}} \Psi''(\u) \partial_{x_i} \u + \sqrt{a(\u)}\partial_{x_i} \u,
$$
we have
\begin{align*}
\|\partial_{x_i}g\|_{L^2(\Omega)} &\leq \frac12
\left\|\frac{a'(\u)}{a^\frac32(\u)}\right\|_{L^\infty(\Omega)} 
\left(  2\|\Psi'(\u)\|_{L^\infty(\Omega)} + \frac{1}{2 |\Omega|} 
\| a'(u)\|_{L^\infty(\Omega)} \| \nabla u\|_{L^2(\Omega)}^2 \right)
\| \partial_{x_i} \u\|_{L^2(\Omega)}\\
&\quad +
\left\|\frac{1}{\sqrt{a(\u)}}\right\|_{L^\infty(\Omega)}\| \Psi''(\u)\|_{L^\infty(\Omega)}\| \partial_{x_i} \u\|_{L^2(\Omega)} + \|\sqrt{a(\u)}\|_{L^\infty(\Omega)}\|\partial_{x_i} \u\|_{L^2(\Omega)},
\end{align*}
which entails that $g\in H^1(\Omega)$. Hence, the classical elliptic regularity theory yields that $A(\u)\in H^3(\Omega)$. Next,
on account of \eqref{phi-Aphi-rel}, $u$ solves
\begin{equation}
\label{Aphi-ELL-0-1}
\begin{cases}
-\Delta \u+ \u = h \quad &\text{in }\Omega,\\
\partial_{\n} \u=0\quad &\text{on }\partial\Omega,
\end{cases}
\end{equation}
where
$$
h=- \frac{1}{\sqrt{a(u)}} \Delta A(u) + \frac{a'(u)}{2 a^2(u)} | \nabla A(u)|^2 +u.
$$
We notice that
\begin{align}
\partial_{x_i} h
&= \frac12 \frac{a'(u)}{a^\frac32 (u)} \Delta A(u) \partial_{x_i} u
- \frac{1}{\sqrt{a(u)}} \partial_{x_i} \Delta A(u)
+\frac12 \left( \frac{a''(u)a(u)- 2 (a'(u))^2}{a^3(u)} \right)  \partial_{x_i} u   | \nabla A(u)|^2
\notag
\\
&\quad + \frac{a'(u)}{a^2(u)} \nabla A(u) \cdot \partial_{x_i} \nabla A(u)
+ \partial_{x_i} u.
\label{h1}
\end{align}
Therefore, 
\begin{align}
\|\partial_{x_i} h\|_{L^2(\Omega)}
&\leq \left\|\frac{a'(u)}{2a^\frac32 (u)}\right\|_{L^\infty(\Omega)}\| \Delta A(u)\|_{L^4(\Omega)}\| \partial_{x_i} u\|_{L^4(\Omega)}
+ \left\|\frac{1}{\sqrt{a(u)}}\right\|_{L^\infty(\Omega)}\| \partial_{x_i}\Delta A(u)\|_{L^2(\Omega)}
\notag
\\
&\quad+\left\|\frac{a''(u)a(u)- 2 (a'(u))^2}{2a^3(u)} \right\|_{L^\infty(\Omega)}\|  \partial_{x_i} u\|_{L^6(\Omega)}\| \nabla A(u)\|_{L^6(\Omega)}^2
\notag
\\
&\quad  +\left\|\frac{a'(u)}{a^2(u)}\right\|_{L^\infty(\Omega)}\| \nabla A(u)\|_{L^4(\Omega)} \|\partial_{x_i} \nabla A(u)\|_{L^4(\Omega)}
+ \|\partial_{x_i} u\|_{L^2(\Omega)}
\notag
\\
& \leq
\frac{C}{2a_m^\frac32}\|A(u)\|_{H^3(\Omega)}\| u\|_{H^2(\Omega)}
+ \frac{1}{\sqrt{a_m}}\|A(u)\|_{H^3(\Omega)}
\notag
\\
&\quad+\frac{C}{2a_m^3}\|u\|_{H^2(\Omega)}\|A(u)\|_{H^2(\Omega)}^2
  +\frac{C}{a_m^2}\|A(u)\|_{H^2(\Omega)}\|A(u)\|_{H^3(\Omega)}
+ \|u\|_{H^1(\Omega)}.
\label{h2}
\end{align}
Since $u \in H^2(\Omega)$ and $A(u) \in H^3(\Omega)$, we obtain that $h \in H^1(\Omega)$, and thus $\u\in H^3(\Omega)$ as claimed.
\end{proof}

\section{A gradient Lojasiewicz-Simon inequality with nonlinear diffusion}
\label{LSineq}
In this section we demonstrate a gradient inequality of Lojasiewicz-Simon
type for the energy \eqref{free-energy}, namely we prove Theorem \ref{LSg-intro}. To this aim, we embed the problem in a more general setting as follows.
\subsection{A general result}
We set $H=L^2(\Omega)$ and we introduce the Hilbert space
$$
V=\{u\in H^2(\Omega):\partial_\n u=0\text{ on }\partial\Omega\}\subset H,
$$
with continuous and dense embedding into $H$.
We consider the class of functionals $\E:V\to \R$
$$
\E(u)=\int_\Omega \frac{\widetilde{a}(u)}{2}|\nabla u|^2+\widetilde{\Psi}(u) \, \d x,
$$
where $\widetilde{a}$ and $\widetilde{\Psi}$ are two arbitrary regular functions such that
\begin{equation}
\label{APSI-ass}
\widetilde{a},\widetilde{\Psi}\in C^3(\R) \text{ are analytic on }
\left(-1+\frac{\widetilde{\delta}}{2},1-\frac{\widetilde{\delta}}{2} \right) \!, \text{ for some $\widetilde{\delta}>0$.}
\end{equation}
We preliminarily observe that, for any $u \in V$,
$$
\langle \E'(u),v \rangle_{V' \times V}=
\int_{\Omega} \widetilde{a}(u)\nabla u\cdot \nabla v+\frac{\widetilde{a}'(u)}{2} |\nabla u|^2v +\widetilde{\Psi}'(u) v\, \d x, \quad \forall \, v \in V.
$$
Since $u \in V$, it is easily seen that
$$
\langle \E'(u),v \rangle_{V' \times V}=
\int_{\Omega} \left(-\div(\widetilde{a}(u)\nabla u)+\frac{\widetilde{a}'(u)}{2} |\nabla u|^2+\widetilde{\Psi}'(u)\right)v \, \d x, \quad \forall \, v \in V.
$$
Hence, for any $u\in V$, $\E$ has an $H$-gradient (see, for instance, \cite[Definition 5.1]{RUPP}) defined by
\begin{equation}
\label{GRAD-E}
\nabla \E(u):=-\div(\widetilde{a}(u)\nabla u)+\frac{\widetilde{a}'(u)}{2} |\nabla u|^2+\widetilde{\Psi}'(u) \in H.
\end{equation}
Finally, we introduce the set
$$
U=\left\{u\in V:-1+\widetilde{\delta}<u(x)<1-\widetilde{\delta},\quad \forall x\in \overline{\Omega}\right\}
$$
which is an open subset of $V$, due to the continuous embedding $V\hookrightarrow C(\overline{\Omega})$.

\begin{theorem}
\label{LS}
Let $d=2,3$. Assume that \eqref{APSI-ass} hold.
Let $u^\star \in U \cap H^3(\Omega)$ with $\overline{u^\star}=m$ with $m \in (-1,1)$ such that
\begin{equation}
\label{LS-cp}
( \nabla \E (u^\star), v ) =0, \quad \forall \, v \in H \text{ with } \overline{v}=0.
\end{equation}
Then, there exist $\theta \in \left(0,\frac12\right]$, $C>0$, $\beta>0$ such that
\begin{equation}
\label{LS-formula}
| \E(u) - \E(u^\star)|^{1-\theta}
\leq C \norm{\nabla \E(u) -\overline{\nabla \E(u)}}_{H}
\end{equation}
for all $u\in U$ such that $\overline{u}=m$ and
$\| u-u^\star\|_{V}\leq \beta$.
\end{theorem}

The strategy of the proof consists in interpreting $u^\star$ as a constrained critical point of $\E$ on a manifold $\mathcal{M}$ defined as
the nodal set of a suitable map $\mathcal{G}:V\to \R$, see \eqref{manifold} below, accounting for the conservation of mass of the model. Accordingly,
hinging upon the abstract result \cite[Corollary 5.2]{RUPP} (see also \cite{CHILL}), we will obtain a refined Lojasiewicz–Simon gradient inequality, that will yield the claimed  inequality \eqref{LS-formula}.

\subsection{Proof of Theorem \ref{LS}.}
We proceed by recasting our problem in the framework of \cite[Corollary 5.2]{RUPP} throughout the subsequent three steps.

\smallskip {\bf Step 1. Analyticity of  $\E$ e $\nabla \E$ on $U$.}
We will make use of the following result on the analyticity of superposition functions.

\begin{lemma}
\label{ana-psi}
If $f:(-1+ \frac{\widetilde{\delta}}{2},1-\frac{\widetilde{\delta}}{2})\to\R$ is analytic, then the map
$$\mathcal{F}:U\to  C(\overline{\Omega}),\qquad u\mapsto f(u)$$ is analytic.
\end{lemma}
\begin{proof}
We argue as in the proof of \cite[Corollary 4.6]{CHILL}. We preliminarily observe that $f$ is uniformly analytic in $J=(-1+\widetilde{\delta},1-\widetilde{\delta})$. This means that there exists $0< \rho<\frac{\widetilde{\delta}}{2}$ such that the power series expansion of $f$ near any given $\widehat{s}\in J$
\begin{equation}
\label{ef}
f(s)=\sum_n \alpha_n(\widehat{s})(s-\widehat{s})^n,\qquad \alpha_n(\widehat{s})\in \R,
\end{equation}
converges absolutely for every $s\in (-1,1):|s-\widehat{s}|<\rho$.
Let $\mathcal{U}:=\{u\in C(\overline{\Omega}): u(x)\in J,\,\forall \, x\in\overline{\Omega}\}$,
noticing that $\mathcal U$ is an open set of $C(\overline{\Omega})$, and
fix $\widehat{u}\in\mathcal{U}$. Defining the ball $B_\rho(\widehat{u})=\{v\in C(\overline{\Omega}):\|v-\widehat{u}\|_{L^\infty(\Omega)}<\rho\}$, we learn from \eqref{ef} that
$$
f(u)=\sum_n a_n(u-\widehat{u})^n\quad\text{in}\quad C(\overline{\Omega}),
$$
for any $u\in B_\rho(\widehat{u})$, where the multilinear forms $a_n:[C(\overline{\Omega})]^n\to C(\overline{\Omega})$ and $a_n(u-\widehat{u})^n:=a_n(u-\widehat{u},\dots,u-\widehat{u})$  by the rule $x\mapsto \alpha_n(\widehat{u}(x))(u(x)-\widehat{u}(x))^n$, for any $x\in \overline{\Omega}$.
Accordingly, since $\widehat{u}$ is arbitrary in $\mathcal{U}$, the map
$$
 \mathcal{U}\to C(\overline{\Omega}), \qquad  u\mapsto f(u)
 $$ is analytic.
In order to prove that $\mathcal{F}$ is analytic as a map from $U$ to $C(\overline{\Omega})$, we rely on the continuous embedding $U\hookrightarrow C(\overline{\Omega})$. Since the inclusion $i:U \to C(\overline{\Omega})$ is analytic, $i(U)\subset \mathcal{U}$ and the composition of analytic maps is analytic (see e.g. \cite[Theorem 2.2]{RUPP}), the conclusion immediately follows.
\end{proof}

We are now ready to prove the required analyticity of $\E$ and $\nabla \E$.

\begin{lemma}
\label{E-analytic}
The functional $\E$ is analytic on $U$.
\end{lemma}

\begin{proof}
The following properties hold:
 \begin{enumerate}
 \item[(M1)] the pointwise multiplications $C(\overline{\Omega})\times L^2(\Omega)\to L^2(\Omega): (f,g) \mapsto fg$ is analytic;
 \smallskip
 \item[(M2)] The map $L^2(\Omega)\to\R : f\mapsto \int_\Omega f \, \d x$ is analytic;
  \smallskip
 \item[(M3)] The map $H^2(\Omega)\to L^2(\Omega): u \mapsto \partial_i\partial_j u$, for any $i, j$, is analytic;
  \smallskip
\item[(M4)] The map $H^2(\Omega)\to L^2(\Omega): u\mapsto |\nabla u|^2$ is analytic.
\end{enumerate}
Indeed, (M1)-(M3) follow as the maps are (multi)linear and bounded (see e.g. \cite[Example 2.3]{RUPP}). As to (M4), the map is analytic since it is the diagonal of the bounded bilinear map $a: H^2(\Omega)\times H^2(\Omega)\to L^2(\Omega)$ such that $a(u,v)=\nabla u\cdot \nabla v$.

Now, by \eqref{APSI-ass} and Lemma \ref{ana-psi}, the map $U\to C(\overline{\Omega}): u\mapsto \widetilde{a}(u)$ is analytic.
Since the pointwise multiplication $C(\overline{\Omega})\times L^2(\Omega)\to L^2(\Omega)$ is analytic (see (M1)), the map $U\to L^2(\Omega):u\mapsto \frac{\widetilde{a}(u)}{2}|\nabla u|^2$
is analytic in light of (M4). Then, as the composition of analytic maps is analytic, we deduce from (M2) that
$$
 U\to \R : \ u\mapsto \int_\Omega \frac{\widetilde{a}(u)}{2}|\nabla u|^2 \, \d x \text{ is analytic.}
$$
Analogously, the analyticity of the map $U\to \R: \ u\mapsto \int_\Omega \widetilde{\Psi}(u)\, \d x$ easily follows. Hence, we obtain the desired claim.
\end{proof}

\begin{lemma} The map $\nabla \E:U\to L^2(\Omega)$
is analytic.
\end{lemma}
\begin{proof} By standard computations and \eqref{GRAD-E}, we have
$$
\nabla \E(u)=-\widetilde{a}(u)\Delta u-\frac{\widetilde{a}'(u)}{2} |\nabla u|^2+\widetilde{\Psi}'(u)\in H.
$$
Hence, by arguing as in Lemma \ref{E-analytic} for the analyticity of $\E$, we immediately learn that the map
$U\to L^2(\Omega): u\mapsto -\frac{\widetilde{a}'(u)}{2} |\nabla u|^2+\widetilde{\Psi}'(u)$
is analytic.
Finally,
the map $U \to L^2(\Omega): u\mapsto \widetilde{a}(u)\Delta u$ is analytic by (M1) and (M3).
\end{proof}

\smallskip

{\bf Step 2. $\E''(u^\star)$ is a Fredholm operator of index zero on $V$.}
First of all, for $u \in U$, the second derivative $\E''(u)= (\nabla \E )' (u)$ is such that
$$
\langle \E''(u)v,z\rangle_{V' \times V} =
\int_\Omega \frac{\widetilde{a}''(u)}{2}|\nabla u|^2vz+\widetilde{a}'(u)v\nabla u\cdot \nabla z+\widetilde{a}'(u)z\nabla u\cdot \nabla v+\widetilde{a}(u)\nabla v\cdot \nabla z+\widetilde{\Psi}''(u)vz \, \d x,
$$
for any $v, z \in V$. We notice that
\begin{align*}
&\norm{ \frac{\widetilde{a}''(u)}{2} |\nabla u|^2 v + \widetilde{a}'(u) \nabla u \cdot \nabla v + \widetilde{\Psi}''(u) v }_{L^2(\Omega)}
\\
&\quad \leq
C \norm{\widetilde{a}''(u)}_{L^\infty(\Omega)} \| \nabla u\|_{L^4(\Omega)}^2
\| v\|_{L^\infty(\Omega)} + \norm{\widetilde{a}'(u)}_{L^\infty(\Omega)} \| \nabla u\|_{L^4(\Omega)} \| \nabla v\|_{L^4(\Omega)}
\\
& \qquad
+ \| \widetilde{\Psi}''(u)\|_{L^\infty(\Omega)} \| v\|_{L^2(\Omega)},
\end{align*}
and
\begin{align*}
&\norm{ -\div( \widetilde{a}(u) \nabla v) -\div(\widetilde{a}'(u) v \nabla u) }_{L^2(\Omega)}
\\
&\quad  \leq
\| \widetilde{a}(u) \|_{L^\infty(\Omega)} \| \Delta v\|_{L^2(\Omega)} +
2 \| \widetilde{a}'(u)\|_{L^\infty(\Omega)} \| \nabla u\|_{L^4(\Omega)} \| \nabla v\|_{L^4(\Omega)}
\\
&\qquad
+ \| \widetilde{a}''(u) \|_{L^\infty(\Omega)} \| \nabla u\|_{L^4(\Omega)}^2 \| v\|_{L^\infty(\Omega)}
+ \| \widetilde{a}'(u)\|_{L^\infty(\Omega)} \| v\|_{L^\infty(\Omega)} \| \Delta u\|_{L^2(\Omega)}.
\end{align*}
Here, the right-hand sides are both finite since $u\in U$, $v \in V$ and assumption \eqref{APSI-ass} is in place.
Therefore, an integration by parts leads to
$$
\E''(u)v =-\div(\widetilde{a}(u)\nabla v)-\div (\widetilde{a}'(u)v\nabla u)+\frac{\widetilde{a}''(u)}{2}|\nabla u|^2v+\widetilde{a}'(u)\nabla u\cdot \nabla v+\widetilde{\Psi}''(u)v\in H,
$$
for any $v \in V$.

\medskip Let us now take $\widehat{u} \in U\cap H^3(\Omega)$. For any $v \in V$, we write
$$
\E''(\widehat{u})v =Av+Kv,
$$
where
$$
A v = -\div(\widetilde{a}(\widehat{u})\nabla v) + v,
$$
and
$$
K v = -\div (\widetilde{a}'(\widehat{u})v\nabla \widehat{u})+\frac{\widetilde{a}''(\widehat{u})}{2}|\nabla \widehat{u}|^2v+\widetilde{a}'(\widehat{u})\nabla \widehat{u}\cdot \nabla v - v + \widetilde{\Psi}''(\widehat{u}) v.
$$
First of all, we claim that $A \in \mathcal{L}(V,H)$ is an isomorphism, which implies that $A$ is a Fredholm operator of index zero. To this end, we first observe that
\begin{align*}
\| -\div (\widetilde{a}(\widehat{u}) \nabla v) + v \|_{L^2(\Omega)} &\leq \| \widetilde{a}(\widehat{u}) \|_{L^\infty(\Omega)} \| \Delta v\|_{L^2(\Omega)} + \| \widetilde{a}'(\widehat{u})\|_{L^\infty(\Omega)} \| \nabla \widehat{u}\|_{L^4(\Omega)} \| \nabla v\|_{L^4(\Omega)}  + \| v\|_{L^2(\Omega)}
\\
&\leq C \| v\|_{H^2(\Omega)}.
\end{align*}
Besides, since $\widehat{u} \in W^{1,r}(\Omega)$, with $r>3$, we immediately infer from Proposition \ref{NP-nc} that $A$ is an isomorphism from $V$ onto $H$ as desired.

Next, we show that $K$ is a compact operator from $V$ to $H$.
We notice that
$$
K v=   - \widetilde{a}'(\widehat{u}) \Delta \widehat{u} \, v - \frac{\widetilde{a}''(\widehat{u})}{2}|\nabla \widehat{u}|^2v -v + \widetilde{\Psi}''(\widehat{u})v,
$$
and
\begin{align*}
\| K v\|_{L^2(\Omega)}
&\leq \| \widetilde{a}'(\widehat{u}) \|_{L^\infty(\Omega)} \| \Delta \widehat{u}\|_{L^2(\Omega)} \| v\|_{L^\infty(\Omega)}
+ C \| \widetilde{a}''(\widehat{u}) \|_{L^\infty(\Omega)} \| \nabla \widehat{u}\|_{L^4(\Omega)}^2 \| v\|_{L^\infty(\Omega)}
\\
&\quad
+ \| v\|_{L^2(\Omega)} + \| \widetilde{\Psi}''(\widehat{u}) \|_{L^\infty(\Omega)} \| v \|_{L^2(\Omega)}
\\
&
\leq C \| v\|_{H^2(\Omega)},
\end{align*}
where $C$ depends on $\widetilde{a}$, $\| \widehat{u}\|_{H^2(\Omega)}$ and $\widetilde{\delta}$.
Furthermore, since
\begin{align*}
\partial_{x_i} Kv &= - \widetilde{a}''(\widehat{u}) \partial_{x_i} \widehat{u} \,\Delta \widehat{u} \, v
-\widetilde{a}'(\widehat{u})\partial_{x_i} \Delta \widehat{u} \, v
-\widetilde{a}'(\widehat{u}) \Delta \widehat{u} \, \partial_{x_i} v
\\
& \quad - \frac{\widetilde{a}'''(\widehat{u})}{2} \partial_{x_i} \widehat{u} |\nabla \widehat{u}|^2 \, v
- \widetilde{a}''(\widehat{u}) \partial_{x_i} \nabla \widehat{u} \cdot \nabla \widehat{u} \, v
- \frac{\widetilde{a}''(\widehat{u})}{2}|\nabla \widehat{u}|^2 \, \partial_{x_i} v
\\
& \quad -\partial_{x_i} v + \widetilde{\Psi}'''(\widehat{u}) \partial_{x_i}\widehat{u}\, v
+ \widetilde{\Psi}''(\widehat{u}) \partial_{x_i} v, 
\end{align*}
we find that
\begin{align*}
\| \nabla K v\|_{L^2(\Omega)}
&
\leq \norm{\widetilde{a}''(\widehat{u})}_{L^\infty(\Omega)}
\norm{\widehat{u}}_{W^{1,\infty}(\Omega)}
\norm{\Delta \widehat{u}}_{L^2(\Omega)}
\| v\|_{L^\infty(\Omega)}
+ \norm{\widetilde{a}'(\widehat{u})}_{L^\infty(\Omega)}
\| \widehat{u}\|_{H^3(\Omega)}
\| v\|_{L^\infty(\Omega)}
\\
& \quad
+  \norm{\widetilde{a}'(\widehat{u})}_{L^\infty(\Omega)}
\| \Delta \widehat{u}\|_{L^6(\Omega)}
\| v\|_{W^{1,3}(\Omega)}
+ C \norm{\widetilde{a}'''(\widehat{u})}_{L^\infty(\Omega)}
\| \widehat{u}\|_{W^{1,6}(\Omega)}^3
 \| v\|_{L^\infty(\Omega)}
 \\
 & \quad
 + \norm{\widetilde{a}''(\widehat{u})}_{L^\infty(\Omega)}
 \| \widehat{u}\|_{W^{2,3}(\Omega)}
 \| \widehat{u}\|_{W^{1,6}(\Omega)}
 \| v\|_{L^\infty(\Omega)}
 \\
 & \quad
 + C \norm{\widetilde{a}''(\widehat{u})}_{L^\infty(\Omega)}
 \| \widehat{u}\|_{W^{1,6}(\Omega)}^2
 \| v\|_{W^{1,6}(\Omega)}
 \\
 & \quad
 + \| v\|_{H^1(\Omega)} + \| \widetilde{\Psi}'''(\widehat{u}) \|_{L^\infty(\Omega)}
 \| \widehat{u}\|_{W^{1,6}(\Omega)} \| v \|_{L^3(\Omega)}
+ \| \widetilde{\Psi}''(\widehat{u}) \|_{L^\infty(\Omega)}
\|  v\|_{H^1(\Omega)}
 \\
 & \leq C \| v\|_{H^2(\Omega)},
\end{align*}
where the constant $C$ depends on $\widetilde{a}$, $\| \widehat{u}\|_{H^3(\Omega)}$ and $\widetilde{\delta}$.
Hence, we deduce that $K$ is a bounded operator from $V$ to $H^1(\Omega)$, meaning that $K$ is a compact operator from $V$ to $L^2(\Omega)$, in light of the compact embedding $H^1(\Omega) \hookrightarrow L^2(\Omega)$.
Finally, by \cite[Proposition 2.5]{RUPP}, we conclude that $\E''(\widehat{u}):V\to H$ is a Fredholm operator of index zero. In particular, we can take $\widehat{u}=u^\star$.

\medskip
{\bf Step 3. The map $\mathcal{G}$.}
We define the map $\mathcal{G}:V\to \R$ by
$$\mathcal{G}(u)=\int_\Omega u \,\d x - m|\Omega|,$$
and the corresponding nodal set
\begin{equation}
\label{manifold}
\mathcal{M}=\big\{u\in U: \mathcal{G}(u)=0\big\},
\end{equation}
which contains all $u\in U$ such that $\overline{u}=m$.
As in \cite[Lemma 7.2]{RUPP}, it is readily seen that $\mathcal{G}:V\to \R$ is analytic. Indeed,
by the embedding $V\hookrightarrow C(\overline{\Omega})$ and the fact that the inclusion $i: V\to C(\overline{\Omega})$ is analytic, the conclusion follows from (M2). Besides, for all $u\in U$, we have
$$
\langle \mathcal{G}'(u), v\rangle_{V' \times V}=\int_\Omega 1\cdot v \, \d x,\quad \forall \, v\in V.
$$
In particular, $\mathcal{G}$ has an $H$-gradient given by $\nabla\mathcal{G}(u):=1\in H$ for all $u\in V$.
It immediately turns out that the (constant) map $u\mapsto \nabla \mathcal{G}(u):U\to H$ is analytic, and that $\mathcal{G}'(u):V\to\R$ is surjective.
Furthermore, for any $u\in V$, we have that $(\nabla\mathcal{G})'(u)=0$, implying that $(\nabla\mathcal{G})'(u):V\to H$ is compact.
\smallskip

We are now in the position to apply \cite[Theorem 3.1]{RUPP} at \emph{any $u\in \mathcal{M}$}, yielding that the set $\mathcal{M}$ is an analytic submanifold of $V$ of codimension $1$.
Furthermore, it follows from \cite[Proposition 3.3]{RUPP} that, for every $u\in \mathcal{M}$, the tangent space of $\mathcal{M}$ at $u$ is given by
\begin{equation}
\label{Tu}
\mathcal{T}_{u}\mathcal{M}=\ker(\mathcal{G}'(u))=\left\{v\in V:\int_\Omega v \, \d x=0\right\}.
\end{equation}
Finally, we consider $u^\star\in U\cap H^3(\Omega)$ with $\overline{u^\star}=m$. In particular, $u^\star\in \mathcal{M}$ and, since \eqref{LS-cp} holds true, it turns out from \eqref{Tu} that $u^\star$ is a constrained critical 
point\footnote{We recall that $u\in\mathcal{M}$ is a constrained critical point of $\mathcal{E}$ (or a critical point of $\mathcal{E}|_\mathcal{M}$) if
$(\nabla\mathcal{E}(u),v)=0$ for all $v\in \mathcal T_{u}\mathcal{M}$.}
of $\mathcal{E}$ on $\mathcal{M}$.
\medskip

{\bf Step 4. Conclusion of the proof of Theorem \ref{LS}.}
 Collecting the properties proved in the previous steps, all the required assumptions in \cite[Corollary 5.2]{RUPP} are in place for our energy functional $\mathcal{E}$, the map $\mathcal{G}$ defined in Step 3,
and the constrained critical point $u^\star$. Therefore, there exist $\theta \in \left(0,\frac12\right]$, $C>0$, $\beta>0$ such that
\begin{equation}
\label{LS-formula1}
| \E(u) - \E(u^\star)|^{1-\theta}
\leq C \norm{P(u)\nabla \E(u)}_{H}
\end{equation}
for all $u\in \mathcal{M}$ and
$\| u-u^\star\|_{V}\leq \beta$, where $P(u):H\to H$ is the orthogonal projection on the closure in $H$ of the tangent space $\mathcal{T}_{u}\mathcal{M}$.
Since, for every $u\in \mathcal{M}$, the projection $P(u)$ is given by
$$
P(u)y=y-\overline{y},\quad \forall \, y\in H,
$$
it then follows that \eqref{LS-formula1} is exactly the claimed inequality.
$\hfill \Box$

\bigskip

We are now ready to prove the Lojasiewicz–Simon inequality	 announced in Theorem \ref{LSg-intro}.

\subsection{Proof of Theorem \ref{LSg-intro}}
Let $\psi\in \mathcal{S}_m$. By Proposition \ref{esci}, we know that $\psi\in H^3(\Omega)$ and there exists $\delta_{\psi}\in \left(0,\frac12\right)$ such that
\begin{equation}
\label{sepsi}
\|\psi\|_{C(\overline{\Omega})} \leq 1-2\delta_\psi.
\end{equation}
Recalling the Sobolev embedding $H^2(\Omega)\hookrightarrow C(\overline{\Omega})$, we have
\begin{equation*}
\|u\|_{C(\overline{\Omega})}\leq C_S \|u\|_{H^2(\Omega)}
\end{equation*}
for some $C_S>0$. Accordingly, we set
 \begin{align*}
\omega_0:=\frac{\delta_\psi}{C_S}\label{vicino}
\end{align*}
Then, for any $u\in H^2(\Omega)$ such that
\begin{equation}
\label{baseee}\|u-\psi\|_{H^2(\Omega)}<\omega_0,
\end{equation}
it follows from \eqref{sepsi} that
\begin{align*}
\|u\|_{C(\overline{\Omega})}
\leq \|\psi\|_{C(\overline{\Omega})}
+\|u-\psi\|_{C(\overline{\Omega})}
< 1-\delta_\psi.
\end{align*}
Now, having in mind Theorem \ref{LS}, we set
$$
V=\{u\in H^2(\Omega):\partial_\n u=0\text{ on }\partial\Omega\},\qquad H=L^2(\Omega),$$
and
$$
U=\{u\in V:-1+\delta_\psi<u(x)<1-\delta_\psi,\quad \forall \, x\in \overline{\Omega}\},
$$
which is an open subset of $V$. Note that $U$ contains, in particular, all functions $u\in H^2(\Omega)$ satisfying  $\partial_\n u=0$ on $\partial \Omega$ and \eqref{baseee}.
Then, we define $\widetilde{\Psi}: \mathbb{R} \to \mathbb{R}$ such that $\widetilde{\Psi} \in C^3(\mathbb{R})$, $\widetilde{\Psi}|_{[-1+\delta_\psi,1-\delta_\psi]}=\Psi$ and $|\widetilde{\Psi}^{(j)}|$ are bounded for $j=1,2,3$, and similarly we extend $a$ to $\widetilde{a}\in C^3(\R)$. Next,
we introduce the functional $\mathcal{E}:V\to\R$
\begin{equation}
\label{ene-LS}
\mathcal{E}(u):= \int_{\Omega} \frac{\widetilde{a}(u)}{2}|\nabla u|^2 +\widetilde{\Psi}(u) \, \d x.
\end{equation}
Since $\widetilde{\Psi}$ and $\widetilde{a}$, are analytic in $(-1+\frac{\delta_\psi}{2},1-\frac{\delta_\psi}{2})$, the functional $\mathcal{E}$ is
analytic (see Lemma \ref{E-analytic}). Besides, $\mathcal{E}$ coincides with $E$ in $U$ and, for every $u\in U$, we have
\begin{equation}
\label{gra-LS}
\nabla \mathcal{E}(u)=\nabla{E}(u)=-\div(a(u)\nabla u)+\frac{a'(u)}{2} |\nabla u|^2+\Psi'(u)\in H.
\end{equation}
Since $\psi\in \mathcal{S}_m$, it is apparent that
$$\langle \nabla \E (\psi), v \rangle =0, \quad \forall \, v \in H \text{ with } \overline{v}=0,$$
namely $u^\star=\psi$ complies with assumption \eqref{LS-cp}.
Therefore, we can apply Theorem \ref{LS} with $\mathcal{E}$ as in \eqref{ene-LS}, which gives the following result:
 \emph{there exist $\theta \in \left(0,\frac12\right]$, $C_L>0$, $\beta>0$ such that
\begin{equation}
\label{LSaa}
|\mathcal{E}(u)- \mathcal{E}(\psi)|^{1-\theta}
\leq C_L \norm{\nabla \mathcal{E}(u)-\overline{\nabla\mathcal{E}(u)} }_{L^2(\Omega)}\!
\end{equation}
for all $u\in U$ such that $\overline{u}=m$ and
$\| u-\psi\|_{H^2(\Omega)}\leq \beta$.
}

\noindent
Since, without loss of generality, we can assume $\beta<\omega_0$, in light of  \eqref{gra-LS}, we immediately recognize in \eqref{LSaa} the desired inequality \eqref{LS0}. $\hfill \Box$

\section{Improved estimates for weak solutions}
\label{S-WEAK}
\setcounter{equation}{0}

The existence of global weak solutions to system \eqref{CH1}-\eqref{CH2} subject to \eqref{CH-bc}-\eqref{CH-ic} has been proved in \cite{SP2013} under the assumption $b \equiv 1$ and \eqref{a-ndeg}, in both two and three dimensions. The proof can be easily extended to the case with non-degenerate concentration-depending mobility satisfying \eqref{m-ndeg}, see also the analysis in \cite{ADG2013}.
As a result, for any $\phi_0\in \mathcal{V}_m$ with $m\in (-1,1)$, there exists $\phi: \Omega \times [0,\infty)\to \R$ departing from $\phi_0$ at $t=0$, with the following regularity properties
\begin{align}
\label{S1}
&\phi \in  L^\infty(0,T; H^1(\Omega))\cap L^2(0,T;H^2(\Omega))
\\
\label{S2}
&\phi \in L^{\infty}(\Omega\times (0,T)):  |\phi(x,t)|<1  \ \text{a.e. in }\Omega\times (0,\infty),
\\
\label{S3}
&\partial_t \phi \in L^2(0,T; H_{(0)}^{-1}(\Omega)),
\\
\label{S4}
& \mu \in L^2(0,T;H^1(\Omega)), \quad F'(\phi)\in L^2( 0,T;L^2(\Omega)),
\end{align}
for any $T>0$, which satisfies the system in the weak sense as in \eqref{e2}-\eqref{e3}. Furthermore, recalling the definition of the energy in \eqref{free-energy},
the following energy inequality holds true
\begin{equation}
\label{EE}
E(\phi(t))+ \int_0^t \left\| \sqrt{b(\phi(\tau))}\nabla \mu(\tau) \right\|_{L^2(\Omega)}^2 \, \d \tau \leq  E(\phi_0), \quad \forall \, t \geq 0.
\end{equation}
Besides, the total mass of the solution is preserved over time, namely
\begin{equation}
\label{cons-mass}
\overline{\phi(t)}=\frac{1}{|\Omega|}\int_\Omega \phi(t)\, \d x=\frac{1}{|\Omega|}\int_\Omega \phi_0\, \d x=\overline{\phi_0}=m,\quad \forall \, t\geq 0.
\end{equation}
\medskip

In this section, after recalling the basic properties of weak solutions, we establish some improved energy estimates, which in particular leads to the regularity class \eqref{WS1}-\eqref{WS4} when $d=2$ and \eqref{WS1-3}-\eqref{WS4-3} when $d=3$. In the sequel, the generic constant $C>0$ depends on $E(\phi_0)$ and on $\overline{\phi_0}$, but is independent of the specific initial datum.

\smallskip
\textbf{Energy bounds.}
Since $a(\cdot) \geq a_m$, $b(\cdot)\geq b_m$ in the interval $[-1,1]$, we  obtain from \eqref{EE} that
\begin{equation*}
 \int_\Omega |\nabla \phi(t)|^2 \, \d x
+ \int_0^t \left\| \nabla \mu(\tau) \right\|_{L^2(\Omega)}^2 \, \d \tau \leq C, \quad \forall \, t \geq 0.
\end{equation*}
Thanks to \eqref{cons-mass}, an application of the Poincar\'{e} inequality entails that
\begin{align}
\label{Linf-u-phi_1}
&\| \phi\|_{L^\infty( 0,\infty; H^1(\Omega))} \leq C.
\end{align}
Moreover, we also get
\begin{align}
\label{L2-u-mu_1}
&\int_0^\infty \| \nabla \mu (\tau) \|_{L^2(\Omega)}^2 \, \d \tau \leq C.
\end{align}
Besides, recalling that $\nabla A(\phi)= \sqrt{a(\phi)}\nabla \phi$ and observing that
\begin{equation}
\label{A-Linfty}
A(\phi) \in L^{\infty}(\Omega\times (0,\infty)): \quad |A(\phi(x,t))|\leq  \sqrt{a_M}  \quad  \text{a.e. in }\Omega\times (0,\infty),
\end{equation}
it follows from \eqref{Linf-u-phi_1} that
\begin{equation}
\label{A-H1}
\| A(\phi)\|_{L^\infty( 0,\infty; H^1(\Omega))} \leq C.
\end{equation}

\medskip
\textbf{Estimate of the time derivative.} We immediately infer from \eqref{CH1} that
\begin{equation}
\label{est-phit-H-1}
\| \partial_t \phi\|_{ H^{-1}_{(0)}(\Omega)} \leq C \| \nabla \mu\|_{L^2(\Omega)}.
\end{equation}
Hence, by \eqref{L2-u-mu_1},
\begin{equation}
\label{phit-0inf}
\int_0^\infty \| \partial_t \phi (\tau)\|_{ H^{-1}_{(0)}(\Omega)}^2 \, \d \tau
\leq C,
\end{equation}
which gives $\partial_t \phi \in L^2(0,\infty;  H^{-1}_{(0)}(\Omega))$.
\medskip
\smallskip

\textbf{$H^1$-estimate of the chemical potential.} First, observing that
$$
\Delta A(\phi)= \sqrt{a(\phi)}\Delta \phi+ \frac{a'(\phi)}{2 \sqrt{a(\phi)}}|\nabla \phi|^2 \quad \text{a.e. in }\Omega\times (0,\infty),
$$
we infer from \eqref{a-ndeg} and \eqref{inter-W14} that
\begin{align*}
\| \Delta A(\phi)\|_{L^2(\Omega)} 
\leq \sqrt{a_M} \| \Delta \phi\|_{L^2(\Omega)} 
+ \left\| \frac{a'(\phi)}{2 \sqrt{a(\phi)}} \right\|_{L^\infty(\Omega)} \| \phi\|_{L^\infty(\Omega)}\| \phi\|_{H^2(\Omega)}.
\end{align*}
Hence, we preliminarily deduce from \eqref{S1} and \eqref{A-H1} that $A(\phi) \in L^2(0,T;H^2(\Omega))$. 
\smallskip

Now, we proceed with the control of the total mass of $\mu$. It is convenient to observe that
\begin{equation}
\label{mu-equiv}
\mu= -\sqrt{a(\phi)}\Delta A(\phi) + \Psi'(\phi)\quad \text{a.e. in } \Omega \times (0,\infty).
\end{equation}
We recall a standard tool from \cite{MZ}: there exists a positive constant $\overline{C}$, depending only on $m$, such that
\begin{equation}
\label{MZ}
\int_\Omega| F'(\phi)|\, \d x\leq \overline{C}
\left|\int_\Omega F'(\phi)(\phi-\overline{\phi} )\, \d x\right|+\overline{C},
\end{equation}
where $\overline{C} \to +\infty$ as $|\overline{\phi_0}|\to 1$.
Multiplying \eqref{mu-equiv} by $\phi-\overline{\phi}$ and integrating over $\Omega$, we infer that
\begin{align*}
\int_{\Omega} \nabla A(\phi) \cdot \nabla \left(
\sqrt{a(\phi)} \left(\phi-\overline{\phi} \right) \right) \, \d x
&+
\int_{\Omega} F'(\phi) \left( \phi-\overline{\phi} \right) \, \d x
\\
&= \int_\Omega \left( \mu -\overline{\mu}\right) \phi \, \d x
+ \theta_0 \int_\Omega \phi (\phi-\overline{\phi}) \, \d x.
\end{align*}
Hence, by the Poincar\'{e} inequality \eqref{poincare} and \eqref{S2}, we have
\begin{align*}
\left| \int_{\Omega} F'(\phi) \left( \phi-\overline{\phi} \right) \, \d x \right|
&\leq \left| \int_\Omega \left( \mu -\overline{\mu}\right) \phi \, \d x \right|
+ \theta_0 \left| \int_\Omega \phi (\phi-\overline{\phi}) \, \d x\right|
\\
& \quad + \left| \int_{\Omega} \nabla A(\phi) \cdot
\sqrt{a(\phi)} \nabla \phi  \, \d x \right|
+ \left| \int_{\Omega} \nabla A(\phi) \cdot
\frac{a'(\phi)}{2\sqrt{a(\phi)}} \nabla \phi \,  ( \phi-\overline{\phi})  \, \d x \right|
\\
& \leq C \| \nabla \mu\|_{L^2(\Omega)} \| \phi\|_{L^2(\Omega)}
+ C \| \phi -\overline{\phi}\|_{L^2(\Omega)}^2
\\
&\quad
+ \left\| \sqrt{a(\phi)} \right\|_{L^\infty(\Omega)}
\| \nabla A(\phi)\|_{L^2(\Omega)} \| \nabla \phi\|_{L^2(\Omega)}
\\
&\quad +C \left\| \frac{a'(\phi)}{2\sqrt{a(\phi)}} \right\|_{L^\infty(\Omega)} \| \phi-\overline{\phi}\|_{L^\infty(\Omega)}
\| \nabla A(\phi)\|_{L^2(\Omega)} \| \nabla \phi\|_{L^2(\Omega)}.
\end{align*}
Thanks to \eqref{Linf-u-phi_1} and \eqref{A-H1}, we obtain
\begin{equation}
\label{11}
\left| \int_{\Omega} F'(\phi) \left( \phi-\overline{\phi} \right) \, \d x \right|
\leq C \| \nabla \mu\|_{L^2(\Omega)} + C.
\end{equation}
Therefore, by \eqref{MZ} and \eqref{11}, we deduce that
\begin{equation}
\label{est-F'}
\| F'(\phi)\|_{L^1(\Omega)} \leq C  \| \nabla \mu\|_{L^2(\Omega)} + C.
\end{equation}
Now, due to the strict positivity of $a(\cdot)$, from \eqref{mu-equiv} we have
$$
\frac{1}{\sqrt{a(\phi)}} \mu = - \Delta A(\phi) + \frac{1}{\sqrt{a(\phi)}} \Psi'(\phi) \quad \text{a.e. in } \Omega \times (0,\infty).
$$
Integrating over $\Omega$ and exploiting the boundary condition on $\phi$, we find
\begin{equation*}
\int_\Omega \frac{1}{\sqrt{a(\phi)}} \mu \, \d x =
\int_\Omega \frac{1}{\sqrt{a(\phi)}} \Psi'(\phi) \, \d x,
\end{equation*}
which immediately gives us
$$
\left| \int_\Omega \frac{1}{\sqrt{a(\phi)}} \mu \, \d x \right|
\leq \int_\Omega \frac{1}{\sqrt{a(\phi)}} |\Psi'(\phi)| \, \d x
\leq \frac{1}{\sqrt{a_m}} \int_\Omega |\Psi'(\phi)| \, \d x.
$$
By \eqref{Linf-u-phi_1} and \eqref{est-F'}, we get
$$
\left| \int_\Omega \frac{1}{\sqrt{a(\phi)}} \mu \, \d x \right| \leq C \| \nabla \mu\|_{L^2(\Omega)} + C.
$$
Therefore, thanks to \eqref{normH1-2}, we deduce that
\begin{equation}
\label{est-mu-H1}
\| \mu\|_{H^1(\Omega)}\leq C \| \nabla \mu\|_{L^2(\Omega)} + C,
\end{equation}
and, we conclude from \eqref{L2-u-mu_1} that
\begin{equation}
\label{mu-H1}
\| \mu \|_{L^2_{\rm uloc}([0,\infty); H^1(\Omega))}\leq C.
\end{equation}
\smallskip

\textbf{$H^2$-estimate of the concentration.} First, multiplying \eqref{mu-equiv} by $-\Delta A(\phi)$, integrating over $\Omega$ and exploiting the boundary conditions \eqref{CH-bc}, we have
\begin{align*}
\int_\Omega \sqrt{a(\phi)} |\Delta A(\phi)|^2 \, \d x
- \int_\Omega  F'(\phi) \Delta A(\phi) \, \d x
= \int_\Omega \nabla \mu \cdot \nabla A(\phi) \, \d x
+ \theta_0 \int_\Omega \nabla \phi \cdot \nabla A(\phi) \, \d x.
\end{align*}
Setting $\phi_k = h_k \circ \phi$, where $h_k$ is defined in \eqref{trunc}, we rewrite the second term on the left-hand side as
\begin{align*}
- \int_\Omega  F'(\phi) \Delta A(\phi) \, \d x
&= - \int_\Omega  F'(\phi_k) \Delta A(\phi) \, \d x
+ \int_\Omega  \left( F'(\phi_k) - F'(\phi) \right) \Delta A(\phi) \, \d x
\\
&= \int_\Omega  \sqrt{a(\phi)} F''(\phi_k) \nabla \phi_k \cdot \nabla \phi  \, \d x
+ \int_\Omega  \left( F'(\phi_k) - F'(\phi) \right) \Delta A(\phi) \, \d x.
\end{align*}
Now, we observe that the first term is non-negative, whereas the second one converges to $0$ as $k \to \infty$ since $F'(\phi) \in L^2(0,T; L^2(\Omega))$ and $A(\phi) \in L^2(0,T;H^2(\Omega))$, for any $T>0$. Then, passing to the limit as $k \to \infty$, we conclude that the term $- \int_\Omega  F'(\phi) \Delta A(\phi) \, \d x$ is non- negative. Thus, recalling that $a(\cdot) \geq a_m>0$ in $[-1,1]$, and exploiting \eqref{Linf-u-phi_1} and \eqref{A-H1}, we obtain
\begin{equation*}
\int_\Omega  |\Delta A(\phi)|^2 \, \d x \leq C \| \nabla \mu\|_{L^2(\Omega)} + C.
\end{equation*}
Since $\partial_\n A(\phi)= \sqrt{a(\phi)} \partial_\n \phi=0$ on $\partial \Omega \times (0,\infty)$, we immediately infer that
\begin{equation}
\label{est-A-H2}
\| A(\phi)\|_{H^2(\Omega)}^2 \leq C \| \nabla \mu\|_{L^2(\Omega)} + C,
\end{equation}
which entails that
\begin{equation}
\label{bound-A-H2}
A(\phi) \in L_{\rm uloc}^4([0,\infty); H^2(\Omega)).
\end{equation}
Therefore, by exploiting \eqref{phi-Aphi-rel} and using the interpolation inequality \eqref{inter-W14}, we deduce that
\begin{align*}
\| \Delta \phi\|_{L^2(\Omega)}
&\leq \norm{\frac{1}{\sqrt{a(\phi)}}}_{L^\infty(\Omega)}
 \| \Delta A(\phi)\|_{L^2(\Omega)}
+ \norm{ \frac{a'(\phi)}{2 a^2(\phi)} }_{L^\infty(\Omega)}
\| \nabla A(\phi)\|_{L^4(\Omega)}^2
\notag
\\
&\leq \frac{1}{\sqrt{a_m}}
 \| \Delta A(\phi)\|_{L^2(\Omega)}
+ \frac{C}{2 a^2_m}
\| A(\phi)\|_{L^\infty(\Omega)} \| A(\phi)\|_{H^2(\Omega)}
\notag
\\
& \leq C \|  A(\phi)\|_{H^2(\Omega)},
\end{align*}
which, also thanks to \eqref{est-A-H2}, implies that
\begin{equation}
\label{est-phi-H2}
\| \Delta \phi\|_{L^2(\Omega)}^2 \leq C \| \nabla \mu\|_{L^2(\Omega)} + C.
\end{equation}
Finally, we conclude that
\begin{equation}
\label{bound-phi-H2}
\phi \in L_{\rm uloc}^4([0,\infty); H^2(\Omega)).
\end{equation}

\medskip

\textbf{$L^p$-estimate of the potential.} In light of \eqref{a-ndeg} and \eqref{mu-equiv}, we consider the elliptic system 
\begin{align*}
\label{Aphi-ELL-2}
\begin{cases}
- \sqrt{a(\phi)}\Delta A(\phi)+ F'(\phi) = f \quad &\text{a.e. in }\Omega 
\times (0,\infty), \\
\partial_{\n} A(\phi)=0\quad &\text{a.e. on }\partial \Omega \times (0,\infty),
\end{cases}
\end{align*}
where
$$
f=\mu+\theta_0 \phi \in L_{\rm uloc}^2([0,\infty); H^1(\Omega)).
$$
By the theory developed in Section \ref{s-stationary}, in particular Lemma \ref{ell2}, we deduce that
\begin{equation}
\label{F'-LP}
\| F'(\phi)\|_{L^p(\Omega)}\leq C \| \mu\|_{L^p(\Omega)} + C \| \phi\|_{L^p(\Omega)},
\end{equation}
 for any $2 \leq p < \infty$ if $d=2$, and $p=6$ if $d=3$, which entails that
 \begin{equation}
\label{bound-F'-Lp}
F'(\phi) \in L_{\rm uloc}^2([0,\infty); L^p(\Omega)), 
\end{equation}
 for any $2 \leq p < \infty$ if $d=2$, and $p=6$ if $d=3$.

\medskip

\textbf{$W^{2,p}$-estimate of the concentration.} By \eqref{mu-equiv}, we observe that
\begin{equation}
\label{Aphi-ELL}
\begin{cases}
-\Delta A(\phi)+ A(\phi) = f \quad &\text{a.e. in }\Omega 
\times (0,\infty), \\
\partial_{\n} A(\phi)=0\quad &\text{a.e. on }\partial \Omega \times (0,\infty),
\end{cases}
\end{equation}
where
$$
f=\frac{1}{\sqrt{a(\phi)}} \left( \mu- \Psi'(\phi) \right) + A(\phi).
$$
Since the elliptic regularity theory entails that
\begin{equation}
\label{A2p}
\| A(\phi)\|_{W^{2,p}(\Omega)}\leq C_p \| f\|_{L^p(\Omega)},
\end{equation}
and observing that $f \in L_{\rm uloc}^2([0,\infty); L^p(\Omega))$ 
for any $2 \leq p < \infty$ if $d=2$, and $p=6$ if $d=3$,
we deduce that 
$$
A(\phi) \in L_{\rm uloc}^2([0,\infty); W^{2,p}(\Omega)), 
$$
for any $2 \leq p < \infty$ if $d=2$, and $p=6$ if $d=3$.
In addition, by \eqref{GN-utile} and  \eqref{phi-Aphi-rel}, we also infer that
\begin{align}
\| \Delta \phi\|_{L^p(\Omega)}
&
\leq
\left\| \frac{1}{\sqrt{a(\phi)}} \right\|_{L^\infty(\Omega)}
\| \Delta A(\phi) \|_{L^p(\Omega)}
+ \left\| \frac{a'(\phi)}{2 a^2(\phi)} \right\|_{L^\infty(\Omega)}
\| \nabla A(\phi)\|_{L^{2p}(\Omega)}^2
\notag
\\
&\leq
C \| A(\phi)\|_{W^{2,p}(\Omega)} + C \|A(\phi)\|_{L^\infty(\Omega)}  \|A(\phi)\|_{W^{2,p}(\Omega)}
\notag
\\
&\leq C \| A(\phi)\|_{W^{2,p}(\Omega)}.
\label{phi-W2p}
\end{align}
Thus, we conclude that
\begin{equation*}
\phi \in L_{\rm uloc}^2([0,\infty); W^{2,p}(\Omega)), 
\end{equation*}
for any $2 \leq p < \infty$ if $d=2$, and $p=6$ if $d=3$.

\section{Uniqueness of weak solutions in two dimensions}
\label{W-uniq}

We consider two weak solutions $\phi_1$ and $\phi_2$  to problem \eqref{CH1}-\eqref{CH-ic} originating from two initial data $\phi_{1}^0$ and $\phi_2^0$ such that $\overline{\phi_{1}^0}=\overline{\phi_{2}^0}$.
Following the notation in \cite{CGGG2025}, we notice that \eqref{e2} and \eqref{e3} entail that
\begin{equation}
\label{mu-G}
\mu_i-\overline{\mu_i}= - \mathcal{G}_{\phi_i} \partial_t \phi_i, \quad
\text{ a.e. in } \Omega\times (0,\infty), \ i=1,2,
\end{equation}
where $\mu_1$ and $\mu_2$ are the chemical potentials corresponding to $\phi_1$ and $\phi_2$, respectively.
Setting $\Phi=\phi_1-\phi_2$, we find from \eqref{CH2} that
\begin{align*}
&-\div \left( a(\phi_1) \nabla \Phi \right)
-\div \left( (a(\phi_1)-a(\phi_2)) \nabla \phi_2 \right)
+\frac{a'(\phi_1)}{2} \left( |\nabla \phi_1|^2 - |\nabla \phi_2|^2 \right)
\\
&\quad
+ \frac12 \left( a'(\phi_1)-a'(\phi_2)\right) | \nabla \phi_2|^2
+ \Psi'(\phi_1)-\Psi'(\phi_2)= \mu_1-\mu_2,  \quad
\text{ a.e. in } \Omega\times (0,\infty).
\end{align*}
Multiplying the above equation by $\Phi$ and integrating over $\Omega$, we get
 \begin{align*}
& \int_\Omega a(\phi_1) |\nabla \Phi |^2 \, \d x
 +
  \int_\Omega \left( F'(\phi_1)-F'(\phi_2) \right)   \Phi \, \d x
- \int_{\Omega} \left( \mu_1-\mu_2\right)  \Phi \, \d x
\\
&=
 \theta_0 \norm{\Phi}_{L^2(\Omega)}^2
 - \int_\Omega \left( a(\phi_1)-a(\phi_2)\right) \nabla \phi_2 \cdot \nabla \Phi \, \d x
 \\
 &\quad -\int_\Omega \frac{a'(\phi_1)}{2}  \left( |\nabla \phi_1|^2-|\nabla \phi_2|^2\right) \Phi
 \, \d x
 - \int_\Omega \frac12  \left( a'(\phi_1)-a'(\phi_2)\right) | \nabla \phi_2|^2 \Phi \, \d x.
\end{align*}
In light of \eqref{mu-G}, since $\overline{\Phi}\equiv 0$ by the conservation of mass, the third term on the left-hand side is
\begin{align*}
- \int_{\Omega} \left( \mu_1-\mu_2\right)  \Phi \, \d x
&=
 \int_\Omega \left( \G_{\phi_1}  \partial_t \phi_1 - \G_{\phi_2} \partial_t \phi_2  \right)  \Phi \, \d x
 \\
 &=
\int_\Omega  \G_{\phi_1} \partial_t \Phi  \, \Phi \, \d x
+
\int_\Omega \left( \G_{\phi_1}-\G_{\phi_2} \right)
  \partial_t \phi_2 \,  \Phi \, \d x.
\end{align*}
We know from \cite[eq. (3.16)]{CGGG2025} that the following integration-by-parts  formula holds in the class of weak solutions
\begin{align}
\int_\Omega \G_{\phi_1} \partial_t \Phi \,  \Phi \, \d x
&= \ddt \frac{1}{2} \norm{\sqrt{b(\phi_1)}\nabla \G_{\phi_1} \Phi}_{L^2(\Omega)}^2
+\frac12 \int_\Omega \nabla \G \partial_t \phi_1 \cdot b''(\phi_1) \nabla \phi_1 \left| \nabla \G_{\phi_1} \Phi \right|^2 \, \d x \notag
\\[5pt]
&\quad +  \int_\Omega \nabla \G \partial_t \phi_1 \cdot  b'(\phi_1) \left( D^2 \G_{\phi_1} \Phi \nabla \G_{\phi_1} \Phi \right) \, \d x,
\label{IP-time}
\end{align}
almost everywhere in $(0,\infty)$, where $D^2 f$ is the Hessian of $f$.
On the other hand, we have
\begin{align}
\label{GdG}
\int_\Omega \left( \G_{\phi_1}-\G_{\phi_2} \right)
  \partial_t \phi_2 \,  \Phi \, \d x
&=
  \l
  \partial_t \phi_2,  \left( \G_{\phi_1}-\G_{\phi_2} \right)  \Phi \r\notag
\\
 & =
-
  \int_\Omega b(\phi_2)\nabla \mu_2 \cdot \nabla \left( \G_{\phi_1}-\G_{\phi_2} \right)  \Phi \, \d x\notag
\\
&=
- \int_\Omega (b(\phi_2)-b(\phi_1)) \nabla \mu_2 \cdot \nabla \G_{\phi_1} \Phi \, \d x.
\end{align}
Thus, we deduce that
\begin{align*}
&- \int_{\Omega} \left( \mu_1-\mu_2\right)  \Phi \, \d x
\\
&
=
\ddt \frac{1}{2} \norm{\sqrt{b(\phi_1)}\nabla \G_{\phi_1} \Phi}_{L^2(\Omega)}^2
+\frac12 \int_\Omega \nabla \G \partial_t \phi_1 \cdot b''(\phi_1) \nabla \phi_1 \left| \nabla \G_{\phi_1} \Phi \right|^2 \, \d x \notag
\\[5pt]
&\quad +  \int_\Omega \nabla \G \partial_t \phi_1 \cdot  b'(\phi_1) \left( D^2 \G_{\phi_1} \Phi \nabla \G_{\phi_1} \Phi \right) \, \d x  - \int_\Omega (b(\phi_2)-b(\phi_1)) \nabla \mu_2 \cdot \nabla \G_{\phi_1} \Phi \, \d x.
\end{align*}
Therefore, we end up with the differential equality
\begin{align}
&\ddt \frac{1}{2} \norm{\sqrt{b(\phi_1)}\nabla \G_{\phi_1} \Phi}_{L^2(\Omega)}^2
+ \int_\Omega a(\phi_1) |\nabla \Phi |^2 \, \d x
 +
  \int_\Omega \left( F'(\phi_1)-F'(\phi_2) \right)   \Phi \, \d x
  \notag
\\
&=   \theta_0 \norm{\Phi}_{L^2(\Omega)}^2
 - \int_\Omega \left( a(\phi_1)-a(\phi_2)\right) \nabla \phi_2 \cdot \nabla \Phi \, \d x
 \notag
 \\
 &\quad -\int_\Omega \frac{a'(\phi_1)}{2}  \left( |\nabla \phi_1|^2-|\nabla \phi_2|^2\right) \Phi
 \, \d x
 - \int_\Omega \frac12  \left( a'(\phi_1)-a'(\phi_2)\right) | \nabla \phi_2|^2 \Phi \, \d x
 \notag
 \\
 &\quad
 -\frac12 \int_\Omega \nabla \G \partial_t \phi_1 \cdot b''(\phi_1) \nabla \phi_1 \left| \nabla \G_{\phi_1} \Phi \right|^2 \, \d x
 -
  \int_\Omega \nabla \G \partial_t \phi_1 \cdot  b'(\phi_1) \left( D^2 \G_{\phi_1} \Phi \nabla \G_{\phi_1} \Phi \right) \, \d x
  \notag
  \\
  &\quad  + \int_\Omega (b(\phi_2)-b(\phi_1)) \nabla \mu_2 \cdot \nabla \G_{\phi_1} \Phi \, \d x.
  \label{diff-rel-1}
 \end{align}
Let us now proceed by estimating the terms on the right-hand side.
First, the first term on the right-hand side is simply estimated by
\begin{equation*}
\label{I_4}
\theta_0 \norm{\Phi}_{L^2(\Omega)}^2 \leq
\frac{a_m}{16} \norm{\nabla \Phi}_{L^2(\Omega)}^2
 +C
\norm{ \Phi }_{H^{-1}_{(0)}(\Omega)}^2\!,
\end{equation*}
The terms containing the nonlinear diffusion $a$ can be controlled as follows.
By \eqref{a-ndeg}, \eqref{lady_2} and the Hilbert interpolation, we obtain 
\begin{align*}
\left|  - \int_\Omega \left( a(\phi_1)-a(\phi_2)\right) \nabla \phi_2 \cdot \nabla \Phi \, \d x \right|
&\leq \| a(\phi_1)-a(\phi_2)\|_{L^4(\Omega)} \| \nabla \phi_2\|_{L^4(\Omega)}
\| \nabla \Phi\|_{L^2(\Omega)}
\\
&\leq C \| \Phi\|_{L^4(\Omega)} \| \nabla \phi_2\|_{L^4(\Omega)}
\| \nabla \Phi\|_{L^2(\Omega)}
\\
&\leq C \| \Phi\|_{H^{-1}_{(0)}(\Omega)}^\frac14
 \| \nabla \phi_2\|_{L^2(\Omega)}^\frac12  \| \phi_2\|_{H^2(\Omega)}^\frac12
\| \nabla \Phi\|_{L^2(\Omega)}^\frac74
\\
& \leq \frac{a_m}{16}\| \nabla \Phi\|_{L^2(\Omega)}^2
+C  \| \phi_2\|_{H^2(\Omega)}^4  \| \Phi\|_{H^{-1}_{(0)}(\Omega)}^2.
\end{align*}
Similarly, we have
\begin{align*}
&\left| -\int_\Omega \frac{a'(\phi_1)}{2}  \left( |\nabla \phi_1|^2-|\nabla \phi_2|^2\right)\Phi
 \, \d x\right|
 \\
\quad &\leq  \left\| \frac{a'(\phi_1)}{2}\right\|_{L^\infty(\Omega)} \left( \| \nabla \phi_1\|_{L^4(\Omega)}+ \| \nabla \phi_2\|_{L^4(\Omega)} \right) \| \nabla \Phi\|_{L^2(\Omega)} \| \Phi\|_{L^4(\Omega)}
\\
\quad &\leq C(\|\phi_1\|_{H^2(\Omega)}^{\frac12}+\|\phi_2\|_{H^2(\Omega)}^{\frac12})\|\Phi\|_{H^{-1}_{(0)}(\Omega)}^{\frac14}\|\nabla \Phi\|^{\frac74}_{L^2(\Omega)}
\\
&\leq \frac{a_m}{16}\|\nabla \Phi\|^{2}_{L^2(\Omega)}+C(\|\phi_1\|_{H^2(\Omega)}^4+\|\phi_2\|_{H^2(\Omega)}^{4})\|\Phi\|_{H^{-1}_{(0)}(\Omega)}^{2},
 \end{align*}
and
 \begin{align*}
\left| - \int_\Omega \frac12  \left( a'(\phi_1)-a'(\phi_2)\right) | \nabla \phi_2|^2 \Phi \, \d x\right|
&\leq C \| a'(\phi_1)-a'(\phi_2)\|_{L^4(\Omega)}
\|\nabla\phi_2\|_{L^4(\Omega)}^2\|\Phi\|_{L^4(\Omega)}
\\
&\leq C\|\phi_2\|_{H^2(\Omega)}
\| \Phi\|_{L^4(\Omega)}^2
\\
&\leq C\|\phi_2\|_{H^2(\Omega)}
\|\Phi\|_{H^{-1}_{(0)}(\Omega)}^{\frac12}
\|\nabla \Phi\|^{\frac32}
\\
&\leq \frac{a_m}{16}\|\nabla \Phi\|^{2}_{L^2(\Omega)}+C\|\phi_2\|_{H^2(\Omega)}^4\|\Phi\|_{H^{-1}_{(0)}(\Omega)}^{2}.
\end{align*}
The remaining terms
\begin{align}
\label{I1}
I_1 &=- \frac12 \int_\Omega \nabla \G \partial_t \phi_1 \cdot b''(\phi_1) \nabla \phi_1 \left| \nabla \G_{\phi_1} \Phi \right|^2 \, \d x,
\\
\label{I2}
I_2 &=-  \int_\Omega \nabla \G \partial_t \phi_1 \cdot  b'(\phi_1) \left( D^2 \G_{\phi_1} \Phi \nabla \G_{\phi_1} \Phi \right) \, \d x,
\\
\label{I3}
I_3 &= \int_\Omega (b(\phi_1)-b(\phi_2)) \nabla \mu_2 \cdot \nabla \G_{\phi_1} \Phi \, \d x
\end{align}
 require a careful control based on higher order estimates for the operator $\mathcal{G}_{\phi_1} \Phi$. This has been done in \cite[Section 3]{CGGG2025}  based on suitable controls of
$\| \mathcal{G}_{\phi_1} \Phi\|_{H^2(\Omega)}$ and of $\| \mathcal{G}_{\phi_1} \Phi\|_{W^{2,4}(\Omega)}$ in terms of $\|\phi_1\|_{H^2(\Omega)}$, $\|\nabla \Phi\|_{L^2(\Omega)}$ and $\norm{\nabla \G_{\phi_1} \Phi}_{L^2(\Omega)}$ (that is equivalent to $\|\Phi\|_{H^{-1}_{(0)}}$), coming from the elliptic regularity theory.
Based on \cite[Eqn. (3.46)]{CGGG2025}, there holds
\begin{equation*}
\label{H2-Phi}
\| \mathcal{G}_{\phi_1} \Phi\|_{H^2(\Omega)}
\leq C \left( \| \nabla \phi_1\|_{L^2(\Omega)}
 \| \phi_1\|_{H^2(\Omega)}
 \| \nabla \mathcal{G}_{\phi_1} \Phi\|_{L^2(\Omega)}+
 \| \nabla \mathcal{G}_{\phi_1} \Phi \|_{L^2(\Omega)}^\frac12
 \| \nabla \Phi\|_{L^2(\Omega)}^\frac12\right)\!.
\end{equation*}
Thus, we have
\begin{align*}
|I_1|&\leq C  \norm{\partial_t \phi_1}_{ H^{-1}_{(0)}(\Omega)}
\norm{ b''(\phi_1)  \nabla \phi_1 |\nabla \G_{\phi_1} \Phi |^2}_{L^2(\Omega)}\!
\\
&\leq
C \norm{\partial_t \phi_1}_{H^{-1}_{(0)}(\Omega)}
\norm{\nabla \phi_1}_{L^6(\Omega)}
\norm{\nabla \G_{\phi_1} \Phi}_{L^6(\Omega)}^2
\notag
\\
&\leq
C \norm{\partial_t \phi_1}_{H^{-1}_{(0)}(\Omega)}
\norm{\nabla \phi_1}_{L^2(\Omega)}^\frac13 \norm{\phi_1}_{H^2(\Omega)}^\frac23
\norm{\nabla \G_{\phi_1} \Phi}_{L^2(\Omega)}^\frac23
\norm{\G_{\phi_1} \Phi}_{H^2(\Omega)}^\frac43
\notag
\\
&\leq C \norm{\partial_t \phi_1}_{H^{-1}_{(0)}(\Omega)}
 \norm{\phi_1}_{H^2(\Omega)}^\frac23
\norm{\nabla \G_{\phi_1} \Phi}_{L^2(\Omega)}^\frac23
\notag
\\
&\quad \times
\left( \| \phi_1\|_{H^2(\Omega)} \norm{\nabla \G_{\phi_1} \Phi}_{L^2(\Omega)} + \norm{\nabla \G_{\phi_1} \Phi}_{L^2(\Omega)}^\frac12 \| \nabla \Phi\|_{L^2(\Omega)}^\frac12 \right)^\frac43
\notag
\\
& \leq
C \norm{\partial_t \phi_1}_{H^{-1}_{(0)}(\Omega)}
 \norm{\phi_1}_{H^2(\Omega)}^2
\norm{\nabla \G_{\phi_1} \Phi}_{L^2(\Omega)}^2
\notag
\\
&\quad
+
C \norm{\partial_t \phi_1}_{H^{-1}_{(0)}(\Omega)}
 \norm{\phi_1}_{H^2(\Omega)}^\frac23
\norm{\nabla \G_{\phi_1} \Phi}_{L^2(\Omega)}^\frac43
 \| \nabla \Phi\|_{L^2(\Omega)}^\frac23
 \notag
 \\
& \leq \frac{a_m}{16} \| \nabla \Phi\|_{L^2(\Omega)}^2
 + C  \left( \norm{\partial_t \phi_1}_{H^{-1}_{(0)}(\Omega)}
 \norm{\phi_1}_{H^2(\Omega)}^2 + \norm{\partial_t \phi_1}_{H^{-1}_{(0)}(\Omega)}^\frac32
 \norm{\phi_1}_{H^2(\Omega)}\right) \norm{\nabla \G_{\phi_1} \Phi}_{L^2(\Omega)}^2
 \notag
 \\
 \label{I1-f}
& \leq \frac{a_m}{16} \| \nabla \Phi\|_{L^2(\Omega)}^2
 + C \left( \norm{\partial_t \phi_1}_{H^{-1}_{(0)}(\Omega)}^2+
 \norm{\phi_1}_{H^2(\Omega)}^4 \right)\|\Phi\|_{H^{-1}_{(0)}(\Omega)}^{2}\!.
\end{align*}
Concerning $I_2$ and $I_3$, we report the final results from \cite[Eqns. (3.50) and (3.51)]{CGGG2025} 
\begin{equation*}
\label{I2-f}
|I_2| \leq \frac{a_m}{16}  \norm{\nabla \Phi}_{L^2(\Omega)}^2
 + C \left( \norm{\partial_t \phi_1}_{ H^{-1}_{(0)}(\Omega)}^2 +
 \| \phi_1\|_{H^2(\Omega)}^4\right)
\|\Phi\|_{H^{-1}_{(0)}(\Omega)}^{2}\!,
\end{equation*}
and
\begin{equation*}
\label{I3-f}
|I_3| \leq \frac{a_m}{16}  \| \nabla \Phi\|_{L^2(\Omega)}^2
 + C  \left( \norm{\nabla \mu_2}_{L^2(\Omega)}^2+
 \| \phi_1\|_{H^2(\Omega)}^4
   \right) \|\Phi\|_{H^{-1}_{(0)}(\Omega)}^{2}\!.
\end{equation*}
We now collect all the estimates above in \eqref{diff-rel-1}.  Observing that $(F'(\phi_1)-F'(\phi_2), \phi_1-\phi_2) \geq 0$  and that $a(\phi_1)\geq a_m>0$, recalling that
$\|\Phi\|_{H^{-1}_{(0)}(\Omega)}$ is equivalent to $\norm{\sqrt{b(\phi_1)} \nabla \G_{\phi_1} \Phi}_{L^2(\Omega)}$, we arrive at the following differential inequality
\begin{equation}
\label{diff-rel-2}
\ddt  \norm{\sqrt{b(\phi_1)}\nabla \G_{\phi_1} \Phi}_{L^2(\Omega)}^2
+ a_m\|\nabla \Phi \|_{L^2(\Omega)}^2
\leq h(t) \norm{ \sqrt{b(\phi_1)} \nabla \G_{\phi_1} \Phi}_{L^2(\Omega)}^2
\! ,
\end{equation}
where
$$
h(\cdot)= C \Big( 1+
\norm{\nabla \mu_2}_{L^2(\Omega)}^2
+ \norm{\partial_t \phi_1}_{ H^{-1}_{(0)}(\Omega)}^2 +
 \| \phi_1\|_{H^2(\Omega)}^4+
 \| \phi_2\|_{H^2(\Omega)}^4  \Big) \in L^1(0,T), \quad \forall \, T>0.
$$
A straightforward application of the Gronwall lemma entails that
\begin{equation}
\norm{\sqrt{b(\phi_1)} \nabla \G_{\phi_1}\left( \phi_1(t)-\phi_2(t) \right)}_{L^2(\Omega)}^2
\leq \norm{\sqrt{b(\phi_1)} \nabla \G_{\phi_1} \left( \phi_1^0-\phi_2^0 \right)}_{L^2(\Omega)}^2
\mathrm{e}^{\int_0^t h(s)\, \d s}, \quad \forall \, t \geq 0,
\end{equation}
which implies \eqref{UNIQ} by the equivalence of the norms in $H^{-1}_{(0)}(\Omega)$. Thus, if $\phi_1^0=\phi_2^0$, the  uniqueness of the weak solutions immediately follows. The proof of part (3) of Theorem \ref{well-pos-2D} is done.

\section{Global Regularity of weak solutions in two dimensions}
\label{S-Strong}

This section is devoted to the propagation of regularity for any global weak solution in two dimensions, namely we prove part (3) of Theorem \ref{well-pos-2D}. Furthermore, relying on the achieved regularity for positive times, we show the validity of the energy equality claimed in part (4) of Theorem \ref{well-pos-2D}.

\subsection{Global regularity}
Below, we carry out formal estimates, which can be easily performed in the approximation scheme exploited in \cite{SP2013}. Multiplying \eqref{CH1} by $\partial_t \mu$ and integrating over $\Omega$, we find
\begin{equation}\label{counts}
\frac{1}{2}\ddt  \int_{\Omega} b(\phi) \left|\nabla \mu \right|^2 \, \d x  - \frac{1}{2} \int_{\Omega} b'(\phi) \partial_t \phi  |\nabla \mu|^2  \, \d x
+ \int_\Omega \partial_t \phi \, \partial_t \mu \, \d x =0.
\end{equation}
By exploiting \eqref{mu-equiv}, we observe that
\begin{align*}
\int_\Omega \partial_t \phi \, \partial_t \mu \, \d x
&=
\int_\Omega \partial_t \left( -\sqrt{a(\phi)} \Delta A(\phi)\right) \partial_t \phi \, \d x
+ \int_\Omega F''(\phi)|\partial_t \phi|^2 \, \d x - \theta_0 \| \partial_t \phi\|_{L^2(\Omega)}^2
\\
&= - \int_\Omega \frac{a'(\phi)}{2\sqrt{a(\phi)}} \Delta A(\phi) |\partial_t \phi|^2 \, \d x
- \int_\Omega \sqrt{a(\phi)} \Delta \partial_t A(\phi) \, \partial_t \phi \, \d x
\\
&\quad + \int_\Omega F''(\phi)|\partial_t \phi|^2 \, \d x - \theta_0 \| \partial_t \phi\|_{L^2(\Omega)}^2
\\
&=- \int_\Omega \frac{a'(\phi)}{2\sqrt{a(\phi)}} \Delta A(\phi) |\partial_t \phi|^2 \, \d x
+ \int_\Omega  \nabla \partial_t A(\phi) \cdot \nabla \left( \sqrt{a(\phi)}\partial_t \phi \right) \, \d x
\\
&\quad + \int_\Omega F''(\phi)|\partial_t \phi|^2 \, \d x - \theta_0 \| \partial_t \phi\|_{L^2(\Omega)}^2
\\
&=- \int_\Omega \frac{a'(\phi)}{2\sqrt{a(\phi)}} \Delta A(\phi) |\partial_t \phi|^2 \, \d x
+ \int_\Omega  \Big| \nabla \left( \sqrt{a(\phi)}\partial_t \phi \right) \Big|^2 \, \d x
\\
&\quad + \int_\Omega F''(\phi)|\partial_t \phi|^2 \, \d x - \theta_0 \| \partial_t \phi\|_{L^2(\Omega)}^2.
\end{align*}
By a direct calculation, we notice that
\begin{align*}
 &\int_\Omega  \Big| \nabla \left( \sqrt{a(\phi)}\partial_t \phi \right) \Big|^2 \, \d x
 \\
 & = \int_\Omega \left( \frac{a'(\phi)}{2\sqrt{a(\phi)}}  \nabla \phi \, \partial_t \phi +
 \sqrt{a(\phi)} \nabla \partial_t \phi\right) \cdot  \left( \frac{a'(\phi)}{2\sqrt{a(\phi)}}  \nabla \phi \, \partial_t \phi +
 \sqrt{a(\phi)} \nabla \partial_t \phi \right)  \, \d x
 \\
 & = \int_\Omega a(\phi) |\nabla \partial_t \phi|^2 \, \d x +
 \int_\Omega \frac{(a'(\phi))^2}{4 a(\phi)} |\nabla \phi|^2 |\partial_t \phi|^2 \, \d x
 + \int_\Omega a'(\phi) \nabla \phi \cdot \nabla \partial_t \phi \, \partial_t \phi \, \d x.
\end{align*}
Collecting the above calculations all together, we end up with
\begin{align}
\label{first-rel}
&\frac{1}{2} \ddt  \int_{\Omega} b(\phi) \left|\nabla \mu \right|^2 \, \d x
+\int_\Omega a(\phi) |\nabla \partial_t \phi|^2 \, \d x
\notag \\
&\quad +
 \int_\Omega \frac{(a'(\phi))^2}{4 a(\phi)} |\nabla \phi|^2 |\partial_t \phi|^2 \, \d x + \int_\Omega F''(\phi)|\partial_t \phi|^2 \, \d x
\notag\\
&= \theta_0 \| \partial_t \phi\|_{L^2(\Omega)}^2
+ \frac{1}{2}  \int_{\Omega} b'(\phi) \partial_t \phi  |\nabla \mu|^2 \, \d x
+ \int_\Omega \frac{a'(\phi)}{2\sqrt{a(\phi)}} \Delta A(\phi) |\partial_t \phi|^2 \, \d x
\notag\\
&\quad
-  \int_\Omega a'(\phi) \nabla \phi \cdot \nabla \partial_t \phi \, \partial_t \phi \, \d x.
\end{align}

We proceed with the control of the nonlinear terms.
First, by interpolation in Hilbert spaces and \eqref{est-phit-H-1}, we know that
$$
\theta_0 \|\partial_t \phi \|_{L^2(\Omega)}^2 \leq \frac{a_m}{4} \| \nabla \partial_t \phi\|_{L^2(\Omega)}^2 +
C \| \nabla \mu\|_{L^2(\Omega)}^2.
$$
Reasoning as in \cite{CGGG2025} (see, in particular, Eqns. (4.26)-(4.27) and the subsequent estimates therein), we have
\begin{equation}\label{third_method}
\begin{split}
\left| \int_{\Omega} b'(\phi )\partial_t \phi |\nabla \mu |^2 \, \d x\right|
&\leq \frac{a_m}{4} \norm{\nabla \partial_t \phi}_{L^2(\Omega)}^2 +
 C\left(  \| \phi\|_{H^2(\Omega)}^4 + \| \nabla \mu\|_{L^2(\Omega)}^2
 \right) \| \nabla \mu\|_{L^2(\Omega)}^2.
\end{split}
\end{equation}
Next, we consider
\begin{align*}
\left| \int_\Omega \frac{a'(\phi)}{2\sqrt{a(\phi)}} \Delta A(\phi) |\partial_t \phi|^2 \, \d x \right|
&
\leq \left\| \frac{a'(\phi)}{2\sqrt{a(\phi)}} \right\|_{L^\infty(\Omega)} \| \Delta A(\phi)\|_{L^2(\Omega)} \| \partial_t \phi\|_{L^4(\Omega)}^2
\\
&\leq
C \| \Delta A(\phi)\|_{L^2(\Omega)}  \norm{\nabla \partial_t \phi}_{L^2(\Omega)}^\frac32
\| \partial_t \phi\|_{H^{-1}_{(0)}(\Omega)}^\frac12
\\
&\leq
\frac{a_m}{8} \norm{\nabla \partial_t \phi}_{L^2(\Omega)}^2 +
C \| A(\phi)\|_{H^2(\Omega)}^4  \| \nabla \mu\|_{L^2(\Omega)}^2.
\end{align*}
Similarly, by \eqref{lady_2} and \eqref{Linf-u-phi_1}, we get
\begin{align*}
\left|  \int_\Omega a'(\phi) \nabla \phi \cdot \nabla \partial_t \phi \, \partial_t \phi \, \d x  \right|
&\leq
\| a'(\phi)\|_{L^\infty(\Omega)} \norm{\nabla \partial_t \phi}_{L^2(\Omega)} \| \nabla \phi\|_{L^4(\Omega)}
\| \partial_t \phi\|_{L^4(\Omega)}
\\
&\leq
C   \norm{\nabla \partial_t \phi}_{L^2(\Omega)}^\frac74
\| \phi\|_{H^2(\Omega)}^\frac12
\| \partial_t \phi\|_{H^{-1}_{(0)}(\Omega)}^\frac14
\\
&
\leq
\frac{a_m}{8} \norm{\nabla \partial_t \phi}_{L^2(\Omega)}^2 +
C \| \phi\|_{H^2(\Omega)}^4  \| \nabla \mu\|_{L^2(\Omega)}^2.
\end{align*}
Hence, recalling that \eqref{est-A-H2} and \eqref{est-phi-H2}, we arrive at
\begin{align}
\label{final-3}
\ddt \left( \int_{\Omega} b(\phi) \left|\nabla \mu\right|^2 \, \d x\right)
 + \frac{a_m}{4}
 \norm{\nabla \partial_t \phi}_{L^2(\Omega)}^2
\leq C\left(1+  \| \nabla \mu\|_{L^2(\Omega)}^2
 \right)
  \int_{\Omega} b(\phi) \left|\nabla \mu\right|^2 \, \d x.
\end{align}
In order to derive a global control independent of time, we recall that, due to \eqref{mu-H1},
$$
\sup_{t \geq 0} \int_t^{t+1} \int_{\Omega} b(\phi) \left|\nabla \mu\right|^2 \, \d x \, \d s\leq C_0,\quad
\sup_{t \geq 0} \int_t^{t+1} C \left(1+ \| \nabla \mu\|_{L^2(\Omega)}^2
 \right) \, \d s \leq C_1,
$$
where $C_0$ and $C_1$ are two positive constant depending on $E(\phi_0)$, $m$ and the parameters of the system.
By applying the uniform Gronwall lemma (see \cite[Chapter III, Lemma
1.1]{TEMAM}),
we derive that
\begin{equation}
\label{Glob_muH1_3}
 \sup_{t \geq \tau} \int_{\Omega} |\nabla \mu(t)|^2 \, \d x
  \leq  \frac{C_0}{b_m \tau } \, {\rm exp} (C_1).
 \end{equation}
In the sequel, we denote by $C(\tau)$ a generic constant depending on the parameters of the system, the initial energy $E(\phi_0)$, and $\tau>0$.
Thanks to \eqref{est-mu-H1}, it immediately follows from \eqref{Glob_muH1_3} that
\begin{equation}
\label{Glob_muH1_4}
\| \mu\|_{L^\infty(\tau,\infty; H^1(\Omega))} \leq C(\tau).
\end{equation}
As a consequence, we learn from \eqref{est-phit-H-1} that
\begin{equation}
\label{phit-inf}
 \| \partial_t \phi \|_{L^\infty(\tau,\infty;  H^{-1}_{(0)}(\Omega))}\leq C(\tau).
\end{equation}
Now, we integrate \eqref{final-3} on the time interval $[t,t+1]$ with $t \geq \tau$. Thanks to \eqref{L2-u-mu_1} and \eqref{Glob_muH1_4}, we infer that
\begin{equation}
\label{phit-H1}
\sup_{t \geq \tau} \int_t^{t+1} \| \nabla \partial_t \phi (s)\|_{L^2(\Omega)}^2 \, \d s
\leq  C(\tau),
\end{equation}
which,  by the conservation of mass,
 gives 
 $$
 \partial_t \phi \in L^2_{\uloc}([\tau,\infty); H^{1}_{(0)}(\Omega)).
 $$
  By \eqref{F'-LP}, \eqref{A2p} and \eqref{phi-W2p}, we learn that
\begin{equation}
\label{phi2p}
\| \phi\|_{L^\infty(\tau,\infty; W^{2,p}(\Omega))}
+ \| F'(\phi)\|_{L^\infty(\tau,\infty; L^p(\Omega))}
\leq C(\tau, p), \quad \forall \, p \in [2,\infty).
\end{equation}
Next, we prove the separation property \eqref{RS2} following \cite{CG} and \cite{HW2021}. To this end, we first observe that $\phi \in L^\infty(\tau,\infty; W^{2,6}(\Omega)) \cap W^{1,2}_{\rm uloc}(\tau, \infty;H^1(\Omega)) \hookrightarrow C(\overline{\Omega} \times [\tau; T])$, for any $T \geq \tau$, by standard interpolation. We claim that
\begin{equation}
\label{F''p}
\| F''(\phi)\|_{L^\infty(\tau,\infty; L^p(\Omega))}
\leq  C \left(\tau, p\right) \!, \quad \forall \, p \in [2,\infty).
\end{equation}
In fact, recalling that $\|f\|_{L^p(\Omega)}\leq C\sqrt{p}\|f\|_{H^1(\Omega)}$ for $2\leq p<\infty$, we learn from \eqref{F'-LP} that
\begin{align*}
\esssup_{t \geq \tau }\| F'(\phi(t))\|_{L^p(\Omega)}
&\leq C\left( 1+ \sqrt{p} \| \mu\|_{L^\infty(\tau,\infty;H^1(\Omega))} \right)
\leq C(\tau) ( 1+ \sqrt{p} ),
\end{align*}
where the constant $C(\tau)$ is independent of $p$. Since $F''(s)\leq C \mathrm{e}^{C |F'(s)|}$ for all $s \in (-1,1)$, by arguing as in \cite[pg. 2279--2281]{GGG2023}, we obtain the desired conclusion \eqref{F''p}.
Now, by \cite[Lemma 3.2]{HW2021}, $\partial_{x_i} F'(\phi)= F''(\phi) \partial_{x_i} \phi$ in the sense of distributions (and thus almost everywhere).
As an immediate consequence, we derive that
\begin{align*}
\esssup_{t \geq \tau } \| F'(\phi(t))\|_{W^{1,3}(\Omega)}
&\leq 
\| F'(\phi(t))\|_{L^\infty(\tau,\infty; L^{3}(\Omega))}
\\
& \quad + \| F''(\phi(t))\|_{ L^\infty(\tau,\infty; L^{6}(\Omega))}
\| \nabla \phi(t)\|_{L^\infty(\tau,\infty; L^{6}(\Omega))}
\leq  C(\tau).
\end{align*}
As a consequence, we obtain that $\| F'(\phi)\|_{L^\infty( \Omega \times (\tau,\infty) )} \leq C(\tau)$. 
Then, we deduce from $\displaystyle \lim_{s\to \pm 1} |F'(s)| = \infty$ that there exists $\delta>0$ such that  \eqref{RS2} holds.

\smallskip
We now go back to the elliptic problem \eqref{Aphi-ELL} for $A(\phi)$ with right-hand side given by
$$
f= \frac{1}{\sqrt{a(\phi)}} \left( \mu- \Psi'(\phi) \right) + A(\phi).
$$
It is immediate to observe that $f \in L^\infty(\tau,\infty; L^2(\Omega))$. Moreover, we have
$$
\partial_{x_i}f = - \frac12
\frac{a'(\phi)}{a^\frac32(\phi)}  \left( \mu- \Psi'(\phi) \right) \partial_{x_i} \phi
+
\frac{1}{\sqrt{a(\phi)}} \left( \partial_{x_i} \mu- \Psi''(\phi) \partial_{x_i} \phi\right) + \sqrt{a(\phi)} \partial_{x_i} \phi.
$$
By exploiting \eqref{a-ndeg}, we immediately obtain
\begin{align}
\|\partial_{x_i}f  \|_{L^2(\Omega)}
& \leq
\left\| 
\frac{a'(\phi)}{2a^\frac32(\phi)} \right\|_{L^\infty(\Omega)} \| \mu- \Psi'(\phi)\|_{L^2(\Omega)} \| \partial_{x_i} \phi\|_{L^\infty(\Omega)}
+ \left\| \frac{1}{\sqrt{a(\phi)}}\right\|_{L^\infty(\Omega)} \| \mu\|_{H^1(\Omega)}
\notag \\
&\quad +  \left\| \frac{1}{\sqrt{a(\phi)}}\right\|_{L^\infty(\Omega)}  \| \Psi''(\phi)\|_{L^2(\Omega)} \| \partial_{x_i} \phi\|_{L^\infty(\Omega)} 
+ \| \sqrt{a(\phi)}\|_{L^\infty(\Omega)} \| \partial_{x_i} \phi\|_{L^2(\Omega)}.
\label{fdeltai}
\end{align}
Hence, by \eqref{Glob_muH1_4}, \eqref{phi2p} and \eqref{F''p}, we deduce that
$$
\| f \|_{L^\infty(\tau,\infty; H^1(\Omega))}\leq C(\tau),
$$
and by the elliptic regularity theory for \eqref{Aphi-ELL}, we get
\begin{equation}
\label{Aphi-H3}
\| A(\phi) \|_{L^\infty( \tau,\infty; H^3(\Omega))} \leq C(\tau).
\end{equation}
Finally, since $\phi$ solves the elliptic problem \eqref{Aphi-ELL-0-1} with $u=\phi$, we deduce from \eqref{h1}-\eqref{h2}, \eqref{phi2p} and \eqref{Aphi-H3} that
\begin{equation*}
\| \phi\|_{L^\infty(\tau,\infty; H^3(\Omega))}\leq C(\tau).
\end{equation*}

\subsection{Energy equality} Let $\phi$ be the weak solution originating from $\phi_0 \in \mathcal{V}_m$. For any $\tau>0$, we recall that $\phi$ satisfies the properties \eqref{RS1}-\eqref{RS3}. Thus, an application of \cite[Lemma 2.6]{SP2013} entails that
\begin{equation}
\label{EE-s}
E(\phi(t))+ \int_\tau^t
\left\| \sqrt{b(\phi(s))} \nabla \mu(s)\right\|_{L^2(\Omega)}^2 \, \d s =  E (\phi(\tau)), \quad \forall \, t \geq \tau,
\end{equation}
for any $\tau>0$. In order to prove the claimed energy equality \eqref{EI}, we are left to show the validity of \eqref{EE-s} for $\tau=0$.
To this aim, we first observe that the gradient part of the energy can be rewritten in terms of $A$, namely, for any $\phi\in \mathcal{V}_m$
\begin{equation}
\label{LSC-4}
E(\phi) = \int_\Omega \frac{1}{2} |\nabla A(\phi)|^2 + \Psi(\phi) \,\d x.
\end{equation}
Besides, since $\left\| \sqrt{b(\phi(s))} \nabla \mu(s)\right\|_{L^2(\Omega)}^2 \in L^1(0,\infty)$, we can pass to the limit as $\tau \to 0$ in \eqref{EE-s},
which leads to
\begin{equation}
\label{EE-s-2}
E(\phi(t))+ \int_0^t
\left\| \sqrt{b(\phi(s))} \nabla \mu(s)\right\|_{L^2(\Omega)}^2 \, \d s = \lim_{\tau \to 0}  E (\phi(\tau)), \quad \forall \, t \geq 0.
\end{equation}
Owing to the energy inequality \eqref{EE}, we thus deduce that
\begin{equation}
\label{EE-s-3}
\lim_{\tau \to 0}  E (\phi(\tau)) \leq E(\phi_0).
\end{equation}
On the other hand, it follows from \eqref{WS1} and \eqref{WS3}
that $\phi \in BC_{\rm w}([0,\infty); H^1(\Omega)) \cap C([0,\infty); L^2(\Omega))$. Then, we have
\begin{equation}
\label{LSC-1}
\lim_{\tau \to 0} \| \phi(\tau)\|_{L^2(\Omega)}^2
= \| \phi_0\|_{L^2(\Omega)}^2,
\end{equation}
and
\begin{equation}
\label{LSC-1bis}
\lim_{\tau \to 0} \| A(\phi(\tau)) \|_{L^2(\Omega)}^2
= \| A(\phi_0) \|_{L^2(\Omega)}^2.
\end{equation}
By the convexity of $F$, it also follows that
\begin{equation}
\label{LSC-2}
\liminf_{\tau \to 0} \int_{\Omega} F(\phi(\tau)) \, \d x\geq \int_{\Omega} F(\phi_0) \, \d x.
\end{equation}
In addition, recalling that $A(\phi) \in L^\infty(0,\infty; H^1(\Omega))$, by the lower-semicontinuity of the norm and \eqref{LSC-1bis},
we derive that there exists a sequence $\lbrace \tau_n \rbrace_{n=0}^\infty$ such that $\tau_n \to 0$ as $n \to \infty$ and
\begin{equation}
\label{LSC-3}
\liminf_{n \to \infty} \int_\Omega |\nabla A(\phi(\tau_n))|^2 \, \d x
\geq \int_\Omega |\nabla A(\phi_0)|^2 \, \d x.
\end{equation}
By collecting \eqref{LSC-4} and \eqref{LSC-1}-\eqref{LSC-3} together,
we obtain
\begin{equation}
\label{LSC-5}
\liminf_{n \to \infty}  E (\phi(\tau_n)) \geq E(\phi_0).
\end{equation}
Thus, in light of \eqref{EE-s-3} and \eqref{LSC-5}, we conclude that
\begin{equation}
\label{EE-0}
\lim_{\tau \to 0}  E (\phi(\tau)) = E(\phi_0),
\end{equation}
which, by \eqref{EE-s-2}, proves the energy identity \eqref{EI}.

\smallskip Let us now focus on the continuity properties of $\phi$ with respect to time. By \eqref{EI}, we immediately have that $E(\phi(\cdot)) \in AC([0,T])$, for any $T>0$. Hence, for any $t,s\in [0,T]$ we find that
\begin{align*}
\limsup_{t \to s} \int_\Omega \frac12 |\nabla A(\phi(t))|^2 \, \d x
&= \limsup_{t \to s} \left( E(\phi(t))- \int \Psi(\phi(t)) \, \d x\right)
= E(\phi(s)) - \liminf_{t \to s} \int \Psi(\phi(t)) \, \d x
\\
&\leq E(\phi(s)) - \int \Psi(\phi(s)) \, \d x= \int_\Omega \frac12 |\nabla A(\phi(s))|^2 \, \d x.
\end{align*}
Therefore, since
$$
\liminf_{t \to s} \int_\Omega |\nabla A(\phi(t))|^2 \, \d x\geq
\int_\Omega |\nabla A(\phi(s))|^2 \, \d x,
$$
we conclude that
\begin{equation}
\label{Aphi-C}
\|\nabla A(\phi(\cdot))\|_{L^2(\Omega)}\in C([0,T]), \quad \forall \, T\geq 0.
\end{equation}
By the equality
\begin{align*}
&\nabla \phi(t) -\nabla \phi(s)=
\frac{1}{\sqrt{a(\phi(t))}}\nabla A(\phi(t)) -\frac{1}{\sqrt{a(\phi(s))}}\nabla A(\phi(s)),
\end{align*}
adding and subtracting the term $\frac{1}{\sqrt{a(\phi(t))}}\nabla A(\phi(s))$, 
we infer that
\begin{align*}
\left\|\nabla \phi(t)-\nabla \phi(s)\right\|_{L^2(\Omega)}^2
&\leq
\frac{2}{a_m}\left\|\nabla A(\phi(t))-\nabla A(\phi(s))\right\|^2_{L^2(\Omega)}\\
&\quad +
2\int_\Omega\left(\frac{1}{\sqrt{a(\phi(t))}}-\frac{1}{\sqrt{a(\phi(s))}}\right)^2|\nabla A(\phi(s))|^2\,\d x.
\end{align*}
Now notice that, if we take any sequence $\lbrace \tau_n \rbrace_{n=0}^\infty$ such that $\tau_n \to s$ as $n \to \infty$, by the dominated convergence theorem,
the second term on the right-hand
side converges to 0 as $n\to\infty$. Indeed, since $\phi \in C([0,\infty); L^2(\Omega))$, we have
$$f_n:=\left(\frac{1}{\sqrt{a(\phi(\tau_n))}}-\frac{1}{\sqrt{a(\phi(s))}}\right)^2|\nabla A(\phi(s))|^2\to 0\quad \text{a.e. in } \Omega,$$
as $n\to\infty$, and $|f_n|\leq \frac{2}{a_m}|\nabla A(\phi(s))|^2\in L^1(\Omega)$.
At this point we easily conclude from \eqref{Aphi-C} that
$$\lim_{t\to s}\left\|\nabla \phi(t)-\nabla \phi(s)\right\|_{L^2(\Omega)}^2=0,$$
yielding the conclusion $\phi \in C([0,\infty); H^1(\Omega))$.
The proof of Theorem \ref{well-pos-2D}, part (4), is now completed. 

\section{Convergence to equilibrium in two dimensions}
\label{s-convergence}

This section is devoted to characterizing the longtime behavior of any solution in two dimensions. Let $\phi_0\in \mathcal{V}_m$ and let $\phi(t)$ be the corresponding global weak solution to \eqref{CH1}-\eqref{CH-ic}. Fixed $\tau>0$, since $\phi \in L^\infty(\tau,\infty; H^3(\Omega))$ and $\partial_t \phi \in L^2(0,\infty; H_{(0)}^{-1}(\Omega))$, it immediately follows that $\phi \in BC_{\rm w}([\tau,\infty); H^3(\Omega)) \cap BUC([\tau,\infty); H^2(\Omega))$.
Then, the $\omega$-limit set, defined by
$$
\omega(\phi_0)= \Big\lbrace \phi' \in H^3(\Omega): \ \exists \, t_n \to \infty \text{ such that } \phi(t_n) \to \phi' \text{ in } H^2(\Omega) \Big\rbrace,
$$
is non-empty and compact. 
\begin{lemma}
For any $\phi_0\in \mathcal{V}_m$, we have $\omega(\phi_0)\subset \mathcal{S}_m.$
\end{lemma}
\begin{proof}
We introduce some terminology from the theory of dynamical systems, see e.g. \cite[Chapter 9]{CH1998}. According to Theorem \ref{well-pos-2D}, we associate problem \eqref{CH1}-\eqref{CH-ic} with the solution map
$$S(t):\mathcal{V}_m\to \mathcal{V}_m, 
\quad
z_0 \mapsto S(t)z_0= z(t), \forall \, t \geq 0,
$$
where $z(t) $ is the unique weak solution to \eqref{CH1}-\eqref{CH-bc} departing from $z_0$. This is a one-parameter family of maps satisfying $S(0)z_0=z_0$ for any $z_0 \in \mathcal{V}_m$ and the concatenation rule
$S(t+s)=S(t)S(s)$, for every $t,s \geq 0$.
In addition,owing to the energy inequality \eqref{EE}, the following properties hold:
\begin{itemize}
\item[(i)] $E(S(t)z_0)\leq E(z_0)$, for any $t\geq 0$, $\forall\, z_0\in \mathcal{V}_m$;
\smallskip

\item[(ii)] $E(S(t)z_0)=E(z_0)$ for any $t\geq 0$ implies that $z_0$ is a stationary point of the map $S(t)$, namely $S(t)z_0=z_0$, $\forall \,t\geq 0$, which means $z\in \mathcal{S}_m$ (see Section \ref{s-stationary}).
\end{itemize}
Besides, the continuous dependence estimate \eqref{UNIQ} and the regularity $BC_{\rm w}([\tau,\infty); H^3(\Omega))$, for any $\tau>0$, entail by interpolation that $S(t) \in C(H^2(\Omega); H^2(\Omega))$, for any $t >0$. 

We are now ready to prove the result. To this end, we observe that $\phi(t)=S(t)\phi_0$. Since the map $t\mapsto E(\phi(t))$ is decreasing by (i) and bounded from below, there exists $E_\infty\in\R$ such that
\begin{equation}
\label{Einf}
E_\infty=\lim_{t\to \infty} E(\phi(t)).
\end{equation}
Next, we claim that
\begin{equation}
\label{Einf-2}
\quad E(\phi')=E_\infty, \ \forall \, \phi' \in \omega(\phi_0).
\end{equation}
Indeed, by the separation property \eqref{RS2} of the solution $\phi$, there exists $\delta>0$ such that 
$$
\|\phi(t)\|_{L^\infty(\Omega)}\leq 1-\delta, \quad \forall \, t\geq 1.
$$
By the embedding $H^2(\Omega)\hookrightarrow C(\overline{\Omega})$, we deduce that
$$
\|\phi'\|_{L^\infty(\Omega)}\leq 1-\delta,\quad \forall \, \phi'\in \omega(\phi_0).
$$
Therefore, for any $\phi'\in \omega(\phi_0)$ and $t\geq 1$, we obtain
\begin{align*}
E(\vp(t))-E(\vp')
&=\frac12\int_\Omega a(\vp(t))(|\nabla \vp(t)|^2-|\nabla\vp'|^2)\,\d x
+\frac12\int_\Omega (a(\vp(t))-a(\vp'))|\nabla\vp'|^2\,\d x
\notag
\\ 
&\quad
+\int_\Omega \Psi(\vp(t))-\Psi(\vp')\, \d x 
\notag
\\
 &\leq \frac{a_M}{2}
 \|\nabla (\vp(t)+\vp')\|_{L^2(\Omega)}
\|\nabla (\vp(t)-\vp')\|_{L^2(\Omega)}
\notag
\\
&\quad
+\max_{s\in [-1,1]}|a'(s)| \|\vp(t)-\vp'\|_{L^\infty(\Omega)}\|\nabla\vp'\|_{L^2(\Omega)}^2
+ \max_{s\in[-1+\delta,1-\delta]}|\Psi'(s)|
  \|\vp(t)-\vp'\|_{L^1(\Omega)}\nonumber\\
& \leq C \|\vp(t)-\vp'\|_{H^2(\Omega)}.
\end{align*}
Since $\phi'\in \omega(\phi_0)$, there exists $t_n\to\infty$ such that $\phi(t_n)\to \phi'$ in $H^2(\Omega)$. Thus, exploiting the above inequality, we have
$$
0\leq E(\vp(t_n))-E(\phi')\leq \gamma \|\vp(t_n)-\phi'\|_{H^2(\Omega)},
$$
which yields
$$
E(\phi')=\lim_{n\to\infty} E(\vp(t_n)).
$$
Hence, in light of \eqref{Einf}, we finally derive that $E(\phi')=E_\infty$, as claimed.
\smallskip

Let now $\phi' \in \omega(\phi_0)$ be fixed and let
$t_n\to\infty$ be such that $S(t_n)\phi_0\to \phi'$ in $H^2(\Omega)$. 
Setting $\tau_n=t+t_n$, where $t >0$ is arbitrarily fixed, by the concatenation rule and the $H^2$-continuity of the map $S(t)$, 
we obtain
$$
S(\tau_n)\phi_0=S(t)S(t_n)\phi_0\to S(t)\phi'\quad \text{in }H^2(\Omega),
$$
namely 
$$
S(t)\phi'\in \omega(\phi_0),\quad\forall \, t >0.
$$
Owing to \eqref{Einf-2}, we thus learn that
$$
E(S(t)\phi')=E_\infty=E(\phi'), \quad \forall \, t \geq 0.
$$
At this point, we exploit property (ii) to conclude that $\phi'\in \mathcal{S}_m$. By the arbitrariness of $\phi'\in \omega(\phi_0)$, the proof is complete.
\end{proof}

\medskip
We are now ready to prove the following result, which implies the claim of part (5) of Theorem \ref{well-pos-2D}. 

\begin{lemma}
For any $\phi_0\in \mathcal{V}_m$, the set $\omega(\phi_0)$ is a singleton.
\end{lemma}

\begin{proof}
Let $\phi'\in \omega(\phi_0)$. Since $\phi'\in \mathcal{S}_m$, the Lojasiewicz--Simon inequality in Theorem \ref{LSg-intro} holds true at $\phi'$, yielding
 constants $\theta\in \left(0, \frac12\right)$, $C>0$ and
$\beta >0$ such that 
\begin{equation*}
\label{Ls-1}
|E(u) -E(\phi')|^{1-\theta}\leq C_L \norm{-\div(a(u)\nabla u)+\frac{a'(u)}{2} |\nabla u|^2+\Psi'(u)-\overline{\frac{a'(u)}{2} |\nabla u|^2+\Psi'(u)}}_{L^2(\Omega)},
\end{equation*}
for all $u\in H^2(\Omega)$ such that $\overline{u}=m$, $\partial_\n u=0$ on $\partial\Omega$ and $\|u-\phi'\|_{H^2(\Omega)}\leq \beta$.

\medskip
\noindent
Exploiting the fact that $\omega(\phi_0)$ is a compact subset of $H^2(\Omega)$, we can cover $\omega(\phi_0)$ by finitely many open balls of $H^2(\Omega)$, denoted by $\lbrace B_i \rbrace_{i=1}^N\subset H^2(\Omega)$, centered at $\phi'_i \in \omega(\phi_0)$ and with radii $\beta_i$,
prescribed by Theorem \ref{LSg-intro} for $\phi'_i$.  Since $E(\phi'_i)=E_\infty$ for any $i$,
we conclude that there exist universal constants $\widetilde{\theta} \in \left(0,\frac12\right)$ and $\widetilde{C}>0$ such that
\begin{equation}
\label{Simon}
|E(u) -E_\infty|^{1-\widetilde{\theta}}\leq \widetilde{C} \norm{-\div(a(u)\nabla u)+\frac{a'(u)}{2} |\nabla u|^2+\Psi'(u)-\overline{\frac{a'(u)}{2} |\nabla u|^2+\Psi'(u)}}_{L^2(\Omega)},
\end{equation}
for all $u\in H^2(\Omega)$ such that $\overline{u}=m$, $\partial_\n u=0$ on $\partial\Omega$ and $u\in \mathcal{B}:= \bigcup_{i=1}^N B_i$.

\smallskip
Now, let $\phi(t)$ be the solution to problem \eqref{CH1}-\eqref{CH-ic}. Recalling that $\mathrm{d}_{H^2(\Omega)}(\phi(t), \omega(\phi_0)) \to 0$ as $t \to \infty$, there exists $t^\star$ such that
$\phi(t) \in \mathcal{B}\cap \mathcal{V}_m$ for all $t\geq t^\star$.  As a result, we can apply \eqref{Simon} with $u=\phi(t)$, for all $t\geq t^\star$. In which case, using the definition of $\mu$ in \eqref{CH2}
$$
\mu=-\div(a(\phi)\nabla \phi)+\frac{a'(\phi)}{2} |\nabla \phi|^2+\Psi'(\phi),
$$
we obtain from \eqref{Simon} the following inequality
\begin{equation}
\label{Simonfi}
|E(\phi(t))- E_\infty|^{1-\widetilde{\theta}}
\leq \widetilde{C}\norm{\mu(t)-\overline{\mu(t)}}_{L^2(\Omega)}\!\!, \quad
\forall \, t\geq t^\star.
\end{equation}
\medskip
Let us define the function $H: [ t^\star, \infty) \to \mathbb{R}_+$ by
$$
H(t):= \left( E(\phi(t))- E_\infty \right)^{\widetilde{\theta}}\!.
$$
By the regularity of the solution \eqref{RS1}-\eqref{RS3}, an application of \cite[Lemma 2.6]{SP2013}, together with the standard chain rule in Bochner spaces, gives
\begin{equation}
\label{Energy-eq.}
\ddt E(\phi) + \int_{\Omega} b(\phi)|\nabla \mu|^2 \, \d x =0,
\quad \text{for a.e. } t \in (1,\infty).
\end{equation}
By exploiting \eqref{poincare}, \eqref{Simonfi} and \eqref{Energy-eq.}, we obtain
\begin{align*}
- \ddt H(t)
&= - \widetilde{\theta} \left( E(\phi(t))- E_\infty \right)^{\widetilde{\theta}-1} \ddt E(\phi(t))
\\
&\geq \frac{\widetilde{\theta} }{\widetilde{C}} \frac{ \norm{ \sqrt{b(\phi(t))} \nabla \mu(t)}_{L^2(\Omega)}^2 }{\norm{\mu(t)-\overline{\mu(t)}}_{L^2(\Omega)} }
\geq \frac{\widetilde{\theta} }{C_P} \frac{b_m}{\widetilde{C} } \|\nabla \mu (t) \|_{L^2(\Omega)},
\quad \text{a.e. in } (t^\star, \infty).
\end{align*}
Now, the conclusion follows in a standard way. Integrating from $t^\star$ to $\infty$, and observing that $\displaystyle \lim_{t \to \infty} H(t)=0$, we infer that
$$
\int_{t^\star}^\infty \| \nabla \mu(t)\|_{L^2(\Omega)} \, \d t
\leq \frac{C_P \, \widetilde{C}}{\widetilde{\theta} b_m}
H(t^\star),
$$
which gives that $\nabla \mu \in L^1(t^\star, \infty; L^2(\Omega))$. Hence, by \eqref{est-phit-H-1}, we have that $\partial_t \phi \in L^1(t^\star, \infty; H^{-1}_{(0)}(\Omega))$. Since
$$
 \phi(t) = \phi(t^\star) +  \int_{t^\star}^t \partial_t \phi(\tau) \, \d \tau, \quad \forall \, t \geq t^\star\!,
$$
we conclude that $\displaystyle\lim_{t \to \infty} \phi(t)$ exists in $H^{-1}_{(0)}(\Omega)$.
Therefore, $\omega(\phi_0)=\lbrace \phi_\infty \rbrace$, where $\phi_\infty= \displaystyle \lim_{t \to \infty} \phi(t)$.
\end{proof}

\section{Local strong solutions in three dimensions}
\label{Local-Strong-3D}

In this section, we study the existence and uniqueness of local strong solutions in three dimensions. We will prove the following result, which provides part (2) of Theorem \ref{well-pos-3D}.

\begin{proposition}\label{CHstrong-solution}
Let $\phi_0\in H^2(\Omega)$ be such that $\overline{\phi_0}=m \in (-1,1) $,  $\partial_\n \phi_0=0$ on $\partial \Omega$ and
\begin{equation}
\label{mu0}
\left\| -\div \left( a(\phi_0) \nabla \phi_0 \right) + a'(\phi_0) \frac{|\nabla \phi_0|^2}{2} + \Psi'(\phi_0) \right\|_{H^1(\Omega)}\leq M.
\end{equation}
Then, there exist $T_M>0$, depending on $M$, $E(\phi_0)$ and $m$, and a unique strong solution $\phi$ to \eqref{CH1}-\eqref{CH-ic} on $[0,T_M]$ such that 
\begin{equation}
\label{SS-phi}
\begin{split}
&\phi \in L^\infty(0,T_M; W^{2,6}(\Omega)), \quad
\partial_t \phi \in L^2(0 ,T_M;H^1(\Omega)),\\
&\phi \in L^{\infty}(\Omega\times (0,T_M)) \quad\text{such that}\quad |\phi(x,t)|<1
\ \text{a.e. }(x,t)\in \Omega\times (0,T_M),\\
&\mu \in L^\infty(0,T_M; H^1(\Omega)),
\quad F'(\phi) \in L^\infty(0,T_M; L^6(\Omega)).
\end{split}
\end{equation}
The solution fulfills the problem \eqref{CH1}-\eqref{CH2} almost everywhere in $\Omega \times (0, T_M)$.
Furthermore, we have the following estimates
\begin{align}
\label{strong est 1}
&\norm{\partial_t \phi}_{L^2(0,T_M; H^1(\Omega))}
+ \norm{\partial_t \phi}_{L^{\infty}(0, T_M; H^{-1}_{(0)}(\Omega))}
\le C_M,
\\
\label{strong est 3}
&\norm{\phi}_{L^{\infty}(0,T_M; W^{2,6}(\Omega))} +
\norm{F'(\phi)}_{L^{\infty}(0, T_M; L^6(\Omega))}
\le C_M,
\\
\label{strong est 2}
& \norm{\mu}_{L^{\infty}(0,T_M; H^1(\Omega))}
\le C_M.
\end{align}
In addition, if $\| \phi_0\|_{L^\infty(\Omega)}\leq 1- \delta_0$, then there exists $\widetilde{T_M}\in (0,T_M]$, depending on $\delta_0$, but independent of the specific initial datum, such that
$\phi \in L^\infty(0,\widetilde{T_M}; H^3(\Omega))$ and
\begin{align}
\label{strong est 4}
&\max_{t \in [0,\widetilde{T_M} ]} \| \phi (t)\|_{C(\overline{\Omega})} \leq 1-\frac{\delta_0}{2}.
\end{align}
\end{proposition}

\begin{proof}
First of all, we collect the basic energy estimates provided in Section \ref{S-WEAK}. Denoting by $C>0$ a generic constant only depending on $E(\phi_0)$ and $m$, the following estimates hold true:
\begin{equation}
\| \phi\|_{L^\infty( 0,\infty; H^1(\Omega))} \leq C, \quad
\int_0^\infty \| \nabla \mu (\tau) \|_{L^2(\Omega)}^2 \, \d \tau \leq C.
\end{equation}
Besides,
\begin{equation}
\label{est-mu-H1-3D-d}
\| \mu\|_{H^1(\Omega)}\leq C \| \nabla \mu\|_{L^2(\Omega)} + C,
\end{equation}
and
\begin{equation}
\label{est-A-H2-3D-d}
\| \phi\|_{H^2(\Omega)}^2  + \| A(\phi)\|_{H^2(\Omega)}^2 \leq C \| \nabla \mu\|_{L^2(\Omega)} + C,
\end{equation}
as well as
\begin{equation}
\label{est-phit-H-1-3D-d}
\| \partial_t \phi\|_{ H^{-1}_{(0)}(\Omega)} \leq C \| \nabla \mu\|_{L^2(\Omega)}.
\end{equation}

\medskip
We now perform higher-order energy estimates by relying on the assumption \eqref{mu0}. As in Section \ref{S-Strong}, this argument can be rigorously justified within the approximation procedure devised in \cite{SP2013} (see also \cite[Section 4]{GMT2019} for the approximation of the initial datum).
By reasoning as in the two dimensional case (cf. \eqref{first-rel}), we have
\begin{align}
&\frac{1}{2} \ddt  \int_{\Omega} b(\phi)\left|\nabla \mu \right|^2 \, \d x
+\int_\Omega a(\phi) |\nabla \partial_t \phi|^2 \, \d x
+ \int_\Omega \frac{(a'(\phi))^2}{4 a(\phi)} |\nabla \phi|^2 |\partial_t \phi|^2 \, \d x + \int_\Omega F''(\phi)|\partial_t \phi|^2 \, \d x
\notag
\\ 
&= \theta_0 \| \partial_t \phi\|_{L^2(\Omega)}^2+ \frac{1}{2}  \int_{\Omega} b'(\phi) \partial_t \phi  |\nabla \mu|^2 \, \d x + \int_\Omega \frac{a'(\phi)}{2\sqrt{a(\phi)}} \Delta A(\phi) |\partial_t \phi|^2 \, \d x
\notag 
\\
&\quad
-  \int_\Omega a'(\phi) \nabla \phi \cdot \nabla \partial_t \phi \, \partial_t \phi \, \d x.
\label{first-rel-3D}
\end{align}
We control the last two terms on the right hand side. By exploiting \eqref{lady_3}, \eqref{est-A-H2-3D-d} and \eqref{est-phit-H-1-3D-d},  we find
\begin{align*}
\left| \int_\Omega \frac{a'(\phi)}{2\sqrt{a(\phi)}} \Delta A(\phi) |\partial_t \phi|^2 \, \d x \right|
&
\leq \left\| \frac{a'(\phi)}{2\sqrt{a(\phi)}} \right\|_{L^\infty(\Omega)}
\| \Delta A(\phi)\|_{L^2(\Omega)} \| \partial_t \phi\|_{L^6(\Omega)}\| \partial_t \phi\|_{L^3(\Omega)}
\\
&\leq
C \| \Delta A(\phi)\|_{L^2(\Omega)}
 \norm{\nabla \partial_t \phi}_{L^2(\Omega)}^{\frac32}\| \partial_t \phi\|_{L^2(\Omega)}^\frac12
\\
&\leq
C \| \Delta A(\phi)\|_{L^2(\Omega)}
 \norm{\nabla \partial_t \phi}_{L^2(\Omega)}^\frac74
\| \partial_t \phi\|_{H^{-1}_{(0)}(\Omega)}^\frac14
\\
&\leq
\frac{a_m}{12} \norm{\nabla \partial_t \phi}_{L^2(\Omega)}^2 +
C \| A(\phi)\|_{H^2(\Omega)}^8  \| \nabla \mu\|_{L^2(\Omega)}^2
\\
& \leq
\frac{a_m}{12} \norm{\nabla \partial_t \phi}_{L^2(\Omega)}^2 +
C\left( 1+ \| \nabla \mu\|_{L^2(\Omega)}^4\right)
 \| \nabla \mu\|_{L^2(\Omega)}^2.
\end{align*}
Similarly, we also obtain
\begin{align*}
\left|  \int_\Omega a'(\phi) \nabla \phi \cdot \nabla \partial_t \phi \, \partial_t \phi \, \d x  \right|
&\leq
\| a'(\phi)\|_{L^\infty(\Omega)} 
 \| \nabla \phi\|_{L^6(\Omega)}\norm{\nabla \partial_t \phi}_{L^2(\Omega)}
\| \partial_t \phi\|_{L^3(\Omega)}
\\
&\leq
C  \|\phi\|_{H^2(\Omega)} \norm{\nabla \partial_t \phi}_{L^2(\Omega)}^\frac32 
\| \partial_t \phi\|_{L^2(\Omega)}^\frac12
\\
&\leq
C \| \phi\|_{H^2(\Omega)}  \norm{\nabla \partial_t \phi}_{L^2(\Omega)}^\frac{7}{4}
\| \partial_t \phi\|_{H^{-1}_{(0)}(\Omega)}^\frac14
\\
&
\leq
\frac{a_m}{12} \norm{\nabla \partial_t \phi}_{L^2(\Omega)}^2 +
C \| \phi\|_{H^2(\Omega)}^{8}  \| \nabla \mu\|_{L^2(\Omega)}^2
\\
& \leq
\frac{a_m}{12} \norm{\nabla \partial_t \phi}_{L^2(\Omega)}^2 +
C \left( 1+ \| \nabla \mu\|_{L^2(\Omega)}^4\right)
 \| \nabla \mu\|_{L^2(\Omega)}^2.
\end{align*}
We now control the nonlinear term containing $b'$. By Lebesgue interpolation, noticing that $\frac{5}{12}=\frac{a}{2}+\frac{1-a}{6}$ with $a=\frac34$, we have
\begin{align}
\left| \int_{\Omega} b'(\phi )\partial_t \phi |\nabla \mu |^2 \, \d x\right|
& \le C \norm{ \nabla\partial_t \phi }_{L^2(\Omega)}
\norm{\nabla \mu}_{L^\frac{12}{5}(\Omega)}^2
\notag
\\
&\leq \frac{a_m}{24}\norm{ \nabla\partial_t \phi }_{L^2(\Omega)}^2+C\norm{\nabla \mu}_{L^\frac{12}{5}(\Omega)}^4
\notag
\\
&\leq \frac{a_m}{24}\norm{ \nabla\partial_t \phi }_{L^2(\Omega)}^2+C\norm{\nabla \mu}_{L^2(\Omega)}^3\norm{\mu-\overline{\mu}}_{H^2(\Omega)}.
\label{third_method_2}
\end{align}
Owing to \cite[eq. (2.26)]{BB1999} and \eqref{est-phit-H-1-3D-d}, we know that
\begin{align*}
\| \mu - \overline{\mu}\|_{H^2(\Omega)}
&= \norm{\G_{\phi} \partial_t \phi}_{H^2(\Omega)}
\\
&\leq
C \left(
\| \phi\|_{H^2(\Omega)}^2 \| \nabla \mathcal{G}_{\phi} \partial_t \phi\|_{L^2(\Omega)}+ \| \partial_t \phi\|_{L^2(\Omega)} \right)
\\
&\leq C \left(  \| \phi\|_{H^2(\Omega)}^2 \| \nabla \mu\|_{L^2(\Omega)} + \| \nabla \mu\|_{L^2(\Omega)}^\frac12 \| \nabla \partial_t \phi\|_{L^2(\Omega)}^\frac12  \right)\!.
\end{align*}
Thus, we derive that
\begin{align*}
\left| \int_{\Omega} b'(\phi )\partial_t \phi |\nabla \mu |^2 \, \d x\right|
&\leq \frac{a_m}{24}\norm{ \nabla\partial_t \phi }_{L^2(\Omega)}^2
+C\| \phi\|_{H^2(\Omega)}^2\norm{\nabla \mu}_{L^2(\Omega)}^4
+C\| \nabla \mu\|_{L^2(\Omega)}^\frac72 \| \nabla \partial_t \phi\|_{L^2(\Omega)}^\frac12
\\
&\leq \frac{a_m}{12}\norm{ \nabla\partial_t \phi }_{L^2(\Omega)}^2
+C\| \phi\|_{H^2(\Omega)}^2\norm{\nabla \mu}_{L^2(\Omega)}^4+
C \| \nabla \mu\|_{L^2(\Omega)}^\frac{14}{3}
\\
&\leq \frac{a_m}{12}\norm{ \nabla\partial_t \phi }_{L^2(\Omega)}^2+
C \left( 1+ \| \nabla \mu\|_{L^2(\Omega)}^4\right)
 \| \nabla \mu\|_{L^2(\Omega)}^2.
\end{align*}
Lastly, we recall that
$$
\theta_0 \|\partial_t \phi \|_{L^2(\Omega)}^2 
\leq \frac{a_m}{4} \| \nabla \partial_t \phi\|_{L^2(\Omega)}^2 +
C \| \nabla \mu\|_{L^2(\Omega)}^2.
$$
In light of \eqref{m-ndeg}, collecting all the above inequalities in \eqref{first-rel-3D},  and setting $\mathcal{Y}(t):=\int_{\Omega} b(\phi(t))\left|\nabla \mu(t) \right|^2 \, \d x$, we end up with
\begin{align}
\label{final-3D}
& \ddt  \mathcal{Y}(t)
 + a_m
 \norm{\nabla \partial_t \phi}_{L^2(\Omega)}^2
\leq C_0\left(1+  \mathcal{Y}^2(t)
 \right) \mathcal{Y}(t)
\end{align}
for some $C_0$, depending only $E(\phi_0)$, $m$ and the parameters of the system.
By classical ODE's comparison principles (see, e.g., \cite[Lemma II.4.12]{BF2013}), we conclude that
\begin{equation*}
\sup_{0 \leq t \leq T} \|\nabla \mu (t)\|_{L^2(\Omega)}^2
\leq  \frac{1}{b_m} \frac{\mathcal{Y}(0)}{\sqrt{\left( 1+\mathcal{Y}^2(0) \right) \mathrm{e}^{-2C_0 T} -\mathcal{Y}^2(0) }},
\end{equation*}
where
$$
0 < T < T^\star:= \frac{1}{2C_0} \ln \left( 1 + \frac{1}{\mathcal{Y}^2(0)}\right).
$$
Since $
\mu(0)= -\div( a(\phi_0) \nabla \phi_0) + a'(\phi_0) \frac{|\nabla \phi_0|^2}{2}+\Psi'(\phi_0)
$ and \eqref{mu0}, we observe that 
$$
\mathcal{Y}(0)\leq b_M \| \nabla \mu_0\|_{L^2(\Omega)}^2\leq b_M M^2.
$$
Noticing that the function $g: [0,b_M M^2] \to \mathbb{R}$, $s \mapsto (1+s^2)\mathrm{e}^{- 2C_0 T} - s^2$, which is decreasing, and so it has its minimum value at $b_M M^2$. Therefore, setting
$$
0<T_M:= \frac{1}{2C_0} \ln \left( 1+ \frac{3}{1+4b_M^2 M^4}\right) < T^\star,
$$
we derive that
\begin{equation}
\label{mu-TM}
\sup_{0 \leq t \leq T_M} \|\nabla \mu (t)\|_{L^2(\Omega)}^2
\leq
\frac{2 b_M M^2}{b_m}.
\end{equation}

Now, it directly follows from \eqref{est-mu-H1-3D-d} and \eqref{mu-TM} that
\begin{equation}
\label{Glob_muH1_4-3D}
\| \mu\|_{L^\infty(0, T_M; H^1(\Omega))} \leq C_M.
\end{equation}
As a consequence, we learn from \eqref{est-phit-H-1-3D-d} that
\begin{equation}
\label{phit-inf-3D}
 \| \partial_t \phi \|_{L^\infty(0,T_M;  H^{-1}_{(0)}(\Omega))}\leq C_M.
\end{equation}
Then, integrating \eqref{final-3D} on the time interval $[0,T_M]$, we deduce that
\begin{equation}
\label{phit-H1-3D}
 \int_0^{T_M} \| \nabla \partial_t \phi (s)\|_{L^2(\Omega)}^2 \, \d s
\leq C_M.
\end{equation}
On the other hand, by \eqref{est-A-H2-3D-d}, it is easily seen that
\begin{equation}
\label{phiH2-3D}
\| \phi\|_{L^\infty(0,T_M; H^2(\Omega))} \leq C_M.
\end{equation}
Furthermore, by reasoning as in Section \ref{S-WEAK} (cf. \eqref{F'-LP}, \eqref{A2p} and \eqref{phi-W2p}) and exploiting the uniform bounds on $[0,T_M]$, we conclude that
$$\| F'(\phi)\|_{L^\infty(0,T_M; L^6(\Omega))}\leq C_M,$$
and
\begin{equation}
\label{Aphi-W26}
\| A(\phi)\|_{L^\infty(0,T_M; W^{2,6}(\Omega))}+\| \phi\|_{L^\infty(0,T_M; W^{2,6}(\Omega))}
\leq C_M.
\end{equation}

\bigskip
\textbf{Local separation property.}
Next, we consider an initial datum $\phi_0$ satisfying 
$$
\| \phi_0\|_{L^\infty(\Omega)}\leq 1- \delta_0,
$$ 
and we prove the separation property \eqref{strong est 4} on an interval $[0,\widetilde{T_M}]$, where $\widetilde{T_M}$ is possibly smaller than $T_M$, depending on $\delta_0$ but independent of the particular initial datum. 
First of all, we observe that $\phi\in BC_{\rm w}([0,T_M]; W^{2,6}(\Omega))$ and $\| \phi(t)\|_{W^{2,6}(\Omega)} \leq C_M$, for any $t\in [0,T_M]$. Therefore, for any pair $t_1<t_2\in [0,T_M]$, owing to the uniform bound
\eqref{phit-H1-3D}, we have
\begin{align*}
\|\phi(t_1)-\phi(t_2)\|_{W^{1,4}(\Omega)}&\leq \|\phi(t_1)-\phi(t_2)\|_{H^1(\Omega)}^\frac{1}{4}\|\phi(t_1)-\phi(t_2)\|_{W^{1,6}(\Omega)}^{\frac{3}{4}}\\
&\leq C_M\left(\int_{t_1}^{t_2}\|\partial_t\phi\|_{H^1(\Omega)}\right)^\frac{1}{4}\\
&\leq C_M\left(\int_{0}^{T_M}\|\partial_t\phi\|_{H^1(\Omega)}^2\right)^{1/8}(t_2-t_1)^\frac{1}{8}\\
&\leq C_M|t_2-t_1|^\frac{1}{8}.
\end{align*}
In particular, by the embedding $W^{1,4}(\Omega) \hookrightarrow C(\overline{\Omega})$, and by choosing  $t_1=t\in [0,T_M]$ and $t_2=0$, we have
$$
\|\phi(t)\|_{C(\overline{\Omega})}
\leq \|\phi_0\|_{C(\overline{\Omega})}
+C_M t^\frac18
\leq 1-\delta_0+C_M t^\frac18.
$$
Thus, the desired estimate \eqref{strong est 4} follows by setting $\widetilde{T_M}=\min\{T_M,\frac{\delta_0^8 }{2^8C_M^8}\}$.
\bigskip

\textbf{$H^3$-estimate.}
We now go back to the elliptic problem \eqref{Aphi-ELL} with right-hand side
$$
f= \frac{1}{\sqrt{a(\phi)}} \left( \mu- \Psi'(\phi) \right) + A(\phi)\in L^\infty(0,T_M; L^2(\Omega)).
$$
In light of \eqref{fdeltai}, noticing that $\| \Psi''(\phi(t))\|_{L^\infty(\Omega)}\leq C_{\delta_0}$ for $t\in [0, \widetilde{T_M}]$ in light of the separation property \eqref{strong est 4}, by exploiting \eqref{Glob_muH1_4-3D} and \eqref{Aphi-W26}, we deduce that
$$
\| f \|_{L^\infty(0,\widetilde{T_M}; H^1(\Omega))}\leq C(M,\delta_0).
$$
Therefore, by the elliptic regularity theory for problem \eqref{Aphi-ELL}, we obtain
$$
\| A(\phi) \|_{L^\infty(0,\widetilde{T_M}; H^3(\Omega))} \leq C(M,\delta_0).
$$
Finally, by considering once again the elliptic problem \eqref{Aphi-ELL-0-1}, 
and exploiting the estimate \eqref{h2} with $u=\phi$, it is easily seen that
\begin{equation}
\label{fiH3}
\| \phi\|_{L^\infty(0,\widetilde{T_M}; H^3(\Omega))}\leq C(M,\delta_0).
\end{equation}
\medskip

\textbf{Uniqueness.} Let $\phi_{1}^0$ and $\phi_2^0$ be two initial data such that $\overline{\phi_{1}^0}=\overline{\phi_{2}^0}$. We consider two corresponding local strong solutions $\phi_1$ and $\phi_2$  to problem \eqref{CH1}-\eqref{CH-bc} defined on a common interval $[0,T_M]$.
Setting $\Phi=\phi_1-\phi_2$, arguing as in Section \ref{W-uniq}, we find the identity
 \begin{align*}
& \int_\Omega a(\phi_1) |\nabla \Phi |^2 \, \d x
 +
  \int_\Omega \left( F'(\phi_1)-F'(\phi_2) \right)   \Phi \, \d x
- \int_{\Omega} \left( \mu_1-\mu_2\right)  \Phi \, \d x
\\
&=
 \theta_0 \norm{\Phi}_{L^2(\Omega)}^2
 - \int_\Omega \left( a(\phi_1)-a(\phi_2)\right) \nabla \phi_2 \cdot \nabla \Phi \, \d x
 \\
 &\quad -\int_\Omega \frac{a'(\phi_1)}{2}  \left( |\nabla \phi_1|^2-|\nabla \phi_2|^2\right) \Phi
 \, \d x
 - \int_\Omega \frac12  \left( a'(\phi_1)-a'(\phi_2)\right) | \nabla \phi_2|^2 \Phi \, \d x.
\end{align*}
We observe that
\begin{align*}
- \int_{\Omega} \left( \mu_1-\mu_2\right)  \Phi \, \d x
&=
 \int_\Omega \left( \G_{\phi_1}  \partial_t \phi_1 - \G_{\phi_2} \partial_t \phi_2  \right)  \Phi \, \d x
 \\
 &=
\int_\Omega  \G_{\phi_1} \partial_t \Phi  \, \Phi \, \d x
- \int_\Omega (b(\phi_2)-b(\phi_1)) \nabla \mu_2 \cdot \nabla \G_{\phi_1} \Phi \, \d x.
\end{align*}
Next, thanks to the regularity properties of strong solutions, we notice (cf. \cite[Eqn. (2.22)]{BB1999}) that the following chain rule holds true:
\begin{align}
\int_\Omega \G_{\phi_1} \partial_t \Phi \,  \Phi \, \d x
&= \ddt \frac{1}{2} \norm{\sqrt{b(\phi_1)}\nabla \G_{\phi_1} \Phi}_{L^2(\Omega)}^2
+\frac12 \int_\Omega \partial_t \phi_1 b'(\phi_1) \nabla \G_{\phi_1} \Phi \cdot \nabla \G_{\phi_1} \Phi\,\d x.
\notag
\end{align}
Therefore, we end up with the following differential equality
\begin{align}
&\ddt \frac{1}{2} \norm{\sqrt{b(\phi_1)}\nabla \G_{\phi_1} \Phi}_{L^2(\Omega)}^2
+ \int_\Omega a(\phi_1) |\nabla \Phi |^2 \, \d x
 +
  \int_\Omega \left( F'(\phi_1)-F'(\phi_2) \right)   \Phi \, \d x
  \notag
\\
&=   \theta_0 \norm{\Phi}_{L^2(\Omega)}^2
 - \int_\Omega \left( a(\phi_1)-a(\phi_2)\right) \nabla \phi_2 \cdot \nabla \Phi \, \d x
 \notag
 \\
 &\quad -\int_\Omega \frac{a'(\phi_1)}{2}  \left( |\nabla \phi_1|^2-|\nabla \phi_2|^2\right) \Phi
 \, \d x
 - \int_\Omega \frac12  \left( a'(\phi_1)-a'(\phi_2)\right) | \nabla \phi_2|^2 \Phi \, \d x
 \notag
 \\
 &\quad
 -\frac12 \int_\Omega \partial_t \phi_1 b'(\phi_1) \nabla \G_{\phi_1} \Phi \cdot \nabla \G_{\phi_1} \Phi\,\d x
  + \int_\Omega (b(\phi_2)-b(\phi_1)) \nabla \mu_2 \cdot \nabla \G_{\phi_1} \Phi \, \d x.
  \label{DE-3D}
   \end{align}
Let us now proceed by estimating the terms on the right-hand side. It is useful to recall the following control from Proposition \ref{NP-nc}
$$
\| \G_{\phi_1} \Phi\|_{H^2(\Omega)}
\leq C \| \Phi\|_{L^2(\Omega)}.
$$
By exploiting the interpolation in Lebesgue spaces, we obtain
\begin{align*}
-\frac12 \int_\Omega \partial_t \phi_1 b'(\phi_1) \nabla \G_{\phi_1} \Phi \cdot \nabla \G_{\phi_1} \Phi\,\d x
&\leq C\|\partial_t\phi_1\|_{L^6(\Omega)}\|\nabla \G_{\phi_1} \Phi\|_{L^{\frac{12}{5}}(\Omega)}^2\\
&\quad\leq  C \|\partial_t\phi_1\|_{L^6(\Omega)}
\|\nabla \G_{\phi_1} 
\Phi\|_{L^2(\Omega)}^\frac32
\|\G_{\phi_1} \Phi\|_{H^2(\Omega)}^\frac12\\
&\quad \leq \frac{a_m}{16}\| \nabla \Phi\|_{L^2(\Omega)}^2
+C\|\partial_t\phi_1\|_{L^6(\Omega)}^\frac43
\|\nabla \G_{\phi_1} \Phi\|_{L^2(\Omega)}^2.
\end{align*}
Besides, since $\mu_2\in L^\infty(0,T_M;H^1(\Omega))$, the last term on the right-hand side of \eqref{DE-3D} is estimated as follows
\begin{align*}
\int_\Omega (b(\phi_2)-b(\phi_1)) \nabla \mu_2 \cdot \nabla \G_{\phi_1} \Phi \, \d x
 &\leq  \norm{\nabla \mu_2}_{L^2(\Omega)} \norm{b(\phi_1)-b(\phi_2)}_{L^6(\Omega)}
  \norm{\nabla \mathcal{G}_{\phi_1}\Phi}_{L^3(\Omega)}
\\
&\leq
C \| \nabla\Phi\|_{L^2(\Omega)} \|\nabla \G_{\phi_1} \Phi\|_{L^2(\Omega)}^\frac12\|\G_{\phi_1} \Phi\|_{H^2(\Omega)}^\frac12\\
&\leq
C \| \nabla\Phi\|_{L^2(\Omega)}^\frac32 \|\nabla \G_{\phi_1} \Phi\|_{L^2(\Omega)}^\frac12\\
 &\leq \frac{a_m}{16}  \| \nabla \Phi\|_{L^2(\Omega)}^2
 + C\| \nabla \mathcal{G}_{\phi_1} \Phi\|_{L^2(\Omega)}^2.
 \end{align*}
By exploiting \eqref{lady_3} and the regularity $\phi_1,\phi_2\in L^\infty(0,T_M; W^{2,6}(\Omega))$, we control the terms containing the nonlinear diffusion $a$ as 
\begin{align*}
\left|  - \int_\Omega \left( a(\phi_1)-a(\phi_2)\right) \nabla \phi_2 \cdot \nabla \Phi \, \d x \right|
&\quad\leq \| a(\phi_1)-a(\phi_2)\|_{L^3(\Omega)} \| \nabla \phi_2\|_{L^6(\Omega)}
\| \nabla \Phi\|_{L^2(\Omega)}
\\
&\quad\leq C \| \Phi\|_{L^3(\Omega)}
\| \nabla \Phi\|_{L^2(\Omega)}
\\
&\quad\leq C \| \Phi\|_{H^{-1}_{(0)}(\Omega)}^\frac14
 \| \nabla \Phi\|_{L^2(\Omega)}^\frac74
\\
&\quad \leq \frac{a_m}{16}\| \nabla \Phi\|_{L^2(\Omega)}^2
+C  \| \Phi\|_{H^{-1}_{(0)}(\Omega)}^2,
\end{align*}
and
\begin{align*}
\left| -\int_\Omega \frac{a'(\phi_1)}{2}  \left( |\nabla \phi_1|^2-|\nabla \phi_2|^2\right)\Phi
\, \d x\right|
&\leq \frac12 \| a'(\phi_1)\|_{L^\infty(\Omega)}\|\nabla\phi_1+\nabla\phi_2\|_{L^6(\Omega)}  \| \nabla \Phi\|_{L^2(\Omega)} \| \Phi\|_{L^3(\Omega)}
\\
&\leq C\|\Phi\|_{H^{-1}_{(0)}(\Omega)}^{\frac14}\|\nabla \Phi\|^{\frac74}_{L^2(\Omega)}
\\
&\leq \frac{a_m}{16}\|\nabla \Phi\|^{2}_{L^2(\Omega)}+C\|\Phi\|_{H^{-1}_{(0)}(\Omega)}^{2}.
 \end{align*}
Besides, we have
 \begin{align*}
\left| - \int_\Omega \frac12  \left( a'(\phi_1)-a'(\phi_2)\right) | \nabla \phi_2|^2 \Phi \, \d x\right|
&\leq C\int_\Omega   | \nabla \phi_2|^2 \Phi^2 \, \d x
\\
&\leq C \| \nabla \phi_2\|_{L^6(\Omega)}^2
\| \Phi\|_{L^3(\Omega)}^2
\\
&\leq C
\|\Phi\|_{H^{-1}_{(0)}(\Omega)}^{\frac12}
\|\nabla \Phi\|^{\frac32}
\\
&\leq \frac{a_m}{16}\|\nabla \Phi\|^{2}_{L^2(\Omega)}+C\|\Phi\|_{H^{-1}_{(0)}(\Omega)}^{2}.
\end{align*}
Collecting all the estimates above, observing that $(F'(\phi_1)-F'(\phi_2), \phi_1-\phi_2) \geq 0$, $a(\phi_1)\geq a_m>0$, and recalling that
$\|\Phi\|_{H^{-1}_{(0)}(\Omega)}$ is equivalent to $\norm{\sqrt{b(\phi_1)} \nabla \G_{\phi_1} \Phi}_{L^2(\Omega)}$, we deduce that
\begin{equation}
\label{diff-rel-3}
\ddt  \norm{\sqrt{b(\phi_1)}\nabla \G_{\phi_1} \Phi}_{L^2(\Omega)}^2
\leq h(t) \norm{ \sqrt{b(\phi_1)} \nabla \G_{\phi_1} \Phi}_{L^2(\Omega)}^2
\! ,
\end{equation}
where $$
h(\cdot)=C\left( 1+\|\partial_t\phi_1\|_{L^6(\Omega)}^\frac43\right)\in L^1(0,T_M).
$$ 
Therefore, the uniqueness follows by the Gronwall lemma.
\end{proof}

\section{Lyapunov Stability and Longtime Behavior in three dimensions}
\label{S-Lyapunov}
\setcounter{equation}{0}

%

In this section we show that if the initial datum $\vp_0$ is sufficiently close to a local minimizer of the energy functional $E$,  then the local strong solution provided by Theorem \ref{CHstrong-solution} is indeed a global one, and $\vp$ will stay close to that minimizer for all $t \geq 0$. Finally, we investigate the longtime behavior of such global solutions as $t\to+\infty$.
We start by providing a characterization of energy minimizers.

\subsection{Energy minimizers}
In the sequel, let $m\in (-1,1)$ be fixed. We consider local minimizers $\psi$ of the total energy $E$ subject to the mass constraint, namely 
$\psi\in \mathcal{V}_m$ such that
$$
E(\psi) \leq E(u), \quad \forall \, u\in \mathcal{V}_m\ : \ \|u-\psi\|_{H^1(\Omega)}<\nu,
$$
for some $\nu>0$. 
As a first result, we show the connection between local minimizers and the stationary system  \eqref{SSS-i}-\eqref{SSS-ii}.
\begin{lemma}
\label{chmini}
Let $\psi\in \mathcal{V}_m$ be a local energy minimizer of $E$ on $\mathcal{V}_m$. Then, $\psi \in \mathcal{S}_m$. 
\end{lemma}
\begin{proof}
Let $\eps>0$ and consider the viscous Cahn--Hilliard equation with nonlinear diffusion
\begin{align}
 \label{CH1e}
\partial_t \varphi &= \Delta \mu,\\
 \label{CH2e}
\mu &= -\div \left( a(\varphi) \nabla \varphi \right) + a'(\varphi) \frac{|\nabla \varphi|^2}{2}+\Psi'(\varphi)+\eps \partial_t \varphi
\end{align}
in $\Omega \times (0,\infty)$, equipped with the boundary and initial conditions
\begin{equation}
\label{CH3e}
\partial_\n\varphi= \partial_\n \mu=0 \quad \text{on }  \partial \Omega \times (0,\infty),
\qquad  \varphi|_{t=0}=\varphi_0 \quad \text{in } \Omega.
\end{equation}
It has been proved in \cite{SP2013} (see Theorem 5.1 and Proposition 6.3) that for both $d=2,3$ and any $\varphi_0 \in H^1(\Omega)$ with $\Psi(\varphi_0)\in L^1(\Omega)$ and
$\overline{\varphi}_0\in (-1,1)$, the problem \eqref{CH1e}-\eqref{CH3e} admits a global weak solution $\varphi(t)$, which satisfies the energy equality
$$
E(\varphi(t))=E(\varphi_0)
-\int_0^t \left( \|\nabla \mu(\tau) \|_{L^2(\Omega)}^2+\eps\|\partial_t \varphi(\tau)\|_{L^2(\Omega)}^2 \right) \, \d \tau.
$$
In particular, this implies that the map $t\mapsto E(\varphi(t))$ is continuous on $[0,\infty)$. At this point, it is immediate to deduce that $\varphi\in C([0,\infty); H^1(\Omega))$ (see, e.g., \cite[pg. 482]{ABELS2009}). Now, let $\varphi_0=\psi$, where $\psi$ is a local minimizer of $E$ in $\mathcal{V}_m$ and let $\varphi(t)$ be a weak solution to the viscous problem \eqref{CH1e}-\eqref{CH3e}  departing from $\psi$. Then, by the continuity of the trajectory, there exists $\delta>0$ such that $\|\varphi(t)-\psi\|_{H^1(\Omega)}< \nu$ whenever $t\in [0,\delta)$.
Exploiting the energy equality and the fact that $E(\varphi(t))\geq E(\psi)$ on $[0,\delta)$, we have that
$$
E(\psi)\leq E(\varphi(t))=E(\psi) -\int_0^t \left( \|\nabla \mu(\tau) \|_{L^2(\Omega)}^2+\eps\|\partial_t\varphi(\tau)\|_{L^2(\Omega)}^2 \right)\, \d \tau,
$$
namely 
$$
\int_0^t \left( \|\nabla \mu(\tau) \|_{L^2(\Omega)}^2+\eps\|\partial_t\varphi(\tau)\|^2_{L^2(\Omega)} \right) \, \d \tau\leq 0, \quad \forall \, t \in [0,\delta).
$$
This implies that $\partial_t\varphi=0$ for a.e. $t\in [0,\delta)$. That is, $\varphi(t)\equiv \psi$ for any $t\in [0,\delta)$.
As a result, since $\varphi\in L^2([0,\infty);H^2(\Omega))$, we deduce that  $\psi\in H^2(\Omega)$. Besides, $\|\nabla \mu\|_{L^2(\Omega)}=0$  a.e. $t\in [0,\delta)$, meaning that $\mu$ is constant. Therefore, $\psi$ is a solution to the stationary system \eqref{SSS-i}-\eqref{SSS-ii} and belongs to $\mathcal{S}_m$.
\end{proof}

Let $\psi$ be a local minimizer. Since $\psi\in \mathcal{S}_m$, as a consequence of Proposition \ref{esci} we know that $\psi$ is strictly separated from the pure states, namely $\|\psi\|_{L^\infty(\Omega)}<1$. Moreover, we shall exhibit an explicit $L^\infty$-bound for $\psi$ that, besides $m$, only depends on the shape of the nonlinearity $\Psi$. This is done in the spirit of \cite{GM}. In particular, if $m$ is sufficiently close to $\pm 1$, the only minimizer of $E$ is the constant function $\psi\equiv m$.

\smallskip
\noindent
Let $\alpha_0 \in (-1,0)$, $\beta_0\in (0,1)$ be such that $\Psi'(\alpha_0)= \displaystyle \min_{(0,1)} \Psi'$ and $\Psi'(\beta_0)= \displaystyle \max_{(-1,0)} \Psi'$ as
in the figure below:
\begin{center}
\includegraphics[width=0.4\textwidth]{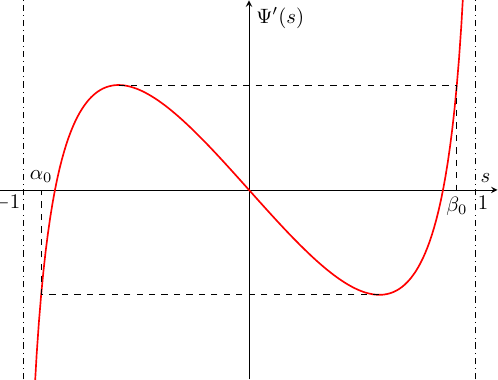}
\end{center}
Observe that there exist numbers $a,b\in (-1,1)$, $a \leq m \leq b$ such that
\begin{align}
\label{A3}
&\Psi'(r) < \Psi'(s) \quad \forall \, r \in (-1,m), \quad \forall\, s \in (b,1),
\\
\label{A3-2}
&\Psi'(r) < \Psi'(s) \quad  \forall \, r \in (-1,a), \quad \forall\, s \in (m,1).
\end{align}
It is easily seen that proper choices are
\begin{align*}
&a = b = m \, \text{ when }\, m \leq \alpha_0 \text{ or } m\geq \beta_0,
\\
&a=\alpha_0,b=\beta_0\,\text{ when }\, \alpha_0 < m < \beta_0.
\end{align*}

\begin{lemma}
\label{l-matano}
Let $\psi$ be a local minimizer of $E$ on $\mathcal{V}_m$. 
Then,
$$
a\leq \psi(x)\leq b\quad \text{for any } x\in \Omega.
$$
In particular, if  $m\leq \alpha_0$ or $m\geq \beta_0$ the only local minimizer is the constant $u(x)\equiv m$.
\end{lemma}

\begin{proof}
Let us take $a\leq m$ as in \eqref{A3-2}. Due to the mass constraint, $\psi$ remains a local minimizer when a linear function is added to $\Psi$, or equivalently, when a
constant is added to $\Psi'$. Thus, in view of the second condition in \eqref{A3-2}, we may assume that
\begin{equation}
\label{210}\Psi' < 0 \,\text{ on }(-1,a), \quad \Psi' > 0 \,\text{ on } (m, 1).
\end{equation}
Assume by contradiction that
$\psi <a$ on a set of nonzero measure.
Choose arbitrarily $\eps>0$. Then, there are nontrivial values\footnote{$c$ is a nontrivial value for $u$ if the two sets $\{x\in\Omega: u(x)<c\}$ and $\{x\in\Omega: u(x)>c\}$ have both positive measure.} $a_\eps\in (-1,a)$ and $m_\eps\in (m,1)$ such that the function $\psi$ chopped above $m_\eps$ and below $a_\eps$, namely
$$
u_\eps(x)=\begin{cases} a_\eps& \psi(x) <a_\eps,\\
\psi(x)&a_\eps\leq  \psi(x)\leq m_\eps,\\
m_\eps & \psi(x)> m_\eps,
\end{cases}
$$
satisfies
\begin{equation}
\label{claim-matano}
\overline{u_\eps}=\overline{\psi}=m,\quad \text{and}\quad \|u_\eps-\psi\|_{H^1(\Omega)}<\eps.
\end{equation}
Since $a_\eps$ is a nontrivial value of $\psi$, the set
$D_\eps=\{x\in \Omega: \psi(x)<a_\eps\}$ 
has positive measure. Besides, $\nabla u_\eps=0$ on $D_\eps$, implying that $\nabla A(u_\eps)=\sqrt{a(u_\eps)}\nabla u_\eps=0$ on $D_\eps$. Therefore, it is apparent that
\begin{align*}
\int_\Omega a(u_\eps)|\nabla u_\eps|^2\, \d x
&=\int_\Omega |\nabla A(u_\eps)|^2\, \d x
=\int_{\Omega\setminus D_\eps}|\nabla A(u_\eps)|^2\, \d x
\leq \int_\Omega |\nabla A(\psi)|^2\, \d x.
\end{align*}
%
Besides, in light of \eqref{210}, we have
$$\int_\Omega \Psi(u_\eps) < \int_\Omega \Psi(\psi).$$
This yields $E(u_\eps) < E(\psi)$,
a contradiction since $\eps>0$ was arbitrarily chosen. A similar contradiction arises when
$\psi>b$ on a set of nonzero measure by using \eqref{A3}.
\end{proof}


We are now ready to prove the main result of this section.
\subsection{Proof of Theorem \ref{well-pos-3D} part (3)}$\,$

\smallskip

\textbf{Step 1. Bounds for initial data close to $\psi$}.
Since $\psi$ is a local minimizer of $E$ on $\mathcal{V}_m$, we already know that $\psi\in H^2(\Omega)$ and there exist $\xi\in (0,1)$ and $\chi>0$ such that
\begin{align}
\|\psi\|_{C(\overline{\Omega})} \leq 1-\xi,\quad\text{and}
\quad E(u)\geq E(\psi)\ \ \text{for all}\ u\in \mathcal{V}_m\cap H^2(\Omega):
\|u-\psi\|_{H^2(\Omega)}<\chi.\label{statta}
\end{align}
Recalling that
$
\|u\|_{C(\overline{\Omega})}\leq C_S \|u\|_{H^2(\Omega)}
$
due to the Sobolev embedding theorem
$H^2(\Omega)\hookrightarrow C(\overline{\Omega})$,
we set
 \begin{align}
\omega_0:=\min\left\{1, \frac{\xi}{3C_S}\right\}.\label{eta1}
\end{align}
Then, for any $u\in H^2(\Omega)$ such that
\begin{equation}
\label{basee}\|u-\psi\|_{H^2(\Omega)}\leq \omega_0,
\end{equation}
it follows from \eqref{statta} that
\begin{align}
\|u\|_{C(\overline{\Omega})}
\leq \|\psi\|_{C(\overline{\Omega})}
+\|u-\psi\|_{C(\overline{\Omega})}
\leq 1-\frac{2\xi}{3}.\label{00vpes}
\end{align}
Besides, by using \eqref{eta1} and \eqref{basee}, we find 
\begin{align}
 E(u)&\leq \frac{a_M}{2}\|\nabla u\|^2+ |\Omega|\max_{s\in[-1,1]}|\Psi(s)|\nonumber\\
 &\leq \frac{a_M}{2}(\|\psi\|_{H^2(\Omega)}+1)^2 + |\Omega|\max_{s\in[-1,1]}|\Psi(s)|
 := \gamma_1,\label{EpsiE}
\end{align}
where, along the proof, the constants $\gamma_i>0$ only depend on $\psi$. 

We now consider an initial datum $\phi_0\in H^2(\Omega)$  with $\partial_\n \vp_0=0$ on $\partial \Omega$, $\overline{\vp_0}=m$ satisfying
\begin{align}
&\|\vp_0-\psi\|_{H^2(\Omega)}\leq \eta,\label{boundeta}
\end{align}
for some $\eta\in (0, \omega_0)$ to be determined later (see \eqref{Eta1a}), 
and \eqref{MM-i}. Then, we have 
\begin{align}
&\|\mu_0\|_{H^1(\Omega)}\leq M,\label{bound0}
\end{align}
for a given $M>0$, where
$
\mu_0= -\div( a(\phi_0) \nabla \phi_0) + a'(\phi_0) \frac{|\nabla \phi_0|^2}{2}+\Psi'(\phi_0)
$. In particular, \eqref{00vpes} and \eqref{EpsiE} hold true with $u= \phi_0$ and the constant $\gamma_1$ is independent of the initial datum $\phi_0$.

\medskip
\textbf{Step 2. Strong solution on a finite interval}.
Owing to \eqref{00vpes}, \eqref{EpsiE} and \eqref{bound0}, we apply Proposition
\ref{CHstrong-solution} with $M$, $\gamma_1$ and $\delta_0:=\frac23\xi$.  As a consequence, there exists a universal time $\widetilde{T}>0$, depending only on $M$, $\gamma_1$ and $\delta_0$, such that problem
\eqref{CH1}-\eqref{CH-ic} with initial datum $\phi_0$ at $t=0$ admits a (unique) strong solution
$\vp: \Omega \times [0,\widetilde{T}] \to [-1,1]$. In particular, it follows from Proposition
\ref{CHstrong-solution} that
\begin{align}
\|\vp(t)\|_{C(\overline{\Omega})}\leq 1-\frac{\xi}{3},\quad \forall\, t\in[0,\widetilde{T}],\label{Lf2}
\end{align}
and that
\begin{equation}
\label{fiH33}
\sup_{t\in [0,\widetilde{T}]}\|\phi(t)\|_{H^3(\Omega)}=: M_1,
\end{equation}
for some $M_1>0$ depending on $M$, $\gamma_1$ and $m$ but independent of $\phi_0$.
Furthermore, in light of \eqref{est-mu-H1}, there exists $\gamma_2>0$ only depending on $E(\phi_0)$ (and in turn, only on $\psi$ by \eqref{EpsiE}), such that
\begin{equation}
\label{unimu}
\|\phi(t)\|_{H^1(\Omega)}\leq \gamma_2,\qquad
\|\mu(t)\|_{H^1}\leq \gamma_2(1+\|\nabla \mu(t)\|_{L^2(\Omega)}),\qquad \forall \, t\in [0,\widetilde{T}].
\end{equation}
Furthermore, from \eqref{Lf2} and \eqref{unimu}, there exists $\gamma_3>0$ such that
\begin{align}
&E(\vp_0)-E(\vp(t)) \notag
\\
&\quad =\frac12\int_\Omega a(\vp_0)(|\nabla \vp_0|^2-|\nabla\vp(t)|^2)\,\d x+\frac12\int_\Omega (a(\vp_0)-a(\vp(t)))|\nabla\vp(t)|^2\,\d x
+\int_\Omega \Psi(\vp_0)-\Psi(\vp(t))\, \d x
\notag \\
&\quad \leq \frac{a_M}{2}\|\nabla (\vp(t)+\vp_0)\|_{L^2(\Omega)}
\|\nabla (\vp(t)-\vp_0)\|_{L^2(\Omega)}
\notag\\
&
\qquad+\max_{s\in [-1,1]}|a'(s)|\|\vp(t)-\vp_0\|_{L^\infty(\Omega)}\|\nabla\vp(t)\|_{L^2(\Omega)}^2
+ \max_{s\in[-1+\frac{\xi}{3},1-\frac{\xi}{3}]}|\Psi'(s)|
  \|\vp(t)-\vp_0\|_{L^1(\Omega)}
  \notag
  \\
& \quad \leq \gamma_3\|\vp(t)-\vp_0\|_{H^2(\Omega)}, \quad \forall\, t\in [0,\widetilde{T}].\label{edd}
\end{align}
In the rest of the proof, we will assume without loss of generality that
\begin{equation}
\label{Mgrande}
M>2\gamma_2.
\end{equation}

\medskip
\textbf{Step 3. Determining $\eta$}. Let $\epsilon>0$ be given, and define
\begin{equation}
\label{zeta}
\omega:=\min \left\{\omega_0, \ \epsilon,\ \chi,\ \beta,\
\frac{\widetilde{T}b_m}{6 \gamma_3}\right\},
\end{equation}
where $\beta$ is given by Theorem \ref{LSg-intro} and $b_m$ as in \eqref{m-ndeg}.
We claim that there exist $\eta>0$, depending on $\omega$, such that
\begin{align}
\label{vH2e}
\|\vp_0-\psi\|_{H^2(\Omega)}\leq \eta\quad \Longrightarrow\quad \|\vp(t)-\psi \|_{H^2(\Omega)} \leq \omega, \quad \forall \,  t\in [0,\widetilde{T}].
\end{align}
In particular, the solution always remains close within $\epsilon$ to $\psi$ as stated in Theorem \ref{well-pos-3D}- part (3).

\smallskip
To this aim, for  $\eta\in (0, \frac{\omega}{2}]$ we define
$$
T_\eta=\inf\{t>0:\ \|\vp(t)-\psi\|_{H^2(\Omega)}\geq \omega \}.
$$
Let $T_\eta< \widetilde{T}$. We recall that $E(\vp(t))$ is
non-increasing and that $E(\vp(t))\geq E(\psi)$ on $[0,T_\eta]$ since $\omega <\chi$. Then, we are in the position to apply Theorem \ref{LSg-intro} to derive on the interval
$[0,T_\eta]\subset [0,\widetilde{T}]$
\begin{align*}
 -\ddt \left( E(\vp(t))-E(\psi) \right)^\theta&=
 -\theta \left( E(\vp(t))-E(\psi) \right)^{\theta-1}\ddt E(\vp(t))
 \non\\
&\geq \frac{\theta}{C_L}
\frac{ \norm{ \sqrt{b(\phi(t))} \nabla \mu(t)}_{L^2(\Omega)}^2 }{\norm{\mu(t)-\overline{\mu(t)}}_{L^2(\Omega)} }
\geq 
\frac{\theta \, b_m}{C_P C_L}  \|\nabla \mu (t) \|_{L^2(\Omega)}\geq \frac{1}{C_1} \|\partial_t\vp (t)\|_{H^{-1}_{(0)}(\Omega)}.
 \end{align*}
Here, we have used \eqref{m-ndeg}, \eqref{poincare} and \eqref{est-phit-H-1}. 
The constant $C_1$ depends only on $b_m$, $\Omega$ and the parameters $\theta$ and $C_L$ from Theorem \ref{LSg-intro}.
Therefore, 
an integration on the time interval $(0,T_\eta)$, together with \eqref{edd}, gives that
 \begin{equation}
 \int_0^{T_\eta}\|\partial_t \vp (\tau)\|_{H^{-1}_{(0)}(\Omega)} \, \d \tau\leq
C_1\big(E(\vp_0)-E(\psi)\big)^\theta\leq \gamma_4 \|\vp_0-\psi\|_{H^2(\Omega)}^\theta.\nonumber
 \end{equation}
Exploiting the $H^3$-bound in \eqref{fiH33}, we then deduce the following control
 \begin{align*}
  \|\vp(T_\eta)-\psi\|_{H^2(\Omega)}
 &\leq \|\vp_0-\psi\|_{H^2(\Omega)}
 +\|\vp(T_\eta)-\vp_0\|_{H^2(\Omega)}\non\\
 &\leq \|\vp_0-\psi\|_{H^2(\Omega)}
 + C_2 \|\vp(T_\eta)-\vp_0\|_{H^3(\Omega)}^\frac34
 \|\vp(T_\eta)-\vp_0\|_{H^{-1}_{(0)}(\Omega)}^\frac14 \non\\
 &\leq \|\vp_0-\psi\|_{H^2(\Omega)}
 +  C_2 (2 M_1)^\frac34\left(\int_0^{T_\eta}\|\partial_t\vp(\tau)\|_{H^{-1}_{(0)}(\Omega)}\, \d \tau
 \right)^\frac14\non\\
 &\leq\|\vp_0-\psi\|_{H^2(\Omega)}+2^\frac34 C_2 M_1^{\frac34}\gamma_4^{\frac14}
 \|\vp_0-\psi\|_{H^2(\Omega)}^\frac{\theta}{4}\non\\
 &\leq \eta+2 C_2 M_1^{\frac34}\gamma_4^{\frac14}
 \eta^\frac{\theta}{4},
\end{align*}
for some interpolation constant $C_2$.
Hence, setting
\begin{align}
\eta:=\min\left\{\frac{\omega}{2},
\left(\frac{\omega}{8 C_2 M_1^{\frac34}\gamma_4^{\frac14}}\right)^\frac{4}{\theta}\right\},\label{Eta1a}
\end{align}
 we find $\|\vp(T_\eta)-\psi\|_{H^2(\Omega)}\leq \frac{3}{4}\omega<\omega$, which yields
 a contradiction with the definition of $T_\eta$.
 As a consequence, for $\eta$ given by \eqref{Eta1a} it holds that $T_\eta\geq \widetilde{T}$, and we learn that
$$
\|\vp(t)-\psi \|_{H^2(\Omega)} \leq \omega, \quad \forall \,  t\in [0, \widetilde{T}],$$
as claimed in \eqref{vH2e}. In turn, since $\omega\leq\frac{\xi}{3C_S}$, we obtain from \eqref{Lf2} that
 \begin{align}
 &\|\vp(t)\|_{C(\overline{\Omega})}
 \leq \|\psi\|_{C(\overline{\Omega})}
 +C_S\|\vp(t)-\psi\|_{H^2(\Omega)}
 \leq 1-\frac{2\xi}{3}, \quad 
\forall \, t\in [0,\widetilde{T}]. 
 \label{dfdf1}
 \end{align}
Furthermore, observing that
\begin{equation*}
\|\vp(t)-\vp_0\|_{H^2(\Omega)}
 \leq \|\vp(t)-\psi\|_{H^2(\Omega)}
 +\|\vp_0-\psi\|_{H^2(\Omega)}
 \leq \frac32\omega,
 \end{equation*}
by \eqref{edd} we can estimate  the energy drop on $[0,\widetilde{T}]$ as
 \begin{align}
 E(\vp_0)-E(\vp(\widetilde{T}))\leq \frac32\omega\gamma_3. \label{energydrop}
 \end{align}
By exploiting \eqref{energydrop} in the energy inequality \eqref{EE}, we find that
$$ 
\int_0^{\widetilde{T}} \mathcal{Y}(\tau) \, \d\tau\leq E(\vp_0)-E(\vp(\widetilde{T}))\leq \frac32\omega\gamma_3,
$$
where $\mathcal{Y}=\int_\Omega b(\phi)|\nabla \mu|^2 \, \d x$.
Since the function $\mathcal{Y}(\cdot)$ is non-negative, there exists
$t^*\in \left[ \frac{\widetilde{T}}{2}, \widetilde{T}\right]$ such that
 \begin{equation}
 \label{LAMA}
 \mathcal{Y}(t^*)\leq \frac{2}{\widetilde{T}}\int_{\frac{\widetilde{T}}{2}}^{\widetilde{T}} \mathcal{Y}(\tau) \, \d\tau\leq
 \frac{2}{\widetilde{T}}\cdot \frac32\omega \gamma_3 < b_m,
 \end{equation}
where the last inequality follows from the definition of $\omega$ (cf. \eqref{zeta}).
Since $b(\phi)\geq b_m$, we deduce the control
$b_m\|\nabla\mu(t^*)\|_{L^2(\Omega)}^2\leq \mathcal{Y}(t^*)$, which in light of \eqref{LAMA} gives
$$\|\nabla\mu(t^*)\|_{L^2(\Omega)}^2<1.$$
Therefore, invoking \eqref{unimu} and \eqref{Mgrande} we end up with
  \begin{align}
 \|\mu(t^*)\|_{H^1(\Omega)}\leq  \gamma_2\left( 1+\|\nabla \mu(t^*)\|_{L^2(\Omega)}\right)
 < 2\gamma_2<M.\label{H3new}
 \end{align}

 \medskip
\textbf{Step 4. Iteration argument}.
We have shown that $\mu(t^*)$ satisfies the same bound
of $\mu_0$ (compare \eqref{bound0} with \eqref{H3new}).
Moreover, by \eqref{dfdf1}, $\vp(t^*)$ complies with the separation property \eqref{00vpes}.

Now, we go back to Step 1 taking $t=t^*$ (instead of $t=0$) as a new initial time, with
$\vp(t^*)$ (and $\mu(t^*)$) as initial datum instead of $\phi_0$ and $\mu_0$. Observing that the same bounds as in Step 1 hold true, we can apply Proposition \ref{CHstrong-solution} with the same $M$ and $\delta_0$ as in Step 2. As a result, problem
\eqref{CH1}-\eqref{CH-bc} with initial datum $\phi(t^*)$ at the initial time $t=t_*$ admits a (unique) strong solution, which is
defined on the interval $[t^*, t^*+\widetilde{T}]$. Notice that $\widetilde{T}$ is exactly the same as for $\phi_0$.

Thanks to the uniqueness of strong solutions, we conclude that $\phi$ (given in Step 2) is indeed defined on the extended interval $[0,t^*+\widetilde{T}]$, which contains in particular the interval $[0,\frac32\widetilde{T}]$.

At this point, we repeat the argument in Step 3 on $[0,\frac32\widetilde{T}]$, and we derive the same refined estimates \eqref{vH2e}-\eqref{energydrop} on  $[0,\frac32\widetilde{T}]$ under exactly the same choice of $\eta$ given by \eqref{Eta1a}. Again, we will find a time $t^{**}\in [\widetilde{T},\frac32 \widetilde{T}]$ such that $\mathcal{Y}(t^{**})<b_m$ as in \eqref{LAMA} and thus $\|\mu(t^{**})\|_{H^1(\Omega)}<M$ as in \eqref{H3new}. Then, we can take $t^{**}$ as the new initial time to repeat the above procedure, thus extending the unique local strong solution to the longer interval $[0, 2\widetilde{T}]$ with the same uniform estimates on the whole $[0, 2\widetilde{T}]$. By iterating the process, we easily arrive at the conclusion of Theorem \ref{well-pos-3D} part (3).

\medskip
Lastly, the convergence as $t\to\infty$ of the global trajectory to a single stationary state can be proved as in the case $d=2$, following exactly the arguments in Section \ref{s-convergence}, where the dimension plays no role. This concludes the proof of Theorem \ref{well-pos-3D}.

\bigskip

\noindent
\textbf{Acknowledgments.}
A. Giorgini and M. Conti are supported by the MUR grant Dipartimento di Eccellenza 2023-2027 of Dipartimento di Matematica, Politecnico di Milano.
This work is partially supported by the GNAMPA (Gruppo Nazionale per l'Analisi Ma\-te\-ma\-ti\-ca, la Probabilit\`{a} e le loro Applicazioni) of INdAM (Istituto Nazionale di Alta Matematica) project CUP E5324001950001. 

%
%
%

\end{document}